\documentclass[12pt,reqno]{amsart}  
\usepackage{amsmath,amsthm,amssymb,amscd,faktor,float,mathdots,mathtools,stmaryrd,enumerate,shuffle,hyperref,xcolor,ytableau,tabularx,tikz}

\RequirePackage{babel} 

\RequirePackage{scrbase} 

\RequirePackage{scrhack} 

\RequirePackage{setspace} 

\RequirePackage{longtable} 

\RequirePackage{siunitx} 

\RequirePackage{graphicx} 
\graphicspath{{Figures/}{./}} 
\usepackage{float}

\usepackage{fancyvrb}

\usepackage{listings}

\usepackage{pdfpages}

\usepackage{ragged2e}

\usepackage{BOONDOX-uprscr}

\definecolor{codegreen}{rgb}{0,0.6,0}
\definecolor{codegray}{rgb}{0.5,0.5,0.5}
\definecolor{codepurple}{rgb}{0.58,0,0.82}
\definecolor{backcolour}{rgb}{0.95,0.95,0.92}
\definecolor{lightblue}{rgb}{0,0,0.65}

\lstdefinestyle{mystyle}{
    backgroundcolor=\color{backcolour},   
    commentstyle=\color{lightblue},
    keywordstyle=\color{codegreen},
    numberstyle=\tiny\color{codegray},
    stringstyle=\color{codepurple},
    basicstyle=\ttfamily\footnotesize,
    breakatwhitespace=false,         
    breaklines=true,                 
    captionpos=b,                    
    keepspaces=true,                 
    numbers=left,                    
    numbersep=5pt,                  
    showspaces=false,                
    showstringspaces=false,
    showtabs=false,                  
    tabsize=2
}

\RequirePackage{booktabs} 

\RequirePackage{caption} 
\mathtoolsset{showonlyrefs}

\hypersetup{pdfborder=0 0 0}
\hypersetup{colorlinks=true}
\hypersetup{citecolor=magenta}
\hypersetup{linkcolor=blue}
\hypersetup{urlcolor=cyan}
\usetikzlibrary{arrows,positioning}
\usetikzlibrary{calc}
\usetikzlibrary{decorations}
\usetikzlibrary{decorations.pathreplacing}
\tikzset{
    >=stealth',
    punkt/.style={
           rectangle,
           rounded corners,
           draw=black, very thick,
           text width=6.5em,
           minimum height=2em,
           text centered},
    pil/.style={
           ->,
           thick,
           shorten <=2pt,
           shorten >=2pt,}
}
\usepackage{wrapfig}
\usepackage{pgfplots}
\allowdisplaybreaks

\address{Department of Mathematics, University of Hamburg, Bundesstrasse 55 \\ 20146 Hamburg, Germany.}
\email{benjamin.brindle@uni-hamburg.de}

\subjclass[2020]{05A30, 11M32, 13J25.}
\keywords{Multiple zeta values.}

\newtheorem{theorem}{Theorem}
\newtheorem{proposition}[theorem]{Proposition}

\newtheorem{lemma}[theorem]{Lemma}
\newtheorem{corollary}[theorem]{Corollary}
\newtheorem{conjecture}[theorem]{Conjecture}
\numberwithin{equation}{theorem}
\theoremstyle{definition}
\newtheorem{definition}[theorem]{Definition}
\newtheorem{example}[theorem]{Example}
\newtheorem{remark}[theorem]{Remark}

\newtheorem{function}[theorem]{Function}
\newtheorem{notation}[theorem]{Notation}

\newcommand{\eqlabel}[1]{\label{#1}}
\newcommand{\noproof}[1]{}

\newcommand{\Q}{\mathbb{Q}}
\newcommand{\Z}{\mathbb{Z}}

\newcommand{\wt}{\operatorname{wt}}
\newcommand{\dep}{\operatorname{depth}}

\newcommand{\len}{\operatorname{len}}

\newcommand{\word}{\mathtt{W}}
\newcommand{\T}{T}
\newcommand{\fz}[1]{\zeta_q^{\mathrm{f}}\left(#1\right)}
\newcommand{\indexset}[2]{\mathcal{J}_{#1,#2}}
\newcommand{\sdim}[2]{\mathscr{s}_{#1,#2}}
\newcommand{\tdim}[2]{\mathscr{t}_{#1,#2}}
\newcommand{\idim}[2]{\mathscr{j}_{#1,#2}}
\newcommand{\Boxmap}[2]{\operatorname{Box}_{#1,#2}}
\newcommand{\kernelbox}[2]{\mathcal{K}_{#1,#2}}

\newcommand{\bone}{{\mathbf{1}}}
\newcommand{\plainbasic}[2]{\mathfrak{M}_{#1,#2}}
\newcommand{\basis}[2]{\mathfrak{B}_{#1,#2}}

\newcommand{\sz}{\zeta_q^\mathrm{SZ}} 

\newcommand{\Zq}{\mathcal{Z}_q^f}
\newcommand{\Zqz}{\mathcal{Z}_q^{f,\circ}}

\newcommand{\reverse}[1]{\operatorname{rev}(#1)}

\newcommand{\qmzv}{$q$MZV}

\newcommand{\CP}{\mathcal{P}}

\newcommand{\dd}{d}
\newcommand{\dr}{r}

\newcommand{\VSbox}[2]{\mathscr{S}_{#1,#2}}
\newcommand{\Vbox}[2]{\mathscr{T}_{#1,#2}}
\newcommand{\spanQ}[1]{\operatorname{span}_\Q\left\{#1\right\}}

\newcommand{\uu}{u}

\newcommand{\one}{\boldsymbol{1}}

\newcommand{\df}{:=}

\newcommand{\QB}{\Q\langle \mathcal{U} \rangle}

\newcommand{\fil}[2]{\operatorname{Fil}^{\mathrm{#1}}_{#2}}
\newcommand{\zero}{\operatorname{zero}}

\newcommand{\F}[1]{\operatorname{F}_{#1}}

\definecolor{dutchorange}{RGB}{220,50,0}
\definecolor{darkblue}{RGB}{0,71,171}

\definecolor{mygray}{RGB}{160,160,160}
\definecolor{mycyan}{RGB}{150,255,255}
\definecolor{darkred}{RGB}{130,0,0}

\tikzset{
    >=stealth',
    punkt/.style={
           rectangle,
           rounded corners,
           draw=black, very thick,
           text width=6.5em,
           minimum height=2em,
           text centered},
    pil/.style={
           ->,
           thick,
           shorten <=2pt,
           shorten >=2pt,}
}

\makeatletter
\newcommand{\gettikzxy}[3]{%
  \tikz@scan@one@point\pgfutil@firstofone#1\relax
  \edef#2{\the\pgf@x}%
  \edef#3{\the\pgf@y}%
}
\makeatother

\title{On the structure of Multiple q-Zeta Values}
\author{Benjamin Brindle}
\date{\today}

\begin{document}
\maketitle
\begin{abstract}
In 2015, Bachmann \cite{Ba3} conjectured that the~$\Q$-vector space~$\Zq$ of (formal)~$q$-analogues of Multiple Zeta Values (\qmzv s) is spanned by a very particular set compared to known spanning sets. In this work, we prove that this conjecture is true for a subspace of~$\Zq$ spanned by words satisfying some condition on their number of zeros and depth. According to this partial result, we give an explicit approach to the whole conjecture, based on particular~$\Q$-linear relations among formal Multiple~$q$-Zeta Values which are implied by duality.
\end{abstract}

\section{Introduction}

Given a field~$F$ and a countable set~$\mathcal{A}$, we call~$\mathcal{A}$ also an \emph{alphabet} and elements of~$\mathcal{A}$ are referred to as \emph{letters}. Denote by~$\operatorname{span}_F \mathcal{A}$ the~$F$-vector space spanned by elements of~$\mathcal{A}$. Furthermore, monomials of elements in~$\mathcal{A}$ (with respect to concatenation) are called \emph{words}. Usually, the neutral element with respect to concatenation is denoted by~$\bone$ and called the \emph{empty word}. Let~$\mathcal{A}^\ast$ denote the set of words with letters in~$\mathcal{A}$, then we write~$F\langle\mathcal{A}\rangle$ for the~$F$-vector space~$\operatorname{span}_F\mathcal{A}^\ast$, equipped with the non-commutative, but associative multiplication, given by concatenation. 

Choosing~$F = \Q$ and~$\mathcal{A} =\mathcal{U} := \{u_j\mid j\in\Z_{\geq 0}\}$, we define the \emph{stuffle product} to be the~$\Q$-bilinear map~$\ast\colon\QB\times \QB\to \QB$ recursively via
\begin{align*}
    u_{j_1} \word_1 \ast u_{j_2} \word_2 =  u_{j_1} ( \word_1 \ast u_{j_2} \word_2)+ u_{j_2} (u_{j_1} \word_1 \ast \word_2) + u_{j_1+j_2}(\word_1 \ast \word_2)
\end{align*}
for all~$j_1,j_2\in\Z_{\geq 0}$ and~$\word_1,\word_2 \in \mathcal{U}^\ast$ with initial condition~$\one \ast \word = \word \ast \one = \word$ for any~$\word \in \mathcal{U}^\ast$. By~Hoffman (\cite{Ho}),~$(\QB,\ast)$ is an associative and commutative~$\Q$-algebra. For a word~$\word = u_{k_1}\cdots u_{k_\dr}\in\mathcal{U}^{\ast}$, we often write~$u_{\mathbf{k}}$ ($u_\emptyset :=\bone$), where~$\mathbf{k} = (k_1,\dots,k_\dr)$, and associate the notion of
\begin{align*}
    \textit{length},&\quad \len(\word) := \len(\mathbf{k}):=\dr,\\
    \textit{depth},&\quad \dep(\word) := \dep(\mathbf{k}):=\#\{k_j\neq 0\mid 1\leq j\leq \dr\},\\
    \textit{number of zeros},&\quad \zero(\word) := \zero(\mathbf{k}):=\#\{k_j = 0\mid 1\leq j\leq \dr\},\\
    \textit{weight},&\quad \wt(\word) := \wt(\mathbf{k}):=|\mathbf{k}| + \zero(\word),
\end{align*}
where~$|\mathbf{k}|:=k_1+\cdots+k_\dr$. Furthermore, we denote~$\mathcal{U}^{\ast,\circ} := \mathcal{U}^\ast\backslash u_0\mathcal{U}^\ast$ to be the set of words not starting with~$u_0$ and we define the corresponding~$\Q$-vector space~$\QB^\circ\subset\QB$ spanned by the words from~$\mathcal{U}^{\ast,\circ}$. Note that~$\QB^\circ$ is closed under~$\ast$ which gives rise to a commutative~$\Q$-algebra~$(\QB^\circ,\ast)$ (see~\cite{Ho}). The map~$\sz\colon (\QB^\circ,\ast)\to (\Q\llbracket q\rrbracket,\cdot)$ is the~$\Q$-algebra homomorphism (see~\cite{HI}) defined via~$\sz(\bone) = 1$,~$\Q$-linearity, and, with~$m_{\dd+1} :=0$,
\begin{align}
\eqlabel{eq:SZ-defApp}
    \sz\left(u_{k_1}u_0^{z_1}\cdots u_{k_\dd}u_0^{z_\dd}\right) := \sum\limits_{m_1>\cdots>m_\dd>0} \prod\limits_{j=1}^\dd \binom{m_j-m_{j+1}-1}{z_j} \frac{q^{m_j k_j}}{(1-q^{m_j})^{k_j}},
\end{align}
for any~$k_1,\dots,k_\dd\in\Z_{>0}$ and~$z_1,\dots,z_\dd\in\Z_{\geq 0}$ where~$\dd\in\Z_{>0}$ (note that this definition is not the usual one, like in~\cite{Sin15}, but equivalent to it; this statement can be deduced, e.g., from~\cite[Theorem 2.18]{Br2}). We denote by~$\mathcal{Z}_q$ the image of~$\sz$ and call elements in~$\mathcal{Z}_q$ \emph{(Schlesinger--Zudilin) \qmzv s} ((SZ-)\qmzv s for short). Note that these~$q$-series are~$q$-analogues of Multiple Zeta Values since in the case~$k_1\geq 2$ and~$z_1 = \dots = z_\dd = 0$, we have
\begin{align*}
    \lim\limits_{q\to 1} (1-q)^{k_1+\cdots+k_\dd}\sz(u_{k_1}\cdots u_{k_\dd}) = \zeta(u_{k_1}\cdots u_{k_\dd}) := \sum\limits_{m_1>\cdots >m_\dd>0} \frac{1}{m_1^{k_1}\cdots m_\dd^{k_\dd}}.
\end{align*}
But in this work, we focus purely on the algebraic structure of (SZ-)\qmzv s and do not consider its implication for classical Multiple Zeta Values. Over the years, several versions of \qmzv s were introduced (see, e.g., \cite{Sch,OOZ,Zu,Ba1,BK1}); for an overview, see \cite{Br2}. Because of Conjecture~\ref{conj:allrelsqversionApp} and since the~$q$-series on the right of~\eqref{eq:SZ-defApp} is invariant under the~$\Q$-linear involution~$\tau: \QB^\circ \rightarrow \QB^\circ$, defined by~$\tau(\bone) :=\bone$ and
\[\tau\left(u_{k_1} u_0^{z_1} \cdots u_{k_\dd} u_0^{z_\dd}\right) \df u_{z_\dd+1} u_0^{k_\dd-1} \cdots u_{z_1+1} u_0^{k_1-1}\]
for all~$d\in\Z_{>0}$, $k_1,\ldots,k_\dd \geq 1, \text{ and } z_1,\dots,z_\dd\geq 0$ (see \cite[Theorem 8.3]{Zh}), we will consider
the algebra of \emph{formal \qmzv s},
\begin{align*}
    \Zq \df \faktor{\left(\QB^\circ,\ast\right)}{T},
\end{align*} 
where~$\T$ is the~$\ast$-ideal in~$\QB^\circ$ generated by~$\left\{\tau(\word) - \word\, |\, \word\in \QB^\circ\right\}$. For~$\word\in \Q\langle\mathcal{U}\rangle^\circ$, we set~$\fz{\word}$ to be the congruence class of~$\word$ in~$\Zq$. Note that depth and weight are invariant under~$\tau$ while the number of zeros generally is not. Furthermore, playing with~$\tau$ and the stuffle product~$\ast$, one obtains non-trivial~$\Q$-linear relations among formal \qmzv s. The following folklore conjecture (see~\cite{BaTalk}; a published version can be found in {\cite[Conjecture 1]{Zu}}) states the expectation of how the~$\Q$-linear relations among SZ-\qmzv s look like.
\begin{conjecture}[Bachmann]
    \label{conj:allrelsqversionApp}
    All~$\Q$-linear relations among elements in~$\mathcal{Z}_q$ are obtained by the stuffle product~$\ast$ and duality~$\tau$.
\end{conjecture}\noproof{conjecture}
I.e., one expects~$\mathcal{Z}_q\simeq\Zq$. We will consider in this paper only~$\Q$-linear relations in~$\Zq$ which are implied by
\begin{align}
\eqlabel{eq:relshapeApp}
        \fz{\word_1\ast (\word_2-\tau(\word_2))} = 0
\end{align}
for any words~$\word_1,\word_2\in\mathcal{U}^{\ast,\circ}$. For investigating~$\Zq$ in more detail, we need the following notion of filtrations.

\begin{notation}
    \begin{enumerate}
    \item For every~$(\mathrm{N},\operatorname{op})\in\{(\mathrm{Z},\zero),(\mathrm{D},\dep),(\mathrm{W},\wt)\}$, $n\in\Z$, and sets~$\mathcal{S}\subset\QB^\circ,\, \mathcal{S}'\subset \Zq$, write
    \begin{align*}
        \fil{N}{n}\mathcal{S} &:= \spanQ{\word\in\mathcal{U}^{\ast,\circ}\mid \operatorname{op}(\word)\leq n}\cap\mathcal{S},
        \\
        \fil{N}{n}\mathcal{S}' &:= \spanQ{\fz{\word} \in \Zq \mid \word \in \mathcal{U}^{\ast,\circ},\, \operatorname{op}(\word) \leq n}\cap\mathcal{S}'
    \end{align*}
    for the filtration by number of zeros~(if~$\mathrm{N}=\mathrm{Z}$), depth~(if~$\mathrm{N}=\mathrm{D}$), and weight~(if~$\mathrm{N}=\mathrm{W}$) respectively on~$\mathcal{S}$ and~$\mathcal{S}'$ respectively. 
        \item For~$\mathcal{S}\subset\QB^\circ$ or~$\mathcal{S}\subset\Zq$, $\mathrm{N}_1,\dots,\mathrm{N}_m\in\{\mathrm{Z,D,W}\}$, where~$m\in\Z_{>0}$, and for integers~$n_1,\dots,n_m\in\Z$, we abbreviate
        \begin{align}
            \operatorname{Fil}^{\mathrm{N}_1,\dots,\mathrm{N}_m}_{n_1,\dots,n_m}\mathcal{S} := \bigcap\limits_{j=1}^m \operatorname{Fil}^{\mathrm{N}_j}_{n_j}\mathcal{S}.
        \end{align}
    \end{enumerate}
\end{notation}\noproof{notation}
The following particular filtration will play a main role in this paper.
\begin{definition}
    We define
        \begin{align*}
            \Zqz :=\fil{Z}{0}\Zq.
        \end{align*}
\end{definition}\noproof{definition}

At this point, note that
\begin{align}
\eqlabel{eq:basicstuffle1}
       \fil{Z,D,W}{z,\dd,w}\QB^\circ \ast \fil{Z,D,W}{z',\dd',w'}\QB^\circ\subset\fil{Z,D,W}{z+z',\dd+\dd',w+w'}\QB^\circ
\end{align}
and 
\begin{align}
\eqlabel{eq:basicstuffle2}
    \tau\left(\fil{Z,D,W}{z,\dd,w}\QB^\circ\right) = \fil{Z,D,W}{w-z-\dd,\dd,w}\QB^\circ
\end{align}
for all~$z,z',\dd,\dd',w,w'\in\Z$. Hence, considering~\eqref{eq:relshapeApp},~$\word_1\ast\word_2$ and~$\word_1\ast\tau(\word_2)$ are, in general, in different filtrations of~$\QB^\circ$ regarding the number of zeros since we have, in general~$z\neq w-z-d$. Therefore, for given~$\word\in\mathcal{U}^{\ast,\circ}$, it is difficult to find the minimal~$z\in\Z_{\geq 0}$ such that~$\fz{\word}\in\fil{Z}{z}\Zq$.

Let us consider a small example of how we use~$\Q$-linear relations of shape~\eqref{eq:relshapeApp} to obtain that, e.g.,~$\fz{\word}\in\Zqz$ for~$\word = u_2u_0\in\mathcal{U}^{\ast,\circ}$. First, we note that
\begin{align*}
    u_2u_0 = u_1\ast u_1u_0 - 2u_1u_1u_0 - u_1u_1 - u_1u_0u_1.
\end{align*}
Now, 
\begin{align*}
    0 = &\, \fz{u_1\ast (u_1u_0 - \tau(u_1u_0))} -2\fz{\bone\ast (u_1u_1u_0 - \tau(u_1u_1u_0))} 
    \\ &\, - \fz{\bone \ast (u_1u_0u_1 - \tau(u_1u_0u_1))}
    \\
    = &\,  \fz{u_1\ast u_1u_0} - \fz{u_1\ast u_2} -2\fz{u_1 u_1 u_0} + 2 \fz{u_2u_1} 
    \\ &\, - \fz{u_1u_0u_1} + \fz{u_1 u_2},
\end{align*}
and so,
\begin{align}
\eqlabel{eq:z=1firstcalc}
    \fz{u_2u_0} =\, & \fz{u_1\ast u_2} - 2 \fz{u_2u_1} -\fz{u_1 u_1} - \fz{u_1 u_2}
    \\
    =\, & \fz{u_1 u_2} + \fz{u_3} - \fz{u_2u_1} -\fz{u_1 u_1} - \fz{u_1 u_2}\in\Zqz.
\end{align}
That formal~\qmzv s are in~$\Zqz$ already is not just a coincidence, as the following conjecture shows.

\begin{conjecture}[Bachmann,~{\cite[Conjecture 3.9]{Ba1}}]
\label{conj:mdbdApp}
    For all~$z,\dd,w\in\Z_{>0}$, we have
    \begin{align}
    \eqlabel{eq:conjmdbd}
        \fil{Z,D,W}{z,\dd,w}\Zq\subset \fil{D,W}{z+\dd,w}\Zqz.
    \end{align}
    In particular, we have~$\Zq = \Zqz$.
\end{conjecture}\noproof{conjecture}
We say that Bachmann's Conjecture~\ref{conj:mdbdApp} \emph{is true for}~$(z_0,\dd_0,w_0)\in\Z_{>0}^3$ if~\eqref{eq:conjmdbd} is true for~$(z,\dd,w) = (z_0,\dd_0,w_0)$.

Partial results already exist; we will collect them in the following.
\begin{theorem}
\label{thm:mdbdknownbibr}
\begin{enumerate}
    \item By Bachmann (\cite[Proposition 4.4]{Ba1}), Bachmann's Conjecture~\ref{conj:mdbdApp} is true for all~$(z,1,w)\in\Z_{>0}^3$.
    \item also by Bachmann (\cite[Proposition 5.9]{Ba1}), Bachmann's Conjecture~\ref{conj:mdbdApp} is true for all~$(1,2,w)\in\Z_{>0}^3$.
    \item by Vleeshouwers (\cite[Theorem 5.3]{V}), Bachmann's Conjecture~\ref{conj:mdbdApp} is true for all~$(z,2,w)\in\Z_{>0}^3$ with some parity condition on~$w$,
    \item and by Burmester (\cite[Theorem 6.4]{Bu1}), Bachmann's Conjecture~\ref{conj:mdbdApp} is true for all~$(1,\dd,w)\in\Z_{>0}^3$.
\end{enumerate}
\end{theorem}\noproof{theorem}

While the proofs of (i)--(iii) are mainly based on generating series of the corresponding~$q$-series, the proof of (iv) uses the stuffle product and duality relations. Using relations among formal Multiple Zeta Values of shape~\eqref{eq:relshapeApp} only suffices to prove the following theorem.

\begin{theorem}[Theorem~\ref{thm:sumtoz2App}]
\label{thm:sumtozApp}
    Let be~$z,\dd\in\Z_{>0}$,~$\mathbf{k} = (k_1,\dots,k_\dd)\in\Z_{>0}^\dd$, and consider integers~$1\leq j_1 \leq j_2\leq \dd$. Deconcatenate~$\mathbf{k}$ as
    \begin{align*}
        \mathbf{k}_{(1;j_1)} = (k_1,\dots,k_{j_1}),\ \mathbf{k}_{(j_1+1;j_2)} = (k_{j_1+1},\dots,k_{j_2}),\ \mathbf{k}_{(j_2+1;\dd)} = (k_{j_2+1},\dots,k_\dd).
    \end{align*} We have
    \begin{align}
        \fz{u_{\mathbf{k}_{(1;j_1)}}\left(u_{\mathbf{k}_{(j_1+1;j_2)}}\ast u_{\mathbf{k}_{(j_2+1;\dd)}}u_0^z\right)}\in\sum\limits_{s=1}^z \fil{Z,D,W}{z-s,\dd+s,w}\Zq,
    \end{align}
    where $w= |\mathbf{k}| + z$. 
\end{theorem}\noproof{theorem}
\begin{remark}

    \begin{enumerate}
        \item Theorem~\ref{thm:sumtozApp} is a generalization of Bachmann's Theorem~\cite[Proposition 4.4]{Ba1} via the case~$\dd=1$. We have already seen the proof for an example of this theorem using our methods in~\eqref{eq:z=1firstcalc}. We will generalize this approach in Proposition~\ref{prop:dep1explicit} to generalize Bachmann's Theorem~\ref{thm:mdbdknownbibr}(i).
        \item Note that Theorem~\ref{thm:sumtozApp} also generalizes Burmester's Theorem~\cite[Theorem~6.4]{Bu1} via considering the special cases~$z=1$. For details, we refer to Corollary~\ref{cor:z=1}.
    \end{enumerate}
\end{remark}\noproof{remark}
Extending our methods of playing with relations of shape~\eqref{eq:relshapeApp}, we observe the following theorem.
\begin{theorem}[Theorem~\ref{thm:main2App}]
\label{thm:mainApp}
    Bachmann's Conjecture~\ref{conj:mdbdApp} is true for all~$(z,\dd,w)\in\Z_{>0}^3$ with~$z+\dd\leq 6$.
\end{theorem}\noproof{theorem}

In this paper, we will use duality and the stuffle product only for an approach to write~$\fz{\word}$ for every~$\word\in\mathcal{U}^{\ast,\circ}$ satisfying~$\zero(\word)\geq 1$ as linear combination of~$\fz{\word'}$'s with~$\zero(\word') < \zero(\word)$ and~$\word'\in\mathcal{U}^{\ast,\circ}$. We need the following notion of~$\F{z,\dd,w}$ for this.

\begin{definition}
    For~$z,\dd,w\in\Z_{>0}$, we define
    \begin{align*}
        \F{z,d,w} := \fil{Z,D,W}{z,d,w-1}\Zq + \sum\limits_{\substack{z'+d' = z+d-1\\ 0\leq z'\leq z}} \fil{Z,D,W}{z',d',w}\Zq.
    \end{align*}
\end{definition}\noproof{definition}

In this paper, our main approach towards Bachmann's Conjecture~\ref{conj:mdbdApp} is to strengthen the conjecture as follows and then to investigate the strengthened version for obtaining results like Theorem~\ref{thm:mainApp}.
\begin{conjecture}[Refined Bachmann Conjecture]
    \label{conj:mdbdstrongApp}
    For all~$z,d,w\in\Z_{>0}$, we have 
    \begin{align}
    \eqlabel{eq:refined}
        \fil{Z,D,W}{z,d,w}\Zq\subset\F{z,d,w}.
    \end{align}
\end{conjecture}\noproof{conjecture}
We say that the refined Bachmann Conjecture~\ref{conj:mdbdstrongApp} \emph{is true for}~$(z_0,\dd_0,w_0)\in\Z_{>0}^3$, if~\eqref{eq:refined} is true for~$(z,\dd,w) = (z_0,\dd_0,w_0)$.
\begin{lemma}[Lemma~\ref{lem:genideaApp}]
\label{lem:genidea2App}
    Fix~$z,\dd,w\in\Z_{>0}$. If the refined Bachmann Conjecture~\ref{conj:mdbdstrongApp} is true for~$(z,\dd,w)$ and if Bachmann's Conjecture~\ref{conj:mdbdApp} is true for all~$(z',\dd',w')\in\Z_{>0}^3$ with~$z'+\dd'+w' < z+\dd+w$, then Bachmann's Conjecture~\ref{conj:mdbdApp} is true for~$(z,\dd,w)$. In particular, the refined Bachmann Conjecture~\ref{conj:mdbdstrongApp} implies Bachmann's Conjecture~\ref{conj:mdbdApp}.
\end{lemma}\noproof{lemma}
To study the refined Bachmann Conjecture~\ref{conj:mdbdstrongApp}, we will introduce the box product (see Definition~\ref{def:boxApp}) that provides a connection to the stuffle product (see Lemma~\ref{lem:mainideaboxApp}) and allows us to refine the refined Bachmann Conjecture~\ref{conj:mdbdstrongApp} for~$z\geq\dd$ again (see Conjecture~\ref{conj:systemApp}). In this way, we obtain another particular result towards the refined Bachmann Conjecture~\ref{conj:mdbdstrongApp}.

\begin{theorem}[Theorem~\ref{thm:rleq42App}]
    \label{thm:rleq4App}
    The refined Bachmann Conjecture~\ref{conj:mdbdstrongApp} is true for all triples of positive integers~$(z,\dd,w)\in\Z^3_{>0}$ with~$1\leq \dd\leq 4$.
\end{theorem}\noproof{theorem}
Theorem~\ref{thm:rleq4App} will follow mainly using Theorem~\ref{thm:mainApp} and the investigation of the box product from Section~\ref{sec:boxprod}. Furthermore, Theorem~\ref{thm:rleq4App} is a strong statement since - together with some more results of this paper -  now, Bachmann's Conjecture~\ref{conj:mdbdApp} is almost proven for~$z+\dd\leq 7$ as well: Namely, following Lemma~\ref{lem:genidea2App}, it remains to prove the refined Bachmann Conjecture~\ref{conj:mdbdstrongApp} for triples of shape~$(2,5,w)\in\Z_{>0}^3$. 

All our main results (and those implied by the box product) are based on~$\Q$-linear relations of shape~\eqref{eq:relshapeApp} only. Following our approach to a general proof of the refined Bachmann Conjecture~\ref{conj:mdbdstrongApp} (and so of Bachmann's Conjecture~\ref{conj:mdbdApp} too), described in Section~\ref{sec:idea}, it is conjecturally possible to prove the refined Bachmann Conjecture~\ref{conj:mdbdstrongApp} using~$\Q$-linear relations of shape~\eqref{eq:relshapeApp} only. Based on our results, it seems that this approach works. Furthermore, our explicit approach has the advantage that it is (compared to other approaches) easy to obtain explicit formulas for~$\fz{\word}$ (with~$\word\in\mathcal{U}^{\ast,\circ}$) as element of~$\Zqz$. Proposition~\ref{prop:dep1explicit}, for example, contains such an explicit formula. Nevertheless, the explicitness limits this method in the sense that the larger~$z+\dd$ is in the refined Bachmann Conjecture~\ref{conj:mdbdstrongApp}, the more confusing the~$\Q$-linear relations~\eqref{eq:relshapeApp}, one needs to consider following our approach, become.

\medskip

\paragraph{\textbf{Organization of the paper.}}  Section~\ref{sec:mdbd:boxproductintro} contains the introduction of the box product mentioned. Section~\ref{sec:annika} contains generalizations of theorems by Bachmann and Burmester concerning the refined Bachmann Conjecture~\ref{conj:mdbdstrongApp}, like Theorem~\ref{thm:sumtozApp}. In Section~\ref{sec:boxprod}, we will investigate the \emph{box product} and consider its connection to the stuffle product. Furthermore, Section~\ref{sec:idea} contains the rough description of our approach to the refined Bachmann Conjecture~\ref{conj:mdbdstrongApp}.  Using the approach from Section~\ref{sec:idea}, in Section~\ref{sec:zlessr}, we prove new partial results towards Bachmann's Conjecture~\ref{conj:mdbdApp}. Particularly, there, we will provide proofs for Theorems~\ref{thm:mainApp} and~\ref{thm:rleq4App}. Last, Section~\ref{sec:mdbd:outlook} ends the paper with some open questions and a rough generalization of our calculations from Section~\ref{sec:zlessr}.

\medskip

\paragraph{\textbf{Acknowledgements.}} The author thanks Henrik Bachmann, Annika Burmester, Jan-Willem van Ittersum, and Ulf Kühn for valuable discussions and helpful comments on this paper.

\section{Introduction of the box product}
\label{sec:mdbd:boxproductintro}
In this section, we introduce the box product and consider elementary properties. First, we briefly remark on a property of the stuffle product in the following proposition.

\begin{proposition}
\label{prop:stufflemaxzero}
    Let be~$\word_1,\word_2\in\mathcal{U}^{\ast,\circ}$ and write
    \begin{align*}
        z = \zero(\tau(\word_1)) + \zero(\tau(\word_2)),\ \dd_1 = \dep(\word_1),\ \dd_2 = \dep(\word_2),\ w=\wt(\word_1) + \wt(\word_2).
    \end{align*}
    Then, for~$0\leq s\leq\min\{\dd_1,\dd_2\}$, there are uniquely determined
    \begin{align*}
        \mathcal{L}_{\max\{\dd_1,\dd_2\}+s}\in\spanQ{\word\in\mathcal{U}^{\ast,\circ}\, |\, \dep(\word) = \max\{\dd_1,\dd_2\}+s}
    \end{align*}
    such that
    \begin{align*}
        \word_1\ast\word_2 = \sum\limits_{s=0}^{\min\{\dd_1,\dd_2\}} \mathcal{L}_{\max\{\dd_1,\dd_2\}+s}.
    \end{align*}
    Furthermore, for all~$0\leq s\leq\min\{\dd_1,\dd_2\}$, we have
    \begin{align*}
        \tau\left(\mathcal{L}_{\max\{\dd_1,\dd_2\}+s}\right)\in\fil{Z,D,W}{z-s,\max\{\dd_1,\dd_2\}+s,w}\QB^\circ.
    \end{align*}
    In particular,~$\tau\left(\mathcal{L}_{\max\{\dd_1,\dd_2\}}\right)$ is the part of~$\tau(\word_1\ast\word_2)$ having the maximum number of zeros and we have
    \begin{align*}
        \tau(\word_1\ast\word_2) \in \sum\limits_{s=0}^{\min\{\dd_1,\dd_2\}} \fil{Z,D,W}{z-s,\max\{\dd_1,\dd_2\}+s,w}\QB^\circ.
    \end{align*}
\end{proposition}

\begin{proof}
    This is a direct consequence of Equations~\eqref{eq:basicstuffle1} and~\eqref{eq:basicstuffle2}.
\end{proof}
Let us consider an example to point out the statement of Proposition~\ref{prop:stufflemaxzero}.

\begin{example}
\label{ex:boxstuffle1}
    Choose~$\word_1 = u_2,\, \word_2 = u_1 u_2$, i.e.,~$\dd_1 =1,\, \dd_2 = 2$ in the notion of Proposition~\ref{prop:stufflemaxzero}. We have
    \begin{align*}
        \word_1\ast\word_2 =&\, \underbrace{\textcolor{dutchorange}{u_3 u_2 + u_1 u_4}}_{=\,\mathcal{L}_2} + \underbrace{u_2 u_1 u_2 + 2 u_1 u_2 u_2}_{=\,\mathcal{L}_{3}}.
    \end{align*}
    Observe
    \begin{align*}
        \tau(\mathcal{L}_2) =\, \textcolor{dutchorange}{u_1 u_0 u_1 u_0 u_0 + u_1 u_0 u_0 u_0 u_1},\quad
         \tau(\mathcal{L}_{3}) =\, u_1 u_0 u_1 u_1 u_0 + 2u_1 u_0 u_1 u_0 u_1.
    \end{align*}
    We see that~$\tau(\mathcal{L}_2)$ indeed has the maximum number of zeros in the expression~$\tau(u_2\ast u_1 u_2)$.
\end{example}\noproof{example}
Since we want to reduce the number of zeros, we often will be interested in the part of the stuffle product only that has the maximum number of zeros. Therefore, Proposition~\ref{prop:stufflemaxzero} motivates the definition of the \emph{box product} that basically extracts this part of the stuffle product after one applies~$\tau$. 

\begin{definition}[Box product]
\label{def:boxApp}
    The~$\Q$-bilinear \emph{box product}~$\boxast\colon \QB^\circ\times\QB^\circ\to\QB^\circ$ is defined as follows: For~$\word_j\in\mathcal{U}^{\ast,\circ}$ with~$\dep(\word_j) = \dd_j$, where~$j\in\{1,2\}$, we set
    \begin{align*}
        \word_1\boxast\word_2 := \mathcal{L}_{\max\{d_1,d_2\}}
    \end{align*}
    in the notion of Proposition~\ref{prop:stufflemaxzero}.
\end{definition}\noproof{definition}
 For illustration, we continue Example~\ref{ex:boxstuffle1}. 

\begin{example}
\label{ex:boxstuffle}
    We have
    \begin{align*}
        \uu_2\boxast \uu_1 \uu_2 = \uu_3 \uu_2 + \uu_1 \uu_4,
    \end{align*}
    which is exactly~$\mathcal{L}_2$ of Example~\ref{ex:boxstuffle1}, i.e., after applying~$\tau$, one obtains the part of the stuffle product~$u_2\ast u_1 u_2$ having maximum number of zeros. We state and prove the generalization of this observation in Lemma~\ref{lem:mainideaboxApp}.
\end{example}\noproof{example}

\begin{corollary}
\label{cor:stufflemaxzero}
    Let be~$\word_1,\word_2\in\mathcal{U}^{\ast,\circ}$ and write
    \begin{align*}
        z = \zero(\tau(\word_1)) + \zero(\tau(\word_2)),\ \dd_1 = \dep(\word_1),\ \dd_2 = \dep(\word_2),\ w=\wt(\word_1) + \wt(\word_2).
    \end{align*}
    Then,
    \begin{align*}
        \tau(\word_1\ast\word_2) - \tau(\word_1\boxast\word_2)\in \sum\limits_{s=1}^{\min\{\dd_1,\dd_2\}} \fil{Z,D,W}{z-s,\max\{\dd_1,\dd_2\}+s,w}\QB^\circ.
    \end{align*}
\end{corollary}
\begin{proof}
    This is an immediate consequence of Proposition~\ref{prop:stufflemaxzero} and the definition of the box product.
\end{proof}

\begin{lemma}
\label{lem:boxproduct}
    Consider the alphabet~$\mathcal{U}\backslash\{u_0\} = \{\uu_j\mid j\in\Z_{> 0}\}$. The restriction of the box product~$\boxast\colon \Q\langle\mathcal{U}\backslash\{u_0\}\rangle\times \Q\langle\mathcal{U}\backslash\{u_0\}\rangle\to \Q\langle\mathcal{U}\backslash\{u_0\}\rangle$ can be described as follows. For any two words~$\word_1 = \uu_{n_1}\cdots \uu_{n_s},$~$\word_2 = \uu_{\ell_1}\cdots \uu_{\ell_\dr}\in\left(\mathcal{U}\backslash\{u_0\}\right)^\ast$, we set recursively
    \begin{align*}
        \word_1\tilde{\boxast} \word_2 := \begin{cases}
            0,\ \text{if } s>\dr,\\
            \word_2,\, \text{if } \word_1 = \bone,\\
            \uu_{\ell_1}\left(\word_1\tilde{\boxast} \uu_{\ell_2}\cdots \uu_{\ell_\dr}\right) + \uu_{n_1 + \ell_1}\left(\uu_{n_2}\cdots \uu_{n_s}\tilde{\boxast} \uu_{\ell_2}\cdots \uu_{\ell_\dr}\right),\ \text{if } s\leq \dr.
        \end{cases}
    \end{align*}
    Then,~$\word_1\tilde{\boxast} \word_2 = \word_1\boxast \word_2$ whenever~$\len(\word_1)\leq\len(\word_2)$.
\end{lemma}\noproof{lemma}
Note that the box product satisfies the following connection to the stuffle product.

\begin{lemma}
    \label{lem:stuffle-boxstuffle}
    For all indices of positive integers~$\boldsymbol{n_1},\boldsymbol{n_2},\boldsymbol{\ell}$, we have
    \begin{align*}
        \uu_{\boldsymbol{n_1}}\boxast (\uu_{\boldsymbol{n_2}}\boxast \uu_{\boldsymbol{\ell}}) = (\uu_{\boldsymbol{n_1}}\ast \uu_{\boldsymbol{n_2}})\boxast \uu_{\boldsymbol{\ell}} = \uu_{\boldsymbol{n_2}}\boxast (\uu_{\boldsymbol{n_1}}\boxast \uu_{\boldsymbol{\ell}}).
    \end{align*}
\end{lemma}

\begin{proof}
    The proof of the first equality follows by induction on~$\len(\boldsymbol{n_1}) + \len(\boldsymbol{n_2})$ and the definition of stuffle and box product. The second equality then follows from the commutativity of the stuffle product and the first equality.
\end{proof}


Next, we make an easy but instrumental observation. For this, we denote for an given index~$\mathbf{k} = (k_1,\dots,k_\dr)$ its \emph{reversed index} by~$\reverse{\mathbf{k}} := (k_\dr,\dots,k_1)$.

\begin{proposition}
\label{prop:reverse}
    Given~$\mathbf{n}\in\Z_{>0}^s,\, \boldsymbol{\ell}\in\Z_{>0}^\dd$ with~$1\leq s\leq\dd$. Writing
    \begin{align*}
        \uu_{\mathbf{n}}\boxast \uu_{\boldsymbol{\ell}} = \sum\limits_{\boldsymbol{\mu}\in\Z_{>0}^\dd} a_{\boldsymbol{\mu}} \uu_{\boldsymbol{\mu}}
    \end{align*}
    with~$a_{\boldsymbol{\mu}}\in\Z$ appropriate, we have
    \begin{align*}
        \uu_{\reverse{\mathbf{n}}}\boxast \uu_{\reverse{\boldsymbol{\ell}}} = \sum\limits_{\boldsymbol{\mu}\in\Z_{>0}^\dd} a_{\boldsymbol{\mu}} \uu_{\reverse{\boldsymbol{\mu}}}.
    \end{align*}
\end{proposition}
\begin{proof}
    Using Lemma~\ref{lem:boxproduct} and induction on~$\len(\mathbf{n})+\len(\boldsymbol{\ell})$, the claim follows immediately.
\end{proof}

\section{A common approach to theorems by Bachmann and Burmester}
\label{sec:annika}

In this section, we consider the cases of~$\dd=1$ (and~$z\in\Z_{>0}$ arbitrary), and~$z=1$ (and~$\dd\in\Z_{>0}$ arbitrary), respectively, of Bachmann's Conjecture~\ref{conj:mdbdApp}. The first case mainly is a result originally due to Bachmann (\cite[Proposition 4.4]{Ba1}), which we will reprove in a way giving explicit formulas for every element of~$\fil{Z,D,W}{z,1,w}\Zq$ as linear combination of elements in~$\fil{D,W}{z+1,w}\Zq$. The second case is done by Burmester's thesis ({\cite[Theorem 6.4]{Bu1}}), which we will extend in Section~\ref{ssec:annika}.

\subsection{Bachmann's Conjecture~\ref{conj:mdbdApp} for $(z,1,w)$}
By~{\cite[Proposition 4.4]{Ba1}} (see also Theorem~\ref{thm:mdbdknownbibr}(i)), it is known that Bachmann's Conjecture~\ref{conj:mdbdApp} is true for all triples~$(z,1,w)$. Here, we give an alternative proof which gives an explicit expression in terms of elements in~$\Zqz$.

\begin{proposition}
\label{prop:dep1explicit}
    For all~$k\in\Z_{>0}$ and~$z\in\Z_{\geq 0}$, we have that~$\fz{u_{k}u_0^z}$ equals
    \begin{align}
        & (-1)^z\sum\limits_{\substack{j_1,j_2\geq 0\\ j_1+j_2 = z}} \sum\limits_{\substack{n_0,\dots,n_{j_2}\geq 0\\ n_0+\cdots+n_{j_2} = k-1}}\sum\limits_{\substack{1\leq p\leq j_2\\ 0\leq \varepsilon_p\leq\min\{1,n_p\}}}\fz{u_{n_{j_2}-\varepsilon_{j_2}+1}\cdots u_{n_1-\varepsilon_1+1}u_{n_0+1}u_1^{j_1}}
        \\
        &+\sum\limits_{1\leq j\leq z}\sum\limits_{\substack{\ell_1,\dots,\ell_j\geq 1\\ \ell_1+\cdots+\ell_j \leq z}} \sum\limits_{\substack{j_1,j_2\geq 0\\ j_1+j_2 = z-\ell_1-\cdots-\ell_j}} 
        \\
        &\,\sum\limits_{\substack{n_0,\dots,n_{j_2}\geq 0\\ n_0+\cdots+n_{j_2} = k-1}}\sum\limits_{\substack{1\leq p\leq j_2\\ 0\leq \varepsilon_p\leq\min\{1,n_p\}}}(-1)^{z - j} \fz{u_1^{\ell_1}\ast\cdots\ast u_1^{\ell_j} \ast u_{n_{j_2}-\varepsilon_{j_2}+1}\cdots u_{n_1-\varepsilon_1+1}u_{n_0+1}u_1^{j_1}}.
    \end{align}
    In particular, we have~$\fz{u_{k}u_0^z}\in\fil{D,W}{z+1,k+z}\Zqz$, yielding Bachmann's Conjecture~\ref{conj:mdbdApp} for all triples~$(z,1,w)$.
\end{proposition}
\begin{proof}
    First note that a calculation, using the definition of the stuffle product, shows for all~$a\in\Z_{>0},\, b\in\Z_{\geq 0}$ the identity
    \begin{align}
    \eqlabel{eq:dep1abst}
        u_{a}u_0^{b}
        =\,& \sum\limits_{\ell=1}^{a-1} (-1)^{\ell-1} u_1^{\ell}\ast u_{a-\ell}u_0^{b} + (-1)^{a-1}h(a,b),
    \end{align}
    where~$h(a,b) := \sum\limits_{\substack{j_1,j_2\geq 0\\ j_1+j_2 = a-1}} u_1^{j_1 + 1}\left(u_1^{j_2}\ast u_0^{b}\right)$. Choosing~$a=z+1$ and~$b=k-1$, we obtain
    \begin{align}
        \eqlabel{eq:dep1ab}
        u_{z+1}u_0^{k-1}
        =\, \sum\limits_{\ell=1}^{z} (-1)^{\ell-1} u_1^{\ell}\ast u_{z+1-\ell}u_0^{k-1} + (-1)^{z}h(z+1,k-1).
    \end{align}
    Using the latter formula and~\eqref{eq:dep1abst} repeatedly, we obtain
    \begin{align}
    \eqlabel{eq:dep1id}
        \begin{split}
        u_{z+1}u_0^{k-1} =&\, \sum\limits_{1\leq j\leq z}\sum\limits_{\substack{\ell_1,\dots,\ell_j\geq 1\\ \ell_1+\cdots+\ell_j \leq z}} (-1)^{z - j} u_1^{\ell_1}\ast\cdots\ast u_1^{\ell_j}\ast h(z+1-\ell_1-\cdots - \ell_{j},k-1) 
        \\
        &\, + (-1)^z h(z+1,k-1).
        \end{split}
    \end{align}
    Now, note that for all~$a\in\Z_{>0}$ and~$b\in\Z_{\geq 0}$, we have
    \begin{align*}
        h(a,b) =&\, \sum\limits_{\substack{j_1,j_2\geq 0\\ j_1+j_2 = a-1}} u_1^{j_1+1}\left(u_1^{j_2}\ast u_0^b\right) 
        \\
        =&\, \sum\limits_{\substack{j_1,j_2\geq 0\\ j_1+j_2 = a-1}} \sum\limits_{\substack{n_0,\dots,n_{j_2}\geq 0\\ n_0+\cdots+n_{j_2} = b}}\sum\limits_{\substack{1\leq p\leq j_2\\ 0\leq \varepsilon_p\leq\min\{1,n_p\}}} u_1^{j_1+1}u_0^{n_0}u_1u_0^{n_1-\varepsilon_1}\cdots u_1 u_0^{n_{j_2}-\varepsilon_{j_2}}.
    \end{align*}
    Hence, by~$\tau$-invariance of formal~\qmzv s,
    \begin{align*}
        \fz{h(a,b)}
        =&\, \sum\limits_{\substack{j_1,j_2\geq 0\\ j_1+j_2 = a-1}} \sum\limits_{\substack{n_0,\dots,n_{j_2}\geq 0\\ n_0+\cdots+n_{j_2} = b}}\sum\limits_{\substack{1\leq p\leq j_2\\ 0\leq \varepsilon_p\leq\min\{1,n_p\}}}\fz{u_{n_{j_2}-\varepsilon_{j_2}+1}\cdots u_{n_1-\varepsilon_1+1}u_{n_0+1}u_1^{j_1}},
    \end{align*}
    implying the claim when using~\eqref{eq:dep1id} and~$\fz{u_ku_0^z} = \fz{\tau(u_ku_0^z)} = \fz{u_{z+1}u_0^{k-1}}$. From the obtained representation of~$\fz{u_ku_0^z}$, we get directly~$\fz{u_ku_0^z}\in\fil{D,W}{z+1,k+z}\Zqz$ due to~\eqref{eq:basicstuffle1}.
\end{proof}
Let us consider an example regarding Proposition~\ref{prop:dep1explicit}.
\begin{example}
    For~$k=z=2$, Proposition~\ref{prop:dep1explicit} yields
    \begin{align*}
        &\, \fz{u_2u_0^2}
        \\
        =&\, \fz{u_1\ast u_1 \ast u_2} - 2\fz{u_1\ast u_2u_1} - \fz{u_1\ast u_1u_2} - \fz{u_1\ast u_1u_1} - \fz{u_1u_1\ast u_2} 
        \\
        &\,+ 3\fz{u_2u_1u_1} + 2\fz{u_1u_2u_1} + \fz{u_1u_1u_2} + 3\fz{u_1u_1u_1}
        \\
        =&\, \fz{u_4} - \fz{u_3u_1} - \fz{u_2u_2} - \fz{u_2u_1} - \fz{u_1u_2}\in \fil{D,W}{2,4}\Zqz\subset\fil{D,W}{3,4}\Zqz.
    \end{align*}
\end{example}\noproof{example}

\subsection{Bachmann's Conjecture~\ref{conj:mdbdApp} for $(1,\dd,w)$}
\label{ssec:annika}

Given an index~$\mathbf{k} = (k_1,\dots,k_\dd)\in\Z_{>0}^\dd$, we introduce the following notation of subindices
\begin{align*}
    \mathbf{k}_{(j_1;j_2)} := \begin{cases}
        (k_{j_1},\dots,k_{j_2}),\quad &\text{if } 1\leq j_1\leq j_2\leq\dd,\\ \emptyset,\quad &\text{else.}
    \end{cases}
\end{align*}

\begin{lemma}
    \label{lem:specialcase}
    Fix~$z,\dd\in\Z_{>0}$ and~$\mathbf{k}\in\Z_{>0}^\dd$. 
    For~$1\leq j\leq \dd$, we have
    \begin{align}
        \fz{u_{k_1}\left(u_{\mathbf{k}_{(2;j)}}\ast u_{\mathbf{k}_{(j+1;\dd)}} u_0^z \right)} \in\sum\limits_{s=1}^z\fil{Z,D,W}{z-s,\dd+s,w}\Zq,
    \end{align}
    where $w=|\mathbf{k}| + z$.
\end{lemma}

\begin{proof}
    We prove by induction on~$\dd$. The base case~$\dd =1$ corresponds to Proposition~\ref{prop:dep1explicit} since then~$j=1$ and so~$\mathbf{k}_{(2;j)} = \mathbf{k}_{(j+1;\dd)} = \emptyset$. Hence, we may assume~$\dd>1$ and that Lemma~\ref{lem:specialcase} is proven already for all smaller values of~$\dd$. First, note that the case~$j=\dd$ follows from~$\mathbf{k}_{(j+1;\dd)} = \emptyset$ in this case and from
    \begin{align*}
        \sum\limits_{\substack{n_1,\dots,n_{s'}\geq 1\\ n_1+\cdots+n_{s'} = z\\ 1\leq s'\leq \dd}} \fz{u_{n_1}\cdots u_{n_{s'}} \ast u_{\mathbf{k}}} \in \sum\limits_{s=1}^z\fil{Z,D,W}{z-s,\dd+s,w}\Zq
    \end{align*}
    since
    \begin{align}
        &\sum\limits_{\substack{n_1,\dots,n_{s'}\geq 1\\ n_1+\cdots+n_{s'} = z\\ 1\leq s'\leq \dd}} \fz{u_{n_1}\cdots u_{n_{s'}} \ast \tau(u_{\mathbf{k}})}
        \\
        =\ & \sum\limits_{\substack{n_1,\dots,n_{s'}\geq 1\\ n_1+\cdots+n_{s'} = z\\ 1\leq s'\leq \dd}} \fz{u_{n_1}\cdots u_{n_{s'}} \ast u_1 u_0^{k_\dd-1}\cdots u_1 u_0^{k_1-1}}
        \\
        =\ &\sum\limits_{\substack{n_1,\dots,n_{s'}\geq 1\\ n_1+\cdots+n_{s'} = z\\ 1\leq s'\leq \dd}} \fz{\tau\left(u_{n_1}\cdots u_{n_{s'}} \ast u_1 u_0^{k_\dd-1}\cdots u_1 u_0^{k_1-1}\right)}
        \\
        \equiv\ &\sum\limits_{\substack{n_1,\dots,n_{s'}\geq 1\\ n_1+\cdots+n_{s'} = z\\ 1\leq s'\leq \dd}} \fz{\tau\left(u_{n_1}\cdots u_{n_{s'}} \boxast u_1 u_0^{k_\dd-1}\cdots u_1 u_0^{k_1-1}\right)}\mod \sum\limits_{s=1}^z\fil{Z,D,W}{z-s,\dd+s,w}\Zq.
    \end{align}
    The last identity is a consequence of Proposition~\ref{prop:stufflemaxzero} and the definition of the box product. Furthermore, the remaining expression is
    \begin{align*}
        \equiv \fz{u_{k_1}\left(u_{\mathbf{k}_{(2;\dd)}} \ast u_0^z\right)}\mod \sum\limits_{s=1}^z\fil{Z,D,W}{z-s,\dd+s,w}\Zq,
    \end{align*}
    which can be verified via induction on~$s'+\dd$ and the definition of the stuffle product. Hence, let be~$1\leq j\leq\dd-1$ and assume that the claim holds for all larger values of~$j$. The induction hypothesis on~$d$ implies, since~$\len(\emptyset) + \len\left(\mathbf{k}_{(j+2;\dd)}\right) = \dd-j-1 < \dd - 1$, 
    \begin{align*}
        \fz{u_{\mathbf{k}_{(j+1;\dd)}}u_0^z} = \fz{u_{k_{j+1}}(u_{\emptyset}\ast u_{\mathbf{k}_{(j+2;\dd))}} u_0^z)}\in \sum\limits_{s=1}^z\fil{Z,D,W}{z-s,\dd-j+s,w'},
    \end{align*}
    where~$w'= |\mathbf{k}_{(j+1;\dd)}| + z$. Hence, by~\eqref{eq:basicstuffle1}, we obtain
    \begin{align}
    \eqlabel{eq:uksinzq}
        \fz{u_{\mathbf{k}_{(1;j)}}\ast u_{\mathbf{k}_{(j+1;\dd)}} u_0^z} \in\sum\limits_{s=1}^z\fil{Z,D,W}{z-s,\dd+s,w}\Zq.
    \end{align}
    Now, using the definition of the stuffle product, we obtain
    \begin{align}
        u_{\mathbf{k}_{(1;j)}}\ast u_{\mathbf{k}_{(j+1;\dd)}} u_0^z
        =&\,u_{k_1}\left(u_{\mathbf{k}_{(2;j)}}\ast u_{\mathbf{k}_{(j+1;\dd)}} u_0^z\right) + u_{k_{j+1}}\left(u_{\mathbf{k}_{(1;j)}}\ast u_{\mathbf{k}_{(j+2;\dd)}} u_0^z\right)
        \\
        &\,+ u_{k_1+k_{j+1}}\left(u_{\mathbf{k}_{(2;j)}}\ast u_{\mathbf{k}_{(j+2;\dd)}} u_0^z\right).
    \end{align}
    Note that the formal~\qmzv of the second summand on the right-hand side is an element of~$\sum\limits_{s=1}^z\fil{Z,D,W}{z-s,\dd+s,w}\Zq$ due to the assumption on~$j$, while the formal~\qmzv\ of the third one is by induction hypothesis on~$\dd$. Hence, because of~\eqref{eq:uksinzq}, we obtain
    \begin{align}
        \fz{u_{k_1}\left(u_{\mathbf{k}_{(2;j)}}\ast u_{\mathbf{k}_{(j+1;\dd)}} u_0^z\right)}\in \sum\limits_{s=1}^z\fil{Z,D,W}{z-s,\dd+s,w}\Zq,
    \end{align}
    completing the induction step. Therefore, the lemma is proven.
\end{proof}

\begin{corollary}
\label{cor:z1111111=0}
    Fix~$z,\dd\in\Z_{>0}$. For all~$\mathbf{k}\in\Z_{>0}^\dd$, we have
    \begin{align}
        \fz{u_{\mathbf{k}} u_0^z}\in\sum\limits_{s=1}^z\fil{Z,D,W}{z-s,\dd+s,w}\Zq,
    \end{align}
    where $w = |\mathbf{k}|+z$. 
\end{corollary}
\begin{proof}
    This is the special case~$j=1$ of Lemma~\ref{lem:specialcase}.
\end{proof}

\begin{corollary}
\label{cor:z1111111=02}
    Fix~$\dd\in\Z_{>0}$. For all~$\mathbf{k}\in\Z_{>0}^\dd$, we have
    \begin{align}
        \fz{u_{\mathbf{k}} u_0 u_0} \in\fil{D,W}{\dd+2,w}\Zqz,
    \end{align}
    where~$w = |\mathbf{k}|+2$.
\end{corollary}
\begin{proof}
    The special case~$z=2$ of Corollary~\ref{cor:z1111111=0} and~$\fil{Z,D,W}{1,\dd+1,w}\Zq \subset \fil{D,W}{\dd+2,w}\Zqz$ by Burmester's Theorem~\ref{thm:mdbdknownbibr}(iv) yield the claim.
\end{proof}

Lemma~\ref{lem:specialcase} is a special case of the following theorem.

\begin{theorem}[Theorem~\ref{thm:sumtozApp}]
\label{thm:sumtoz2App}
    Let be~$z,\dd\in\Z_{>0}$,~$\mathbf{k} = (k_1,\dots,k_\dd)\in\Z_{>0}^\dd$, and consider integers~$1\leq j_1 \leq j_2\leq \dd$. We have
    \begin{align}
        \eqlabel{eq:sumtoz3App}\fz{u_{\mathbf{k}_{(1;j_1)}}\left(u_{\mathbf{k}_{(j_1+1;j_2)}}\ast u_{\mathbf{k}_{(j_2+1;\dd)}}u_0^z\right)}\in\sum\limits_{s=1}^z \fil{Z,D,W}{z-s,\dd+s,w}\Zq,
    \end{align}
    where~$w= |\mathbf{k}| + z$. 
\end{theorem}
\begin{proof}
    We prove by induction on~$\dd$. Note that the base case~$\dd=1$ follows from Proposition~\ref{prop:dep1explicit} since then~$j_1=j_2=1$ and so~$\mathbf{k}_{(j_1+1;j_2)} = \mathbf{k}_{(j_2+1;\dd)} = \emptyset$. Hence, choose~$\dd>1$ and assume the theorem is proven for all smaller values of~$\dd$. Furthermore, note that the case~$j_1=1$ is nothing else than Lemma~\ref{lem:specialcase}. Hence, let~$2\leq j_1\leq\dd$ arbitrary. The claim for~$j_2=j_1$ corresponds to Corollary~\ref{cor:z1111111=0} since then~$\mathbf{k}_{(j_1+1;j_2)} = \emptyset$. Therefore, assume~$j_2>j_1>1$ in the following and that the claim is proven for all possible smaller values of~$j_1$,~$j_2$ and~$\len(\mathbf{k}_{(j_1+1;j_2)}) = j_2-j_1$, respectively. Using the recursive definition of the stuffle product gives
    \begin{align*}
        &u_{\mathbf{k}_{(1;j_1)}}\left(u_{\mathbf{k}_{(j_1+1;j_2)}}\ast u_{\mathbf{k}_{(j_2+1;\dd)}}u_0^z\right)
        \\
        =\, &u_{\mathbf{k}_{(1;j_1-1)}}\left(u_{\mathbf{k}_{(j_1+1;j_2)}}\ast u_{k_{j_1}} u_{\mathbf{k}_{(j_2+1;\dd)}}u_0^z\right)
        \\
        &\,-u_{\mathbf{k}_{(1;j_1-1)}}u_{k_{j_1+1}}\left(u_{\mathbf{k}_{(j_1+2;j_2)}}\ast  u_{k_{j_1}}u_{\mathbf{k}_{(j_2+1;\dd)}}u_0^z\right)
        \\
        &\,-u_{\mathbf{k}_{(1;j_1-1)}}u_{k_{j_1}+k_{j_1+1}}\left(u_{\mathbf{k}_{(j_1+2;j_2)}}\ast u_{\mathbf{k}_{(j_2+1;\dd)}}u_0^z\right)
    \end{align*}
    Now, the formal~\qmzv of the first summand on the right-hand side is in~$\sum\limits_{s=1}^z \fil{Z,D,W}{z-s,\dd+s,w}\Zq$ due to the assumption on~$j_1$ (since~$\len\left(\mathbf{k}_{(1;j_1-1)}\right) = \len\left(\mathbf{k}_{(1;j_1)}\right) - 1$), while the second one is as well due to the assumption on~$j_2-j_1$ (since~$\len\left(\mathbf{k}_{(j_1+2;j_2)}\right) = \len\left(\mathbf{k}_{(j_1+1;j_2)}\right)-1$), and the third one is due to the induction hypothesis on~$\dd$. In particular, we have
    \begin{align*}
        \fz{u_{\mathbf{k}_{(1;j_1)}}\left(u_{\mathbf{k}_{(j_1+1;j_2)}}\ast u_{\mathbf{k}_{(j_2+1;\dd)}}u_0^z\right)}\in \sum\limits_{s=1}^z \fil{Z,D,W}{z-s,\dd+s,w}\Zq,
    \end{align*}
    completing the induction step. Hence, the theorem follows.    
\end{proof}

\begin{corollary}
\label{cor:sumtozApp}
    Let be~$z,\dd\in\Z_{>0}$,~$\mathbf{k} = (k_1,\dots,k_\dd)\in\Z_{>0}^\dd$. For all $1\leq j\leq\dd$, we have
    \begin{align}
        \eqlabel{eq:sumtoz2App}
        \sum\limits_{\substack{\ell_j,\dots,\ell_{\dd}\geq 0\\ \ell_j+\cdots+\ell_{\dd} = z}} \fz{u_{\mathbf{k}_{(1;j-1)}} u_{k_{j}} u_0^{\ell_j}\cdots u_{k_\dd} u_0^{\ell_{\dd}}} \in\sum\limits_{s=1}^z \fil{Z,D,W}{z-s,\dd+s,w}\Zq,
    \end{align}
    where~$w= |\mathbf{k}| + z$.
\end{corollary}
\begin{proof}
    For fixed~$1\leq j\leq\dd$, the corollary is obtained from the special case~$j_1 = j$, $j_2 = \dd$ of Theorem~\ref{thm:sumtoz2App} and multiplying out the corresponding stuffle product occurring in~\eqref{eq:sumtoz3App} (since then~$\mathbf{k}_{(j_2+1;\dd)} = \emptyset$).
\end{proof}

As a corollary of Corollary~\ref{cor:sumtozApp}, we obtain Burmester's Theorem~\ref{thm:mdbdknownbibr}(iv).
\begin{corollary}[Burmester, {\cite[Theorem 6.4]{Bu1}}]
    \label{cor:z=1}
    Bachmann's Conjecture~\ref{conj:mdbdApp} is true for all~$(1,\dd,w)\in\Z_{>0}^3$.
\end{corollary}
\begin{proof}
    Let be~$\dd\in\Z_{>0}$,~$\mathbf{k} = (k_1,\dots,k_\dd)\in\Z_{>0}^\dd$ and denote~$w=|\mathbf{k}|+1$ in the following. Considering Corollary~\ref{cor:sumtozApp} with~$z=1$ and~$j=\dd$, we obtain~$\fz{u_{\mathbf{k}}u_0}\in\fil{D,W}{\dd+1,w}\Zqz$. Now, let be~$1\leq j'\leq\dd-1$. Considering the difference of~\eqref{eq:sumtoz2App} with~$z=1$,~$j=j'$ and~\eqref{eq:sumtoz2App} with~$z=1$,~$j=j'+1$, we obtain 
    \begin{align*}
        \fz{u_{\mathbf{k}_{(1;j')}} u_0 u_{\mathbf{k}_{(j'+1;\dd)}}}\in\fil{D,W}{\dd+1,w}\Zqz.
    \end{align*}
    In particular, for every~$\word\in\mathcal{U}^{\ast,\circ}\cap\fil{Z,D,W}{1,\dd,w}\QB^\circ$, we have shown~$\fz{\word}\in\fil{D,W}{\dd+1,w}\Zqz$, i.e., we have~$\fil{Z,D,W}{1,\dd,w}\Zq\subset\fil{D,W}{\dd+1,w}\Zqz$, completing the claim.
\end{proof}

\begin{corollary}
    \label{cor:11111112}
    Let be~$\dd\in\Z_{\geq 2}$ and~$\mathbf{k} = (k_1,\dots,k_\dd)\in\Z_{>0}^\dd$. We have
    \begin{align}
        \eqlabel{eq:11111112}
        \fz{u_{k_1}u_0 u_{k_2} u_0 u_{\mathbf{k}_{(3;\dd)}}} \in\fil{D,W}{\dd+2,w}\Zqz,
    \end{align}
    where $w = |\mathbf{k}|+2$.
\end{corollary}

\begin{proof}
    Consider the difference of~\eqref{eq:sumtoz2App} for~$z=2$,~$j=2$, and~\eqref{eq:sumtoz2App} for~$z=2$,~$j=3$ to obtain, all congruences modulo~$\fil{Z,D,W}{1,\dd+1,w}\Zq$,
    \begin{align*}
        0\equiv\ & \sum\limits_{\substack{\ell_3,\dots,\ell_\dd\geq 0\\ \ell_3+\cdots+\ell_\dd = 2}} \fz{u_{k_1} u_{k_2} u_{k_3} u_0^{\ell_3}\cdots u_{k_\dd} u_0^{\ell_\dd}} - \sum\limits_{\substack{\ell_2,\dots,\ell_\dd\geq 0\\ \ell_2+\cdots+\ell_\dd = 2}} \fz{u_{k_1} u_{k_2} u_0^{\ell_2}\cdots u_{k_\dd} u_0^{\ell_\dd}}
        \\
        \equiv\ & - \fz{u_{k_1} u_{k_2} u_0 u_0 u_{\mathbf{k}_{(3;\dd)}}} - \sum\limits_{\substack{\ell_3,\dots,\ell_\dd\geq 0\\ \ell_3+\cdots+\ell_\dd = 1}} \fz{u_{k_1} u_{k_2} u_0 u_{k_3} u_0^{\ell_3}\cdots u_{k_\dd} u_0^{\ell_\dd}}
        \\
        \equiv\ & - \fz{u_1u_0^{k_\dd-1}\cdots u_1 u_0^{k_3-1}u_3u_0^{k_2-1}u_1 u_0^{k_1-1}} 
        \\
        &- \sum\limits_{\substack{\ell_3,\dots,\ell_\dd\geq 0\\ \ell_3+\cdots+\ell_\dd = 1}} \fz{u_{\ell_\dd + 1} u_0^{k_\dd-1}\cdots u_{\ell_3+1} u_0^{k_3-1} u_2 u_0^{k_2-1} u_1 u_0^{k_1-1}}
        \\
        \equiv\ & \fz{u_1u_0^{k_\dd-1}\cdots u_1 u_0^{k_3-1}u_2u_0^{k_2-1}u_2 u_0^{k_1-1}} - \fz{u_1\ast \tau\left(u_{k_1}u_{k_2}u_0u_{\mathbf{k}_{(3;\dd)}}\right)} 
        \\
        \equiv\ & \fz{u_{k_1}u_0 u_{k_2} u_0 u_{\mathbf{k}_{(3;\dd)}}}.
    \end{align*}
    Since~$\fil{Z,D,W}{1,\dd+1,w}\Zq \subset \fil{D,W}{\dd+2,w}\Zqz$ by Corollary~\ref{cor:z=1}, the claim follows.
\end{proof}

\section{Investigation of the box product}
\label{sec:boxprod}

First, in Section~\ref{ssec:mdbd:monomials}, we show that several monomials can already be written as a~$\Q$-linear combination of non-trivial box products. In Section~\ref{ssec:mdbd:conjecturebox}, we investigate a conjecture (Conjecture~\ref{conj:systemApp}) regarding the structure of box products and give partial results for it. Furthermore, in Section~\ref{ssec:mdbd:connectionbox}, we study the main connection between the box product and the stuffle product that we will need to prove our main results. Last, in Section~\ref{ssec:mdbd:supplementary}, we give some further results about the box product that are interesting for itself but not necessary for the remaining paper.

\subsection{Monomials as linear combination of box products}
\label{ssec:mdbd:monomials}

In the following, we characterize some particular monomials in~$\Q\langle\mathcal{U}\backslash\{u_0\}\rangle$ as a linear combination of box products. The results will be important for proving Theorem~\ref{thm:rleq42App}.

We will need the $\Q$-vector space spanned by (non-trivial) box products in the following.

\begin{definition}
    \label{def:curlyPApp}
    We define
    \begin{align}
    \eqlabel{eq:boxprodQspanApp}
        \CP := \spanQ{\word_1\boxast \word_2\mid \word_1,\, \word_2\in\left(\mathcal{U}\backslash\{u_0\}\right)^\ast,\, 
        \word_1,\word_2\neq\bone}\subset \Q\langle\mathcal{U}\backslash\{u_0\}\rangle.
    \end{align}
\end{definition}\noproof{definition}

\begin{corollary}
\label{cor:reverse}
    Given~$\boldsymbol{\mu}\in\Z_{>0}^\dd$ with~$\dd\in\Z_{>0}$. Then~$\uu_{\boldsymbol{\mu}}\in\CP$ if and only if~$\uu_{\reverse{\boldsymbol{\mu}}}\in\CP$.
\end{corollary}
\begin{proof}
    This is an immediate consequence of Proposition~\ref{prop:reverse}.
\end{proof}

\begin{lemma}
\label{lem:r+1}
    For all~$\dd\in\Z_{>0}$ and~$0\leq j\leq \dd-1$, we have
    \begin{align*}
        \uu_2^\dd,\ \uu_1^j \uu_{1+\dd} \uu_1^{\dd-j-1} \in\CP.
    \end{align*}
\end{lemma}

\begin{proof}
    A direct calculation shows~$\uu_2^\dd =\,  \uu_1^\dd\boxast \uu_1^\dd$, giving the first part of the lemma. Furthermore, for all~$0\leq j\leq \dd-1$, we have
    \begin{align*}
        \uu_1^j \uu_{1+\dd} \uu_1^{\dd-j-1}
        =\,& \sum\limits_{a=1}^{\dd} (-1)^{a-1}\, \uu_1^{a}\boxast \uu_1^j \uu_{1+\dd-a} \uu_1^{\dd-j-1},
    \end{align*}
    giving the second claim of the lemma. 
\end{proof}

\begin{lemma}
\label{lem:123}
    For arbitrary~$\dd\in\Z_{>0}$ and~$0\leq j\leq \dd - 2$, we have
    \begin{align*}
        \uu_1 \uu_2^j \uu_3 \uu_2^{\dd-j-2} \in\CP.
    \end{align*}
\end{lemma}

\begin{proof}
    For any~$0\leq j\leq \dd - 2$, one verifies
    \begin{align*}
        & \uu_1 \uu_2^{j} \uu_3 \uu_2^{\dd-j-2}
        \\
        =\, & \sum\limits_{a=1}^{j+1} (-1)^{a+1}\, \uu_1^{j-a+1} \uu_2 \uu_1^{\dd-j-2} \boxast \uu_a \uu_1^{\dd-1}
         + \sum\limits_{a=1}^{j+2} (-1)^{a+1}\, \uu_1^{\dd-a+1} \boxast \uu_a \uu_1^{\dd-1}.\tag*{\qedhere}
    \end{align*}
\end{proof}

We first need an auxiliary lemma to prove the statements in Corollary~\ref{cor:sum2} and Lemma~\ref{lem:21}.

\begin{lemma}
\label{lem:Hsum2}
    For all~$\dd,\mu_1,\mu_2\in\Z_{>0}$ with~$\mu_1 + \mu_2\leq \dd+2$, we have
    \begin{align*}
        \uu_{\mu_1} \uu_{\mu_2} (\uu_1^{\dd-\mu_1-\mu_2 + 2}\boxast \uu_1^{\dd - 2}) \in\CP.
    \end{align*}
\end{lemma}

\begin{proof}
    We prove by induction on~$\mu_1$. First, consider~$\mu_1 = 1$. Similarly to the proof of Lemma~\ref{lem:123}, we obtain by direct calculation that
    \begin{align*}
        &\uu_1 \uu_{\mu_2} (\uu_1^{\dd-\mu_2 + 1}\boxast \uu_1^{\dd - 2}) 
        \\
        =& - \uu_1 \uu_{\mu_2 - 1} (\uu_1^{\dd-\mu_2 + 2}\boxast \uu_1^{\dd - 2})
         - \sum\limits_{\substack{0\leq a\leq\mu_2-3\\ 0\leq b\leq 1+a}}
         (-1)^{b}\, \uu_1^{\dd-\mu_2+b+2} \boxast \uu_{2+a-b}\uu_{\mu_2-2-a}\uu_1^{\dd - 2} 
        \\
        & \quad +\sum\limits_{\substack{0\leq a\leq \mu_2-3\\ 0\leq b\leq a}}
         (-1)^{a+b}\, \uu_1^{a-b} \uu_2 \uu_1^{\dd-\mu_2+1} \boxast \uu_{1+b} \uu_{\mu_2-2-a} \uu_1^{\dd - 2}.
    \end{align*}
    Hence, we have for all~$\mu_2\in\Z_{>0}$ that~$\uu_1 \uu_{\mu_2} (\uu_1^{\dd-\mu_2 + 1}\boxast \uu_1^{\dd - 2})\in\CP$ if and only if we have~$\uu_1 \uu_{\mu_2 - 1} (\uu_1^{\dd-\mu_2 + 2}\boxast \uu_1^{\dd - 2})\in\CP$, giving recursively that~$\uu_1 \uu_{\mu_2} (\uu_1^{\dd-\mu_2 + 1}\boxast \uu_1^{\dd - 2})\in\CP$ if and only if
    \begin{align*}
        \uu_1 \uu_{3} (\uu_1^{\dd - 2}\boxast \uu_1^{\dd - 2}) \in\CP,
    \end{align*}
    which is true since this is the~$j=0$ case of Lemma~\ref{lem:123}.

    Now, for~$\mu_1 > 1$, assume that the lemma is proven for~$\mu_1-1$ already. We calculate
    \begin{align*}
        \uu_{\mu_1}\uu_{\mu_2}(\uu_1^{\dd-\mu_1-\mu_2 + 2}\boxast \uu_1^{\dd - 2})
        =& \sum\limits_{a=\mu_2}^{\dd-\mu_1+2} (-1)^{\mu_2+a}\, \uu_1^{\dd-\ell_1-a+3}\boxast \uu_{\mu_1-1}\uu_a\uu_1^{\dd - 2}
        \\
        &\quad - \uu_{\mu_1-1}\uu_{\ell_2}(\uu_1^{\dd-\mu_1-\mu_2+3}\boxast \uu_1^{\dd - 2})
        \\
        &\quad + (-1)^{\dd-\mu_1+1-\mu_2}\, \uu_{\mu_1-1}\uu_{\dd-\mu_1+3}\uu_1^{\dd - 2}.
    \end{align*}
    I.e., we have~$\uu_{\mu_1}\uu_{\mu_2}(\uu_1^{\dd-\mu_1-\mu_2 + 2}\boxast \uu_1^{\dd - 2})\in\CP$ by the assumption that the lemma is proven for~$\mu_1-1$.
\end{proof}

\begin{corollary}
\label{cor:sum2}
    For all~$\dd\in\Z_{>0}$ and~$0\leq j\leq \dd$, we have
    \begin{align*}
        \uu_{1+j} \uu_{\dd-j+1} \uu_1^{\dd - 2} \in\CP.
    \end{align*}
\end{corollary}
\begin{proof}
    Setting~$\mu_1 = 1+j$ and~$\mu_2 = \dd-j+1$ in Lemma~\ref{lem:Hsum2}, we obtain the claim.
\end{proof}

Furthermore, Lemma~\ref{lem:Hsum2} is used to prove the following observation.
\begin{lemma}
\label{lem:21}
    For arbitrary~$\dd\in\Z_{>0}$ and all~$0\leq j\leq \dd-3$, we have
    \begin{align*}
        \uu_2 \uu_1 \uu_2^j \uu_3 \uu_2^{\dd-j-3} \in\CP.
    \end{align*}
\end{lemma}

\begin{proof}
    First, a direct calculation gives for all~$0\leq j\leq \dd-3$ that
    \begin{align*}
        &\uu_2\uu_1\uu_2^j\uu_3\uu_2^{\dd-j-3}
        \\
        =\,&\sum\limits_{a=2}^{j+2} (-1)^a\, \uu_1^{j-a+2} \uu_2 \uu_1^{\dd-j-3} \boxast \uu_a\uu_1^{\dd-1}
        +\sum\limits_{a=1}^{j+3} (-1)^{a}\, \uu_1^{\dd-a+1} \boxast \uu_a\uu_1^{\dd-1}
        \\
        &-\sum\limits_{\substack{a,b\geq 2\\ a+b\leq j+3}} (-1)^{a+b}\, \uu_1^{j-(a+b)+3} \uu_2 \uu_1^{\dd-j-3} \boxast \uu_{a}\uu_{b}\uu_1^{\dd - 2}
        \\
        &+(-1)^{j+3}\, \uu_2\uu_{j+3}(\uu_1^{\dd-j-3}\boxast \uu_1^{\dd - 2})
        \\
        &+ (-1)^{j+3}\sum\limits_{\substack{a,b\geq 2\\ a+b=j+3}} \uu_{a+2}\uu_{b+1}(\uu_1^{\dd-j-4}\boxast \uu_1^{\dd - 2}) + \uu_{a+2}\uu_{b}(\uu_1^{\dd-j-3}\boxast \uu_1^{\dd - 2}).
    \end{align*}
    Using Lemma~\ref{lem:Hsum2} now yields the claim.
\end{proof}

Collecting the results of this subsection, we have proven the following theorem.

\begin{theorem} 
\label{thm:boxexplicit} 
Let be~$\dd\in\Z_{>0}$.
    \begin{enumerate}
        \item For all~$0\leq j\leq \dd - 2$, we have
    \begin{align*}
        \uu_1 \uu_2^j \uu_3 \uu_2^{\dd-j-2},\ \uu_2^j \uu_3 \uu_2^{\dd-j-2} \uu_1 \in\CP.
    \end{align*}
    \item For all~$0\leq j\leq \dd$, we have
    \begin{align*}
         \uu_{1+j} \uu_{\dd-j+1} \uu_1^{\dd - 2},\ \uu_1^{\dd - 2} \uu_{\dd-j+1} \uu_{1+j}  \in\CP.
    \end{align*}
    \item For all~$0\leq j\leq \dd-3$, we have
    \begin{align*}
        \uu_2 \uu_1 \uu_2^j \uu_3 \uu_2^{\dd-j-3},\ \uu_2^{j} \uu_3 \uu_2^{\dd-j-3} \uu_1 \uu_2 \in\CP.
    \end{align*}
    \end{enumerate}
\end{theorem}

\begin{proof}
    Using Corollary~\ref{cor:reverse} each, the proof for (i) follows from Lemma~\ref{lem:123}, the proof of (ii) follows from Corollary~\ref{cor:sum2}, and the proof of (iii) follows from Lemma~\ref{lem:21}.
\end{proof}

\subsection{Conjectures about particular box products and implications}
\label{ssec:mdbd:conjecturebox}

We consider in this section the structure of all box products~$u_{\mathbf{n}}\boxast u_{\boldsymbol{\ell}}$ such that~$\len(\boldsymbol{\ell})$ and~$|\mathbf{n}| + |\boldsymbol{\ell}|$ are fixed. For this, we will need the spaces~$\VSbox{z}{\dd}$ and~$\Vbox{z}{\dd}$ in the following.

\begin{definition}
    \begin{enumerate}
        \item For all~$z,\dd\in\Z_{>0}$, we define
    \begin{align*}
        \Vbox{z}{\dd} :=&\, \spanQ{ \uu_{\boldsymbol{\mu}}\mid\boldsymbol{\mu}\in\Z_{>0}^\dd,\, |\boldsymbol{\mu}| = z+\dd},
        \\
        \tdim{z}{\dd} :=&\,  \dim_\Q\Vbox{z}{\dd}.
    \end{align*}
    \item Furthermore, for all~$z,\dd\in\Z_{>0}$, we define
    \begin{align*}
    \indexset{z}{\dd} :=&\, \left\{(\mathbf{n},\boldsymbol{\ell})\,\Big|\, \mathbf{n}\in\Z_{>0}^s,\, \boldsymbol{\ell}\in\Z_{>0}^{\dd},\, 1\leq s\leq \dd,\, |\mathbf{n}|+|\boldsymbol{\ell}| = z+\dd\right\},
    \\
    \idim{z,\dd} :=&\, \#\indexset{z}{\dd}.
    \end{align*}
    and
    \begin{align*}
        \VSbox{z}{\dd} :=&\, \spanQ{\uu_{\mathbf{n}}\boxast \uu_{\boldsymbol{\ell}}\mid (\mathbf{n},\boldsymbol{\ell})\in\indexset{z}{\dd}} = \Vbox{z}{\dd}\cap\CP,
        \\
        \sdim{z}{\dd} :=&\, \dim_\Q\VSbox{z}{\dd}.
    \end{align*}
    \end{enumerate}
\end{definition}\noproof{definition}

Based on numerical calculations (see Lemma~\ref{lem:rzleq8}), we conjecture the following for the dimension of~$\VSbox{z}{\dd}$.

\begin{conjecture}
    \label{conj:systemApp}
    For all~$z,d\in\Z_{>0}$, we have
    \begin{align}
    \eqlabel{eq:systemApp}
        \sdim{z}{\dd} = \binom{z+\dd-1}{\min\{z,\dd\} - 1}.
    \end{align}
\end{conjecture}\noproof{conjecture}

Given~$(z_0,\dd_0)\in\Z_{>0}$, we say that Conjecture~\ref{conj:systemApp} \emph{is true for}~$(z_0,\dd_0)$ if~\eqref{eq:systemApp} is true for~$(z,\dd) = (z_0,\dd_0)$. Note the following equivalent formulation for~$z\geq\dd$.

\begin{corollary}
\label{cor:systemApp}
    Given~$(z,\dd)\in\Z^2_{>0}$ with~$z\geq \dd$. Conjecture~\ref{conj:systemApp} is true for~$(z,\dd)$ if and only if~$\VSbox{z}{\dd} = \Vbox{z}{\dd}$.
\end{corollary}
\begin{proof}
    Clearly, for all~$z,\dd\in\Z_{>0}$, one has
    \begin{align*}
        \tdim{z}{\dd} = \binom{z+\dd-1}{\dd-1}
    \end{align*}
    since~$\tdim{z}{\dd}$ is the number of compositions of~$z+\dd$ into exactly~$\dd$ positive integers. Hence, for~$(z,\dd)\in\Z^2_{>0}$ with~$z\geq \dd$, Conjecture~\ref{conj:systemApp} is equivalent to~$\sdim{z}{\dd} = \tdim{z}{\dd}$ which is equivalent to~$\VSbox{z}{\dd} = \Vbox{z}{\dd}$ since~$\VSbox{z}{\dd} \subset \Vbox{z}{\dd}$ and both~$\VSbox{z}{\dd}$ and~$\Vbox{z}{\dd}$ are finite-dimensional~$\Q$-vector spaces.
\end{proof}

\begin{theorem}
    \label{thm:sys1}
    Fix~$\dd\in\Z_{>0}$. If Conjecture~\ref{conj:systemApp} is true for~$(\dd,\dd)$, then it is also true for all~$(z,\dd)\in\Z_{>0}^2$ with~$z>\dd$.
\end{theorem}

\begin{proof}
    Fix~$\dd\in\Z_{>0}$ and assume that Conjecture~\ref{conj:systemApp} is true for~$(\dd,\dd)$. I.e., by Corollary~\ref{cor:systemApp}, we assume~$\VSbox{\dd}{\dd} = \Vbox{\dd}{\dd}$. This is equivalent to
    \begin{align*}
        \uu_\mathbf{z} = \sum\limits_{(\mathbf{n},\boldsymbol{\ell})\in\indexset{\dd}{\dd}} a_{\mathbf{n},\boldsymbol{\ell}}(\mathbf{z})\, \uu_{\mathbf{n}}\boxast \uu_{\boldsymbol{\ell}}
    \end{align*}
    for all~$\mathbf{z} = (z_1,\dots,z_\dd)\in\Z_{>0}^\dd$ with~$|\mathbf{z}| = 2\dd$ and with~$a_{\mathbf{n},\boldsymbol{\ell}}(\mathbf{z})\in\Q$ appropriate.

    Now, assume~$z>\dd$ and let be~$\mathbf{z} = (z_1,\dots,z_\dd)\in\Z_{>0}^\dd$ with~$|\mathbf{z}| = z+\dd$ arbitrary. We can write
    \begin{align*}
        (z_1,\dots,z_\dd) = (z_1'+\delta_1,\dots,z_\dd'+\delta_\dd)
    \end{align*}
    with~$\delta_1,\dots,\delta_\dd\in\Z_{\geq 0}$ and~$\mathbf{z'} = (z_1',\dots,z_\dd')\in\Z_{>0}^\dd$ with~$|\mathbf{z'}| = 2\dd$. Hence,
    \begin{align*}
        \uu_{\mathbf{z}} = \sum\limits_{(\mathbf{n},\boldsymbol{\ell})\in\indexset{\dd}{\dd}} a_{\mathbf{n},\boldsymbol{\ell}}(\mathbf{z'})\, \uu_{\mathbf{n}}\boxast \uu_{\ell_1 + \delta_1}\cdots \uu_{\ell_\dd + \delta_\dd}.
    \end{align*}
    Since~$\mathbf{z}$ was chosen arbitrary, we obtain $\VSbox{z}{\dd} = \Vbox{z}{\dd}$, proving the theorem.
\end{proof}

\begin{lemma}
\label{lem:rzleq8}
    Conjecture~\ref{conj:systemApp} is true for all~$(z,\dd)\in\Z_{>0}^2$ with~$1\leq \dd\leq 8$.
\end{lemma}
\begin{proof}
    The proof for~$1\leq z\leq\dd\leq 8$ is obtained by computer algebra; for details, see Remark~\ref{rem:calcrem2} and the appendix. By Theorem~\ref{thm:sys1}, Conjecture~\ref{conj:systemApp} is also true for~$z\geq\dd$ when~$1\leq\dd\leq 8$, proving the lemma.
\end{proof}

Note that~$\sdim{z}{\dd}$ is the dimension of the image of the~$\Q$-linear map
\begin{align*}
    \Boxmap{z}{\dd}\colon \operatorname{span}_\Q\indexset{z}{\dd}&\longrightarrow \Vbox{z}{\dd},
    \\
    (\mathbf{n},\boldsymbol{\ell})&\longmapsto u_{\mathbf{n}}\boxast u_{\boldsymbol{\ell}}
\end{align*}
that we continue~$\Q$-bilinearly. By the rank-nullity theorem, we know that
\begin{align}
\eqlabel{eq:sumdim}
    \sdim{z}{\dd} + \dim_\Q \operatorname{ker}\Boxmap{z}{\dd} = \dim_\Q \operatorname{span}_\Q\indexset{z}{\dd}.
\end{align}

The right-hand side is given by~$\idim{z}{\dd}$, which is the number of writing~$z+\dd$ as ordered sum of at least~$\dd+1$ and at most~$\dd+\min\{z,\dd\}$ positive integers, i.e.,
\begin{align}
    \eqlabel{eq:idim}
    \dim_\Q \operatorname{span}_\Q\indexset{z}{\dd} = \idim{z}{\dd} = \sum\limits_{j=1}^{\min\{z,\dd\}} \binom{z+\dd-1}{d+j-1}.
\end{align}
Hence, determining~$\sdim{z}{\dd}$ now is equivalent to determining~$\dim_\Q \operatorname{ker}\Boxmap{z}{\dd}$. While it seems to be difficult to obtain a (conjectured) basis of~$\VSbox{z}{\dd}$, we can give a conjectured basis of~$\operatorname{ker}\Boxmap{z}{\dd}$ explicitly. To do so, we need the notion of stuffle product and box product on index level. I.e., we set~$\mathbf{n}\ast\emptyset := \emptyset\ast \mathbf{n} := \mathbf{n}$, $\mathbf{n}\boxast\emptyset := \emptyset\boxast \mathbf{n} := \mathbf{n}$ for every index~$\mathbf{n}$. Furthermore, for given indices~$\mathbf{n} = (n_1,\dots,n_s)\in\Z_{>0}^s,\, \mathbf{m} = (m_1,\dots,m_t)\in\Z_{>0}^t$ with~$s,t\geq 1$, we set recursively
\begin{align*}
    \mathbf{n}\ast\mathbf{m} :=&\, (n_1).((n_2,\dots,n_s)\ast\mathbf{m}) + (m_1).(\mathbf{n}\ast (m_2,\dots,m_t)) 
    \\
    &\, + (n_1+m_1).((n_2,\dots,n_s)\ast (m_2,\dots,m_t))
\end{align*}
as formal sum of indices, where~$().()$ means the concatenation of indices. Similarly, we define the box product~$\mathbf{n}\boxast\mathbf{m}$ to be the part of~$\mathbf{n}\ast\mathbf{m}$ of smallest length.

\begin{example}
    To illustrate the definition of stuffle product and box product of indices, we consider~$\mathbf{n} = (1,2)$ and~$\mathbf{m} = (3,2)$. We have
    \begin{align*}
        \mathbf{n}\ast\mathbf{m} =&\, (1,2)\ast (3,2)
        \\
        =&\, (4,4) + (1,5,2) + (1,3,4) + 2 (4,2,2) + (3,3,2)
        \\
        &\, + (1,2,3,2) + 2 (1,3,2,2) + 2 (3,1,2,2) + (3,2,1,2)
    \end{align*}
    and
    \begin{align*}
        \mathbf{n}\boxast\mathbf{m} =\, (1,2)\boxast (3,2) =\, (4,4).
    \end{align*}
\end{example}
In the following, for~$z,\dd\in\Z_{>0}$, we consider the set
\begin{align*}
    \kernelbox{z}{\dd} := \left\{(\mathbf{n_1},\mathbf{n_2}\boxast\boldsymbol{\ell}) - (\mathbf{n_1}\ast\mathbf{n_2},\boldsymbol{\ell})\, \Bigg|\, \substack{\mathbf{n_1}\in\Z_{>0}^{s_1},\, \mathbf{n_2}\in\Z_{>0}^{s_2},\, \boldsymbol{\ell}\in\Z_{>0}^{\dd},\\ 1\leq s_1,s_2\leq\dd,\, |\mathbf{n_1}|+|\mathbf{n_2}|+\boldsymbol{\boldsymbol{\ell} = z+\dd}}\right\}\subset \operatorname{span}_\Q\indexset{z}{\dd},
\end{align*}
where~$(\cdot,\cdot)$ is~$\Q$-bilinearly continued.

\begin{lemma}
    \label{lem:kernelboxinkernel}
    For all~$z,\dd\in\Z_{>0}$, we have~$\operatorname{span}_\Q \kernelbox{z}{\dd}\subset\operatorname{ker}\Boxmap{z}{\dd}$.
\end{lemma}
\begin{proof}
    This is an immediate consequence of Lemma~\ref{lem:stuffle-boxstuffle}.
\end{proof}


By numerical calculations (see the appendix), we conjecture that the converse inclusion is also true if~$z\leq \dd$.

\begin{conjecture}
\label{conj:boxkernel1}
    Let be~$z,\dd\in\Z_{>0}$ with~$z\leq\dd$. Then,
    \begin{align}
    \eqlabel{eq:boxkernel1}
        \operatorname{span}_\Q \kernelbox{z}{\dd} = \operatorname{ker}\Boxmap{z}{\dd}.
    \end{align}
\end{conjecture}
We say that Conjecture~\ref{conj:boxkernel1} is true for~$(z_0,\dd_0)$ if~\eqref{eq:boxkernel1} is true for~$(z,\dd) = (z_0,\dd_0)$. Note the following consequence.

\begin{lemma}
    \label{lem:boxkernel1}
    Let be~$z,\dd\in\Z_{>0}$ with~$z\leq\dd$. If Conjecture~\ref{conj:boxkernel1} is true for~$(z,\dd)$, we have
    \begin{align*}
        \sdim{z}{\dd} \geq \binom{z+\dd-1}{\dd}.
    \end{align*}
    In particular, if~$z=\dd$ additionally, then Conjecture~\ref{conj:systemApp} is true for~$(\dd,\dd)$.
\end{lemma}
\begin{proof}
    Let be~$z,\dd\in\Z_{>0}$ with~$z\leq\dd$. We begin by noting that we have
    \begin{align*}
        \#\kernelbox{z}{\dd} = \sum\limits_{j=2}^z \binom{z+\dd-1}{\dd+j-1}
    \end{align*}
    since~$\#\kernelbox{z}{\dd}$ is the number of ways one can write~$z+\dd$ as ordered sum of at least~$\dd+2$ and at most~$\dd+\min\{z,\dd\}$ ($=\dd+z$ in case~$z\leq\dd$) positive integers. Now, if Conjecture~\ref{conj:boxkernel1} is true for~$(z,\dd)$, we obtain by~\eqref{eq:sumdim} and~\eqref{eq:idim}, that
    \begin{align*}
        \sdim{z}{\dd} = \idim{z}{\dd} - \dim_\Q \operatorname{ker}\Boxmap{z}{\dd} \geq \sum\limits_{j=1}^z \binom{z+\dd-1}{\dd+j-1} - \sum\limits_{j=2}^z \binom{z+\dd-1}{\dd+j-1} = \binom{z+\dd-1}{\dd}.
    \end{align*}
    In case~$z=\dd$, the right hand side is~$\tdim{\dd}{\dd}$, i.e., we must have equality and so, Conjecture~\ref{conj:systemApp} is true for~$(\dd,\dd)$. This completes the proof of the lemma.
\end{proof}

The set~$\kernelbox{z}{\dd}$ seems to be of special interest regarding determining a basis of~$\operatorname{ker}\Boxmap{z}{\dd}$ as the following refinement of Conjecture~\ref{conj:boxkernel1} shows.

\begin{conjecture}
\label{conj:boxkernel2}
    Let be~$z,\dd\in\Z_{>0}$ with~$z\leq\dd$. Then~$\kernelbox{z}{\dd}$ is a basis of~$\operatorname{ker}\Boxmap{z}{\dd}$.
\end{conjecture}
As usual, we say that Conjecture~\ref{conj:boxkernel2} is true for~$(z_0,\dd_0)$ if $\kernelbox{z_0}{\dd_0}$ is a basis of~$\operatorname{ker}\Boxmap{z_0}{\dd_0}$. We give evidence for Conjecture~\ref{conj:boxkernel2}.
\begin{lemma}
\label{lem:boxkernel2}
    Conjecture~\ref{conj:boxkernel2} is true for all~$(z,\dd)\in\Z_{>0}^2$ satisfying~$1\leq z\leq \dd\leq 8$.
\end{lemma}
\begin{proof}
    For~$z=1$ and~$\dd\in\Z_{>0}$, we have~$\kernelbox{1}{\dd} = \emptyset$ and~$\idim{z}{\dd} = \dd = \sdim{1}{\dd}$ as we will show in Lemma~\ref{lem:smallz=3}, i.e., $\operatorname{ker}\Boxmap{1}{\dd}$ is the trivial vector space. Hence, Conjecture~\ref{conj:boxkernel2} is true for all~$(1,\dd)\in\Z_{>0}^2$. For~$z\geq 2$, the claim is obtained by numerical calculations, see the appendix.
\end{proof}

Note the following consequence that Conjecture~\ref{conj:boxkernel2} is a refinement of Conjecture~\ref{conj:systemApp}.

\begin{lemma}
    \label{lem:boxkernel3}
    Let be~$z,\dd\in\Z_{>0}$ with~$z\leq \dd$. If Conjecture~\ref{conj:boxkernel2} is true for~$(z,\dd)$, then also Conjecture~\ref{conj:systemApp} is true for~$(z,\dd)$.
\end{lemma}
\begin{proof}
    Let be~$z,\dd\in\Z_{>0}$ with~$z\leq \dd$ and assume that Conjecture~\ref{conj:boxkernel2} is true for~$(z,\dd)$. By~\eqref{eq:sumdim} and~\eqref{eq:idim}, then we obtain
    \begin{align*}
        \sdim{z}{\dd} = \idim{z}{\dd} - \dim_\Q \operatorname{ker}\Boxmap{z}{\dd} = \sum\limits_{j=1}^z \binom{z+\dd-1}{\dd+j-1} - \sum\limits_{j=2}^z \binom{z+\dd-1}{\dd+j-1} = \binom{z+\dd-1}{\dd},
    \end{align*}
    i.e., Conjecture~\ref{conj:systemApp} is true for~$(z,\dd)$.
\end{proof}

We investigate~$\sdim{z}{\dd}$ in the following in more detail.

\begin{lemma}
\label{lem:dineq}
    For all~$z,\dd\in\Z_{>0}$, we have
    \begin{align*}
        \sdim{z}{\dd+1} + \sdim{z+1}{\dd} \leq \sdim{z+1}{\dd+1}.
    \end{align*}
\end{lemma}

\begin{proof}
    Fix~$z,\dd\in\Z_{>0}$. By definition of~$\sdim{z}{\dd+1}$, there are~$\sdim{z}{\dd+1}$ linearly independent linear combinations
    \begin{align*}
        \sum\limits_{(\mathbf{n},\boldsymbol{\ell})\in\indexset{z}{\dd+1}} a_{\mathbf{n},\boldsymbol{\ell}}^{(j)}(\mathbf{z})\, \uu_{\mathbf{n}}\boxast \uu_{\boldsymbol{\ell}}\qquad (1\leq j\leq \sdim{z}{\dd+1}).
    \end{align*}
    Then, the~$\sdim{z}{\dd+1}$ linear combinations ($1\leq j\leq \sdim{z}{\dd+1}$ in the following) of case~$(z+1,\dd+1)$,
     \begin{align}
        \eqlabel{eq:linr+1z}
        \sum\limits_{(\mathbf{n},\boldsymbol{\ell})\in\indexset{z}{\dd+1}} a_{\mathbf{n},\boldsymbol{\ell}}^{(j)}(\mathbf{z})\, \uu_{\mathbf{n}}\boxast \uu_{(\ell_1,\dots,\ell_{\dd},\ell_{\dd+1}+1)},
    \end{align}
    are linearly independent as well. Note that all occurring words~$\uu_{\mu_1}\cdots \uu_{\mu_{\dd+1}}$ in this linear combinations satisfy~$\mu_{\dd+1}\geq 2$.

    Now, by definition of~$\sdim{z+1}{\dd}$, there are~$\sdim{z+1}{\dd}$ linear independent linear combinations
     \begin{align}
     \eqlabel{eq:linrz+11}
        \sum\limits_{(\mathbf{n},\boldsymbol{\ell})\in\indexset{z+1}{\dd}} b_{\mathbf{n},\boldsymbol{\ell}}^{(j)}(\mathbf{z})\, \uu_{\mathbf{n}}\boxast \uu_{\boldsymbol{\ell}}\qquad (1\leq j\leq \sdim{z+1}{\dd}).
    \end{align}
    Considering for~$1\leq j\leq \sdim{z+1}{\dd}$ the following linear combinations in case~$(z+1,\dd+1)$
    \begin{align}
        \eqlabel{eq:linrz+1}
    \begin{split}
        &\sum\limits_{(\mathbf{n},\boldsymbol{\ell})\in\indexset{z+1}{\dd}} b_{\mathbf{n},\boldsymbol{\ell}}^{(j)}(\mathbf{z})\, \uu_{\mathbf{n}}\boxast \uu_{\boldsymbol{\ell}}\uu_1
        \\        =& \left(\sum\limits_{(\mathbf{n},\boldsymbol{\ell})\in\indexset{z+1}{\dd}} b_{\mathbf{n},\boldsymbol{\ell}}^{(j)}(\mathbf{z})\, \uu_{\mathbf{n}}\boxast \uu_{\boldsymbol{\ell}}\right)\uu_1 
        + \sum\limits_{(\mathbf{n},\boldsymbol{\ell})\in\indexset{z+1}{\dd}} b_{\mathbf{n},\boldsymbol{\ell}}^{(j)}(\mathbf{z})\, \left(\uu_{(n_1,\dots,n_{s-1})}\boxast \uu_{\boldsymbol{\ell}}\right)\uu_{1+n_s}
    \end{split}
    \end{align}
    are linearly independent again because of \eqref{eq:linrz+11}. Furthermore, they and the ones from \eqref{eq:linr+1z} are linearly independent since the latter ones contain words ending in~$\uu_{\mu_{\dd+1}}$ with~$\mu_{\dd+1}\geq 2$ while the linear independence of \eqref{eq:linrz+1} already comes from words ending all in~$\uu_1$.

    Summarized, we have proven 
  ~$
        \sdim{z}{\dd+1} + \sdim{z+1}{\dd} \leq \sdim{z+1}{\dd+1}.
  ~$
\end{proof}

\begin{remark}
    Assuming Conjecture~\ref{conj:systemApp}, the inequality in Lemma~\ref{lem:dineq} is an equality if and only if~$z\neq \dd$.
\end{remark}

With Lemma~\ref{lem:dineq}, we can now prove the following partial result towards Conjecture~\ref{conj:systemApp}.

\begin{lemma}
\label{lem:smallz=3}
    Conjecture~\ref{conj:systemApp} is true for all pairs~$(z,\dd)\in\Z_{>0}^2$ with~$1\leq z\leq 3$.
\end{lemma}
\begin{proof}
    Note that the proof for~$1\leq z\leq 2$ is contained in Remark~\ref{rem:small}. Therefore, assume~$z=3$ in the following. For~$(z,\dd)\in\{(3,1),(3,2),(3,3)\}$, the claim follows from Remark~\ref{rem:small}. Hence, consider~$\dd\geq 4$ and prove by induction (with already proven base case~$\dd =3$) on~$\dd$. By Lemma~\ref{lem:dineq}, the induction hypothesis, and the case~$z=2$ of the lemma that is proven in Remark~\ref{rem:small}, we know that
    \begin{align}
    \eqlabel{eq:dr3}
        \sdim{3}{\dd} \geq \sdim{3}{\dd-1} + \sdim{2}{\dd} = \binom{\dd+1}{2} + \binom{\dd+1}{1} = \binom{\dd+2}{2}.
    \end{align}
    Therefore, it suffices to prove~$\sdim{3}{\dd}\leq \binom{\dd+2}{2}$. Note that for~$(z,\dd) = (3,\dd)$ the number of box products spanning~$\VSbox{3}{\dd}$ is~$\binom{\dd+2}{0} + \binom{\dd+2}{1} + \binom{\dd+2}{2}$. I.e., if we can show that~$\binom{\dd+2}{2}$ of those are such that the other~$\binom{\dd+2}{0} + \binom{\dd+2}{1}$ ones are in their~$\Q$-span, we are done. We consider the set of~$\binom{\dd+2}{2}$ box products
    \begin{align*}
        \mathcal{R}_{3,\dd} := \left\{\substack{\uu_2\uu_1\boxast \uu_1^\dd,\ \uu_1\uu_2\boxast \uu_1^\dd,\,\\ \uu_2 \boxast \uu_1^{j_1}\uu_2\uu_1^{\dd-j_1-1},\, \uu_1\boxast \uu_1^{j_2}\uu_3\uu_1^{\dd-j_2-1},\\ \uu_1\boxast \uu_1^{j_3}\uu_2\uu_1^{j_4}\uu_2\uu_1^{\dd-j_3-j_4-2}}\Bigg| \substack{0\leq j_1\leq \dd - 2,\, 0\leq j_2\leq \dd - 2,\\ 0\leq j_3,j_4\leq \dd - 2,\, j_3+j_4\leq \dd - 2}\right\}.
    \end{align*}
    In the following, we show that the other box products in case~$(z,\dd)=(3,\dd)$ are in the~$\Q$-span of~$\mathcal{R}_{3,\dd}$. For~$0\leq j_1\leq \dd - 2$, we obtain
    \begin{align}
    \eqlabel{eq:r3:11}
        \uu_1 \uu_1 \boxast \uu_1^{j_1}\uu_2\uu_1^{\dd-j_1-1} = \frac12\left((\uu_1\ast \uu_1 - \uu_2) \boxast \uu_1^{j_1}\uu_2\uu_1^{\dd-j_1-1}\right)\in \operatorname{span}_\Q \mathcal{R}_{3,\dd}
    \end{align}
    due to Lemma~\ref{lem:stuffle-boxstuffle} and the definition of~$\mathcal{R}_{3,\dd}$.
    Furthermore, we have that
    \begin{align}
    \eqlabel{eq:r3:3}
        \uu_3 \boxast \uu_1^\dd = \sum\limits_{j_2 = 0}^{\dd-1} \uu_1\boxast \uu_1^{j_2}\uu_3\uu_1^{\dd-j_2-1} - (\uu_2\uu_1 + \uu_1\uu_2)\boxast \uu_1^\dd
    \end{align}
    is in the~$\Q$-span of~$\mathcal{R}_{3,\dd}$. This implies, due to~$\uu_2\ast \uu_1 = \uu_2\uu_1 + \uu_1\uu_2 + \uu_3$ and Lemma~\ref{lem:stuffle-boxstuffle}, that
    \begin{align*}
        \uu_2\boxast \uu_1^{\dd-1} \uu_2 =\, & \uu_2\boxast \left(\uu_1\boxast \uu_1^\dd - \sum\limits_{j_1=0}^{\dd - 2} \uu_1^{j_1}\uu_2\uu_1^{\dd-j_1-1}\right)
        \\
        =\, & (\uu_2\uu_1 + \uu_1\uu_2 + \uu_3)\boxast \uu_1^\dd - \sum\limits_{j_1=0}^{\dd - 2} \uu_2\boxast \uu_1^{j_1}\uu_2\uu_1^{\dd-j_1-1}\in \operatorname{span}_\Q \mathcal{R}_{3,\dd}.
    \end{align*}
    Similar to \eqref{eq:r3:11}, one obtains now
    \begin{align*}
        \uu_2\boxast \uu_1^{\dd-1} \uu_2\in \operatorname{span}_\Q \mathcal{R}_{3,\dd}.
    \end{align*}
    Using \eqref{eq:r3:3}, Lemma~\ref{lem:stuffle-boxstuffle} and the definition of~$\mathcal{R}_{3,\dd}$, we get
    \begin{align*}
        \uu_1\uu_1\uu_1\boxast \uu_1^\dd = \frac{1}{3} \left((\uu_1\ast \uu_1\uu_1 - \uu_2\uu_1 - \uu_1\uu_2)\boxast \uu_1^\dd \right)\in  \operatorname{span}_\Q \mathcal{R}_{3,\dd},
    \end{align*}
    completing the claim. In particular, the lemma is proven for~$z=3$.
\end{proof}

\begin{proposition}
\label{prop:zgeqd4}
    Conjecture~\ref{conj:systemApp} is true for~$(4,4)$ and therefore, by Theorem~\ref{thm:sys1}, for all pairs~$(z,4)$ with~$z\geq 4$.
\end{proposition}
\begin{proof}
    Using Corollary~\ref{cor:systemApp}, we have to show~$\VSbox{4}{4} = \Vbox{4}{4}$. From Theorem~\ref{thm:boxexplicit} and Lemma~\ref{lem:r+1}, we already have
    \begin{align*}
        &u_2u_2u_2u_2,\, u_5u_1u_1u_1,\, u_1u_5u_1u_1,\, u_1u_1u_5u_1,\, u_1u_1u_1u_5,\, u_1u_3u_2u_2,\, u_1u_2u_3u_2,\\ &u_1u_2u_2u_3,\, u_3u_2u_2u_1,\, u_2u_3u_2u_1,\, u_2u_2u_3u_1,\, u_2u_4u_1u_1,\, u_3u_3u_1u_1,\, u_4u_2u_1u_1,\\ &u_1u_1u_2u_4,\, u_1u_1u_3u_3,\, u_1u_1u_4u_2,\, u_2u_1u_3u_2,\, u_2u_1u_2u_3,\, u_2u_3u_1u_2,\, u_3u_2u_1u_2\in\VSbox{4}{4}.
    \end{align*}
    Hence, considering~$u_1\boxast u_2u_1u_2u_2$, we obtain~$u_3u_1u_2u_2\in\VSbox{4}{4}$, and so, by Corollary~\ref{cor:reverse},we also have~$u_2u_2u_1u_3\in\VSbox{4}{4}$. Now, considering~$u_1u_1\boxast u_1{j_1}u_2u_1^{j_2}u_2u_1^{j_3}$ for~$j_1,j_2,j_3\in\Z_{\geq 0}$ with~$j_1+j_2+j_3 = 2$, yields~$u_3u_1u_3u_1,\, u_3u_1u_1u_3,\, u_1u_3u_3u_1,\, u_1u_3u_1u_3\in\VSbox{4}{4}$. Last, consider~$u_1\boxast u_1^{j_1}u_2u_1^{j_2}u_3u_1^{j_3}$ for~$j_1,j_2,j_3\in\Z_{\geq 0}$ with~$j_1+j_2+j_3=1$ immediately gives~$u_2u_1u_4u_1,\, u_2u_1u_1u_4,\, u_1u_2u_4u_1,\, u_1u_2u_1u_4\in\VSbox{4}{4}$, yielding, by Corollary~\ref{cor:reverse} again, that~$u_4u_1u_2u_1,\, u_4u_1u_1u_2,\, u_1u_4u_2u_1,\, u_1u_4u_1u_2\in\VSbox{4}{4}$. Therefore,~$\VSbox{4}{4} = \Vbox{4}{4}$ follows, completing the proof.
\end{proof}

\subsection{Abuot a possible basis of the spaces~$\VSbox{z}{\dd}$}

After discussing the spaces~$\VSbox{z}{\dd}$ and its dimension, we give in this subsection a possible basis of~$\VSbox{z}{\dd}$ for~$1\leq z\leq \dd$. The results were obtained using numerical calculations, a proof for Conjecture~\ref{conj:Szdbasis} is postponed to future works.

\begin{definition}
    Given~$1\leq z\leq\dd$. We define the set~$\plainbasic{z}{\dd}$ of \textit{plain basis vectors} recursively as follows:
    \begin{enumerate}
        \item For~$z = 1$, we set~$\plainbasic{0}{\dd} = \{u_1^\dd\}$.
        \item For~$z>1$, we define 
        \begin{align*}
            \plainbasic{z}{\dd} := \{u_\ell\word\colon \word\in\plainbasic{z-\delta_{z=\dd}}{\dd-1},\, 1\leq \ell\leq z+\dd-\wt(\word)\}.
        \end{align*}
    \end{enumerate}
\end{definition}

\begin{remark}
    More explicitly,~$\plainbasic{z}{\dd}$ is the set of~$u_{\ell_1}\cdots u_{\ell_\dd}$ with~$\ell_1,\dots,\ell_\dd\in\Z_{>0}$ such that
    \begin{align*}
        \sum\limits_{k=j}^\dd \ell_k \leq \min\{z,\dd-j+1\} + \dd - j
    \end{align*}
    for all~$1\leq j\leq\dd$.
\end{remark}

\begin{conjecture}
\label{conj:Szdbasis}
    Let be~$1\leq z\leq\dd$. The set
    \begin{align*}
        \basis{z}{\dd} := \left\{u_{n_1}\cdots u_{n_s}\boxast u_{\ell_1}\cdots u_{\ell_\dd}\colon \substack{u_{\ell_1}\cdots u_{\ell_\dd}\in\plainbasic{z}{\dd},\ n_j\in\Z_{>0}\ (1\leq j\leq s),\\ n_1+\cdots+n_s+\ell_1+\cdots+\ell_{\dd} = z + \dd,\, \ell_1 > z-\dd+\delta_{n_1=s=1}}\right\}
    \end{align*}
    is a basis of~$\VSbox{z}{\dd}$.
\end{conjecture}

\begin{example}
    Consider~$\dd = 4$. Then,
    \begin{align*}
        \plainbasic{1}{4} =& \{u_1u_1u_1u_1\},\\
        \plainbasic{2}{4} =& \{u_1u_1u_1u_1,\, u_1u_1u_2u_1,\, u_1u_2u_1u_1,\, u_2u_1u_1u_1\},\\
        \plainbasic{3}{4} =& \left\{u_1u_1u_1u_1,\, u_1u_1u_2u_1,\, u_1u_2u_1u_1,\, u_2u_1u_1u_1,\right.\\
        &\quad \left. u_1u_2u_2u_1,\, u_2u_1u_2u_1,\, u_1u_3u_1u_1,\, u_2u_2u_1u_1,\, u_3u_1u_1u_1\right\},\\
        \plainbasic{4}{4} =& \left\{u_1u_1u_1u_1,\ u_1u_1u_2u_1,\ u_1u_2u_1u_1,\ u_2u_1u_1u_1,\ u_1u_2u_2u_1,\ u_1u_3u_1u_1,\ u_2u_1u_2u_1,\right.\\
        &\quad \left. u_2u_2u_1u_1,\ u_3u_1u_1u_1,\ u_2u_2u_2u_1,\ u_2u_3u_1u_1,\ u_3u_1u_2u_1,\ u_3u_2u_1u_1,\ u_4u_1u_1u_1\right\}.
    \end{align*}
    Hence, conjecturally we have that
    \begin{align*}
        \basis{1}{4} =& \{u_1\boxast\ u_1u_1u_1u_1\},\\
        \basis{2}{4} =& \left\{u_1\boxast u_1u_1u_2u_1, \ u_1\boxast u_1u_2u_1u_1, \ u_1\boxast u_2u_1u_1u_1, \ u_2\boxast u_1u_1u_1u_1, \ \right.\\ &\quad \left. u_1u_1\boxast u_1u_1u_1u_1\right\},\\
        \basis{3}{4} =& \left\{u_1\boxast u_1u_2u_2u_1, \ u_1\boxast u_1u_3u_1u_1, \ u_1\boxast u_2u_1u_2u_1, \ u_1\boxast u_2u_2u_1u_1, \ \right.\\ &\quad \left. u_1\boxast u_3u_1u_1u_1, \ u_2\boxast u_1u_1u_2u_1, \ u_2\boxast u_1u_2u_1u_1, \ u_2\boxast u_2u_1u_1u_1, \ \right.\\ &\quad \left. u_1u_1\boxast u_1u_1u_2u_1, \ u_1u_1\boxast u_1u_2u_1u_1, \ u_1u_1\boxast u_2u_1u_1u_1, \ u_3\boxast u_1u_1u_1u_1, \ \right.\\ &\quad \left. u_1u_2\boxast u_1u_1u_1u_1, \ u_2u_1\boxast u_1u_1u_1u_1, \ u_1u_1u_1\boxast u_1u_1u_1u_1\right\},\\
        \basis{4}{4} =& \left\{u_1\boxast u_2u_2u_2u_1, \ u_1\boxast u_2u_3u_1u_1, \ u_1\boxast u_3u_1u_2u_1, \ u_1\boxast u_3u_2u_1u_1, \ \right.\\ &\quad \left. u_1\boxast u_4u_1u_1u_1, \ u_2\boxast u_1u_2u_2u_1, \ u_2\boxast u_1u_3u_1u_1, \ u_2\boxast u_2u_1u_2u_1, \ \right.\\ &\quad \left. u_2\boxast u_2u_2u_1u_1, \ u_2\boxast u_3u_1u_1u_1, \ u_1u_1\boxast u_1u_2u_2u_1, \ u_1u_1\boxast u_1u_3u_1u_1, \ \right.\\ &\quad \left. u_1u_1\boxast u_2u_1u_2u_1, \ u_1u_1\boxast u_2u_2u_1u_1, \ u_1u_1\boxast u_3u_1u_1u_1, \ u_3\boxast u_1u_1u_2u_1, \ \right.\\ &\quad \left. u_3\boxast u_1u_2u_1u_1, \ u_3\boxast u_2u_1u_1u_1, \ u_1u_2\boxast u_1u_1u_2u_1, \ u_1u_2\boxast u_1u_2u_1u_1, \ \right.\\ &\quad \left. u_1u_2\boxast u_2u_1u_1u_1, \ u_2u_1\boxast u_1u_1u_2u_1, \ u_2u_1\boxast u_1u_2u_1u_1, \ u_2u_1\boxast u_2u_1u_1u_1, \ \right.\\ &\quad \left. u_1u_1u_1\boxast u_1u_1u_2u_1, \ u_1u_1u_1\boxast u_1u_2u_1u_1, \ u_1u_1u_1\boxast u_2u_1u_1u_1, \ u_4\boxast u_1u_1u_1u_1, \ \right.\\ &\quad \left. u_1u_3\boxast u_1u_1u_1u_1, \ u_2u_2\boxast u_1u_1u_1u_1, \ u_3u_1\boxast u_1u_1u_1u_1, \ u_1u_1u_2\boxast u_1u_1u_1u_1, \ \right.\\ &\quad \left. u_1u_2u_1\boxast u_1u_1u_1u_1, \ u_2u_1u_1\boxast u_1u_1u_1u_1, \ u_1u_1u_1u_1\boxast u_1u_1u_1u_1\right\}
    \end{align*}
    build a basis of~$\VSbox{1}{4},\, \VSbox{2}{4},\, \VSbox{3}{4}$, and~$\VSbox{4}{4}$, respectively.
\end{example}

\begin{remark}
    \begin{enumerate}
        \item For proving Conjecture~\ref{conj:Szdbasis}, for the case of~$z=\dd$ in particular, it is sufficient to show linear independence of~$\basis{z}{\dd}$. 
    \item Conjecture~\ref{conj:Szdbasis} is numerically verified for all~$1\leq z\leq\dd = 9$.
    \end{enumerate}
\end{remark}

\subsection{Connection between the box product and the stuffle product}
\label{ssec:mdbd:connectionbox}

First, to connect the box product with the stuffle product, we introduce the maps~$\Psi_{\mathbf{k}}$.

\begin{definition}
\label{def:PsikApp}
    Fix~$\dd\in\Z_{>0}$ and~$\mathbf{k} = (k_1,\dots,k_\dd)\in\Z_{>0}^\dd$. We define the~$\Q$-linear map~$\Psi_{\mathbf{k}}\colon \spanQ{\word\in\mathcal{U}^{\ast,\circ}\,|\, \len(\word) = \dd}\to\QB^\circ$, given on generators by
    \begin{align*}
        \uu_{\mu_1}\cdots \uu_{\mu_\dd} &\longmapsto u_{\mu_1}u_0^{k_\dd-1}\cdots u_{\mu_\dd}u_0^{k_1-1}.
    \end{align*}
\end{definition}\noproof{definition}

Note the following connection of maps~$\Psi_{\mathbf{k}}$ with the box product.
\begin{lemma}
    \label{lem:prepmainideaApp}
    Let be~$z,\dd,w\in\Z_{>0}$ and~$(\mathbf{n},\boldsymbol{\ell})\in\indexset{z}{\dd}$. Furthermore, let be~$\mathbf{k}\in\Z_{>0}^\dd$ satisfying~$|\mathbf{k}| = w-z$. Then,
    \begin{align*}
        u_{\mathbf{n}}\boxast\Psi_{\mathbf{k}}(u_{\boldsymbol{\ell}}) = \Psi_{\mathbf{k}}(u_{\mathbf{n}}\boxast u_{\boldsymbol{\ell}}).
    \end{align*}
\end{lemma}
\begin{proof}
    Using the notation as in the lemma, we note that particularly~$\dep(u_{\boldsymbol{\ell}}) = \dd$. The claim immediately follows by the definition of the box product and the definition of the map~$\Psi_{\mathbf{k}}$.
\end{proof}

The following Lemma~\ref{lem:mainideaboxApp} now connects the stuffle product with the box product. It will be the key for proving Theorem~\ref{thm:concl} below and one of the main observations for our approach to the refined Bachmann Conjecture~\ref{conj:mdbdstrongApp}.

\begin{lemma}
\label{lem:mainideaboxApp}
    Let be~$z,\dd,w\in\Z_{>0}$ and~$(\mathbf{n},\boldsymbol{\ell})\in\indexset{z}{\dd}$. Furthermore, let be~$\mathbf{k}\in\Z_{>0}^\dd$ satisfying~$|\mathbf{k}| = w-z$. Then,
    \begin{align*}
        \fz{\Psi_{\mathbf{k}}(u_{\mathbf{n}}\boxast u_{\boldsymbol{\ell}})} 
        \in\sum\limits_{1\leq s\leq \min\{z,\dd\}} \fil{Z,D,W}{z-s,\dd+s,w}\Zq.
    \end{align*}
\end{lemma}
\begin{proof}
    Let be~$z,\dd,w\in\Z_{>0},\, (\mathbf{n},\boldsymbol{\ell})\in\indexset{z}{\dd},\, \mathbf{k}\in\Z_{>0}^\dd$ such that~$|\mathbf{k}| = w - z$ and write~$s'$ for the length of~$\mathbf{n}$. I.e., we have,~$u_{\mathbf{n}}\in\fil{Z,D,W}{0,s',|\mathbf{n}|}\QB^\circ$ and~$\Psi_{\mathbf{k}}(u_{\boldsymbol{\ell}})\in\fil{Z,D,W}{|\mathbf{k}|-\dd,\dd,|\mathbf{k}|+|\boldsymbol{\ell}|-\dd}\QB^\circ$. Since~$(\mathbf{n},\boldsymbol{\ell})\in\indexset{z}{\dd}$, we have~$|\mathbf{n}| + |\boldsymbol{\ell}| = z + \dd$. Therefore,~\eqref{eq:basicstuffle1} implies, together with the assumption~$|\mathbf{k}| = w-z$,  that
    \begin{align*}
        u_{\mathbf{n}}\ast \Psi_{\mathbf{k}}(u_{\boldsymbol{\ell}})\in \fil{Z,D,W}{w-\dd-z,\dd+s',w}\QB^\circ.
    \end{align*}
    By~\eqref{eq:basicstuffle2}, this implies now
    \begin{align}
        \tau(u_{\mathbf{n}}\ast \Psi_{\mathbf{k}}(u_{\boldsymbol{\ell}}))\in \fil{Z,D,W}{z-s',\dd+s',w}\QB^\circ,
    \end{align}
    yielding, since~$1\leq s'\leq\min\{z,\dd\}$,
    \begin{align}
    \eqlabel{eq:boxstufflemain}
        \fz{u_{\mathbf{n}}\ast \Psi_{\mathbf{k}}(u_{\boldsymbol{\ell}})} = \fz{\tau(u_{\mathbf{n}}\ast \Psi_{\mathbf{k}}(u_{\boldsymbol{\ell}}))}\in \sum\limits_{1\leq s\leq \min\{z,\dd\}} \fil{Z,D,W}{z-s,\dd+s,w}\Zq.
    \end{align}
    Furthermore, due to Corollary~\ref{cor:stufflemaxzero}, we also have
    \begin{align*}
        \fz{u_{\mathbf{n}}\boxast \Psi_{\mathbf{k}}(u_{\boldsymbol{\ell}})}\in \sum\limits_{1\leq s\leq \min\{z,\dd\}} \fil{Z,D,W}{z-s,\dd+s,w}\Zq.
    \end{align*}
    Hence, the lemma follows now from Lemma~\ref{lem:prepmainideaApp}.
\end{proof}

\begin{corollary}
\label{cor:mainideaboxApp}
    Let be~$z,\dd,w\in\Z_{>0}$ and~$\boldsymbol{\mu}\in\Z_{>0}^\dd$ satisfying~$|\boldsymbol{\mu}| = z + \dd$. If~$u_{\boldsymbol{\mu}}\in\CP$ with~$\CP$ from~\eqref{eq:boxprodQspanApp}, then
    \begin{align*}
        \fz{\Psi_{\mathbf{k}}(u_{\boldsymbol{\mu}})}\in\sum\limits_{1\leq s\leq \min\{z,\dd\}} \fil{Z,D,W}{z-s,\dd+s,w}\Zq \subset\F{z,d,w}
    \end{align*}
    for all~$\mathbf{k}\in\Z_{>0}^\dd$ satisfying~$|\mathbf{k}| = w-z$.
\end{corollary}
\begin{proof}
    Let be~$z,\dd,w\in\Z_{>0}$ and~$\boldsymbol{\mu}\in\Z_{>0}^\dd$ satisfying~$|\boldsymbol{\mu}| = z + \dd$. Furthermore, choose an index~$\mathbf{k}\in\Z_{>0}^\dd$ arbitrary with the property~$|\mathbf{k}| = w-z$. Assume~$u_{\boldsymbol{\mu}}\in\CP$, i.e., we have
    \begin{align}
    \eqlabel{eq:maybemain0}
        \uu_{\boldsymbol{\mu}} = \sum\limits_{(\mathbf{n},\boldsymbol{\ell})\in\indexset{z}{\dd}} a_{\mathbf{n},\boldsymbol{\ell}}\, \uu_{\mathbf{n}}\boxast \uu_{\boldsymbol{\ell}}
    \end{align}
    with~$a_{\mathbf{n},\boldsymbol{\ell}}\in\Q$ appropriate. Now, for all~$(\mathbf{n},\boldsymbol{\ell})\in\indexset{z}{\dd}$, by Lemma~\ref{lem:mainideaboxApp}, we have~
    \begin{align*} 
        \fz{\Psi_{\mathbf{k}}(\uu_{\mathbf{n}}\boxast \uu_{\boldsymbol{\ell}})}\in\sum\limits_{1\leq s\leq \min\{z,\dd\}} \fil{Z,D,W}{z-s,\dd+s,w}\Zq.
    \end{align*}
    I.e., by~$\Q$-linearity of~$\zeta_q^{\mathrm{f}}$ and~$\Psi_{\mathbf{k}}$, hence we obtain
    \begin{align*}
        \fz{\Psi_{\mathbf{k}}(u_{\boldsymbol{\mu}})}
        =
        \sum\limits_{(\mathbf{n},\boldsymbol{\ell})\in\indexset{z}{\dd}} a_{\mathbf{n},\boldsymbol{\ell}}\, \fz{\Psi_{\mathbf{k}}(\uu_{\mathbf{n}}\boxast \uu_{\boldsymbol{\ell}})}\in \sum\limits_{1\leq s\leq \min\{z,\dd\}} \fil{Z,D,W}{z-s,\dd+s,w}\Zq,
    \end{align*}
    completing the claim.
\end{proof}

\subsection{Supplementary results and calculations regarding the box product}
\label{ssec:mdbd:supplementary}

We collect in this subsection further results towards the box product that are connected to Conjecture~\ref{conj:systemApp} but not needed in the following. First, we refine Conjecture~\ref{conj:systemApp}. For this, we define for all~$z,\dd,s_{\min}\in\Z_{>0}$ with~$1\leq z\leq \dd$,
    \begin{align}
        \VSbox{z}{\dd,s_{\min}} :=&\, \spanQ{\uu_{\mathbf{n}}\boxast \uu_{\boldsymbol{\ell}}\mid (\mathbf{n},\boldsymbol{\ell})\in\indexset{z}{\dd},\, \len(\mathbf{n})\geq s_{\min}}\subset\VSbox{z}{\dd},
        \\
        \sdim{z}{\dd,s_{\min}} :=&\, \dim_\Q \VSbox{z}{\dd,s_{\min}}.
    \end{align}

\begin{conjecture}
\label{conj:conjrefinement}
    For all~$z,\dd,s_{\min}\in\Z_{>0}$ with~$1\leq z\leq \dd$, we have
    \begin{align}
    \eqlabel{eq:systemsminApp}
        \sdim{z}{\dd,s_{\min}} = \binom{z+\dd-1}{z - s_{\min}}.
    \end{align}
\end{conjecture}\noproof{conjecture}
Given~$(z_0,\dd_0,s_{\min,0})\in\Z_{>0}$ with~$1\leq z\leq \dd$, we say that Conjecture~\ref{conj:conjrefinement} \emph{is true for}~$(z_0,\dd_0,s_{\min,0})$ if~\eqref{eq:systemsminApp} is true for~$(z,\dd,s_{\min}) = (z_0,\dd_0,s_{\min,0})$.

\begin{remark}
\label{rem:boxconj}
    With Theorem~\ref{thm:sys1}, we see that if Conjecture~\ref{conj:systemApp} is true for~$z= \dd$, then the statement for~$z>\dd$ follows as well. Hence, we can view Conjecture~\ref{conj:conjrefinement} (via~$s_{\min} = 1$) indeed as a refinement of Conjecture~\ref{conj:systemApp}, despite it is a refinement for~$z\leq\dd$ only.
\end{remark}\noproof{remark}

\begin{remark} 
\label{rem:calcrem2}
    Conjecture~\ref{conj:conjrefinement} is true for all triples~$(z,\dd,s_{\min})\in\Z_{>0}^3$ with~$1\leq z\leq \dd\leq 8$ and~$1\leq s_{\min}\leq 8$. The proof is obtained by computer algebra; for details, see the appendix. One could use the code in the appendix for verifying Conjecture~\ref{conj:conjrefinement} also for larger values of~$z$ and~$d$. The only limit is the computing capacity and time since the code is based on computing ranks of matrices that grow exponentially in~$z$ and~$d$.
\end{remark}\noproof{remark}


In the next remark, we give an elementary proof, not based on numerical calculations, for the part of Lemma~\ref{lem:rzleq8} that is needed for proving our main results of this paper.
\begin{remark}
\label{rem:small}
    We could verify Conjecture~\ref{conj:systemApp} for all pairs~$(z,\dd)\in\Z_{>0}^2$ with~$1\leq\dd\leq 3$ also without numerical calculations. For this, first, assume~$\dd =1$ and fix~$z\in\Z_{>0}$. Note that~$\Vbox{z}{1} = \spanQ{\uu_{z+1}}$, yielding~$\sdim{z}{1}\leq\dim_\Q\Vbox{z}{1} = 1$. Furthermore, 
    \begin{align*}
        \uu_{z+1} = \uu_1\boxast \uu_z \in\VSbox{z}{1},
    \end{align*}
    giving~$\sdim{z}{1} \geq 1$. Hence, Conjecture~\ref{conj:systemApp} is true for all pairs~$(z,1)\in\Z_{>0}^2$ since
    \begin{align*}
        \sdim{z}{1} = 1 = \binom{1+z-1}{\min\{z,1\}-1}.
    \end{align*}
    Now, assume~$z=1$ and fix~$\dd\in\Z_{>0}$. In this case,~$\VSbox{1}{\dd}=\spanQ{\uu_1\boxast \uu_1^\dd}$, i.e.,~$\sdim{1}{\dd} = 1$. In particular, we have proven Conjecture~\ref{conj:systemApp} for~$z=1$ since
    \begin{align*}
        \sdim{1}{\dd} = 1 = \binom{\dd+1-1}{\min\{1,\dd\}-1}.
    \end{align*}
    Next, assume~$\dd =2$ and fix~$z\in\Z_{\geq 2}$. Note that the case~$(z,\dd) = (1,2)$ follows from the~$z=1$-case we have proven. Note that~$\Vbox{z}{2} = \spanQ{\uu_a\uu_{z+2-a}\mid 1\leq a\leq z+1}$. A direct calculation shows
    \begin{align*}
        \uu_a \uu_{z+2-a} = \begin{cases}
            \uu_1\uu_1\boxast \uu_{a-1} \uu_{z+1-a},&\quad\text{if } 2\leq a\leq z,\\
            \uu_1\boxast \uu_1\uu_{z} - \uu_1\uu_1\boxast \uu_1 \uu_{z-1},&\quad\text{if } a = 1,\\
            \uu_1\boxast \uu_{z}\uu_1 - \uu_1\uu_1\boxast \uu_{z-1}\uu_1,&\quad\text{if } a = z+1.
        \end{cases}
    \end{align*}
    Hence,~$\uu_a \uu_{z+2-a}\in\VSbox{z}{2}$ for all~$1\leq a\leq z+1$, i.e.,~$\VSbox{z}{2}=\Vbox{z}{2}$, giving
    \begin{align*}
        \sdim{z}{2} = \dim_\Q\Vbox{z}{2} = \binom{2+z-1}{\min\{z,2\} - 1}
    \end{align*}
    since we assumed~$z\geq 2 = \dd$. Hence, Conjecture~\ref{conj:systemApp} is true for all pairs~$(z,2)\in\Z_{>0}^2$.

    Now, assume~$z=2$ and fix~$\dd\geq 2$ (since the~$(z,\dd) = (2,1)$-case follows from the~$\dd =1$-case of the theorem). In this case,~$\VSbox{2}{\dd}$ is spanned by the~$\dd+2$ box products
    \begin{align*}
        \uu_1\boxast \uu_1^j\uu_2\uu_1^{\dd-j-1}\quad (0\leq j\leq \dd-1),\quad \uu_1\uu_1\boxast \uu_1^\dd,\quad \uu_2\boxast \uu_1^\dd.
    \end{align*}
    Note that all but the last box product are linear independent since~$\uu_1\uu_1\boxast \uu_1^\dd$ does not contain any word with letter~$\uu_3$ while~$\uu_1\boxast \uu_1^j\uu_2\uu_1^{\dd-j-1}$ does contain exactly one such one which is unique for fixed~$j$. Furthermore, we have
    \begin{align*}
        \uu_2\boxast \uu_1^\dd = \sum\limits_{j=0}^{\dd-1} \uu_1\boxast \uu_1^j\uu_2\uu_1^{\dd-j-1} - 2 \uu_1\uu_1\boxast \uu_1^\dd,
    \end{align*}
    i.e.,~$\uu_2\boxast \uu_1^\dd$ is not linearly independent of the box products. Therefore,
    \begin{align*}
        \sdim{2}{\dd} = \dd+1 = \binom{\dd+2-1}{\min\{2,\dd\} - 1}
    \end{align*}
    since we assumed~$\dd\geq 2$. This proves Conjecture~\ref{conj:systemApp} for~$z=2$.

    Now, assume~$\dd =3$. Since the cases~$(z,\dd)\in\{(1,3),(2,3)\}$ follow from the case~$z=1$, respectively~$z=2$, that we have proven already, we may fix~$z\in\Z_{\geq 3}$. For~$z=3$, from Lemmas~\ref{lem:r+1},~\ref{lem:123}, and~\ref{lem:21}, we obtain~$\VSbox{3}{3} = \Vbox{3}{3}$, yielding, by Corollary~\ref{cor:systemApp}, the claim. For~$z>3$, we apply Theorem~\ref{thm:sys1} to obtain the remaining part for the proof that Conjecture~\ref{conj:systemApp} is true for all pairs~$(z,3)\in\Z_{>0}^2$ from the case~$z=3$.
\end{remark}\noproof{remark}

Noting Corollary~\ref{cor:systemApp}, Conjecture~\ref{conj:systemApp} is equivalent to~$\VSbox{z}{\dd} = \Vbox{z}{\dd}$ for all~$z\geq\dd$. I.e., in these cases, every $u_{\boldsymbol{\mu}}$ with~$\boldsymbol{\mu}\in\Z_{>0}^\dd$ and~$|\boldsymbol{\mu}| = z+\dd$ conjecturally can be written as~$\Q$-linear combination of box products~$u_{\mathbf{n}}\boxast u_{\boldsymbol{\ell}}$ with~$(\mathbf{n},\boldsymbol{\ell})\in\indexset{z}{\dd}$. With the following lemmas, we reduce the number of such~$\boldsymbol{\mu}$'s. For that, we have to show this, which can be seen as progress towards Conjecture~\ref{conj:systemApp}. For this, given~$\word_1,\word_2\in\left(\mathcal{U}\backslash\{u_0\}\right)^\ast$, we call the box product~$\word_1\boxast \word_2$ \emph{non-trivial} if~$1\leq\len(\word_1)\leq\len(\word_2)$.

\begin{lemma}
\label{lem:red1}
    Let be~$\boldsymbol{\mu}\in\Z_{>0}^{\dd}$ for some~$\dd\geq 1$. Then,~$\uu_{\boldsymbol{\mu}}$ can be written as a linear combination of words ending in~$\uu_1$ and non-trivial box products.
\end{lemma}

\begin{proof}
    Choose~$\boldsymbol{\mu} = (\mu_1,\dots,\mu_{\dd})\in\Z_{>0}^{\dd}$ with~$\mu_{\dd} > 1$ (for~$\mu_{\dd} = 1$ there is nothing to prove). Then, 
    \begin{align}
        \uu_{\mu_\dd-1}\boxast \uu_{\mu_1}\cdots \uu_{\mu_{\dd-1}} \uu_1
        = \uu_{\boldsymbol{\mu}} + \left(\uu_{\mu_\dd-1}\boxast \uu_{\mu_1}\cdots \uu_{\mu_{\dd-1}}\right) \uu_1,
    \end{align}
    i.e., after rearranging, one obtains the claim.
\end{proof}

\begin{lemma}
\label{lem:red2}
    Fix~$z,\dd\in\Z_{>1}$ with~$z\geq \dd\geq 2$. If Conjecture~\ref{conj:systemApp} is true for~$(z,\dd-1)$, then every~$\uu_{\boldsymbol{\mu}}$ with~$\boldsymbol{\mu}\in\Z_{>0}^\dd$ and~$|\boldsymbol{\mu}| = z+\dd$ can be written as linear combination of words ending in~$\uu_2$ and non-trivial box products.
\end{lemma}

\begin{proof}
    Assume~$\dd$ and~$z$ as in the lemma. Let be~$\boldsymbol{\mu} = (\mu_1,\dots,\mu_\dd)\in\Z_{>0}^\dd$ with~$|\boldsymbol{\mu}| = z+\dd$. If~$\mu_\dd=2$, there is nothing to prove. If~$\mu_\dd > 2$, we proceed as in the proof of Lemma~\ref{lem:red1}. If~$\mu_\dd = 1$, by assumption and Theorem~\ref{thm:sys1}, we have
    \begin{align*}
        \uu_{\mu_1}\cdots \uu_{\mu_{\dd-1}} = \sum\limits_{(\mathbf{n},\boldsymbol{\ell})\in\indexset{z}{\dd-1}} a_{\mathbf{n},\boldsymbol{\ell}}(\boldsymbol{\mu})\, \uu_{\mathbf{n}}\boxast \uu_{\boldsymbol{\ell}}
    \end{align*}
    for appropriate~$a_{\mathbf{n},\boldsymbol{\ell}}(\boldsymbol{\mu})\in\Q$. Then,
    \begin{align*}
        &\, \sum\limits_{(\mathbf{n},\boldsymbol{\ell})\in\indexset{z}{\dd-1}} a_{\mathbf{n},\boldsymbol{\ell}}(\boldsymbol{\mu})\, \uu_{\mathbf{n}}\boxast \uu_{\boldsymbol{\ell}}\uu_1
        =\, \uu_{\boldsymbol{\mu}} + \sum\limits_{(\mathbf{n},\boldsymbol{\ell})\in\indexset{z}{\dd-1}} a_{\mathbf{n},\boldsymbol{\ell}}(\boldsymbol{\mu}) \left(\uu_{(n_1,\dots,n_{s-1})}\boxast \uu_{\boldsymbol{\ell}}\right) \uu_{1+n_s}.
    \end{align*}
    The latter sum consists of words ending in some~$\uu_{\mu_\dd'}$ with~$\mu_\dd'\geq 2$. However, such words can be written as linear combinations of words ending in~$\uu_2$ and box products, similar to the proof of Lemma~\ref{lem:red1}, completing the proof.
\end{proof}

\begin{lemma}
\label{lem:red3}
    Fix~$z,\dd\in\Z_{>1}$ with~$z\geq \dd\geq 2$. If Conjecture~\ref{conj:systemApp} is true for~$(z-1,\dd-1)$, then every~$\uu_{\boldsymbol{\mu}}$ with~$\boldsymbol{\mu}\in\Z_{>0}^\dd$ and~$|\boldsymbol{\mu}| = z+\dd$ can be written as linear combination of words ending in~$\uu_3$ and non-trivial box products.
\end{lemma}

\begin{proof}
    Assume~$\dd$ and~$z$ as in the lemma. Using Lemma~\ref{lem:red2}, we only have to show that a word ending in~$\uu_2$ can be written as a linear combination of words ending in~$\uu_3$ and box products. Choose such a word~$\uu_{\mu_1}\cdots \uu_{\mu_{\dd-1}}\uu_2$, i.e.,~$2+\sum\limits_{j=1}^{\dd-1} \mu_j = z+\dd$. Then, by assumption, one has
    \begin{align*}
        \uu_{\mu_1}\cdots \uu_{\mu_{\dd-1}} = \sum\limits_{(\mathbf{n},\boldsymbol{\ell})\in\indexset{z-1}{\dd-1}} a_{\mathbf{n},\boldsymbol{\ell}}(\boldsymbol{\mu})\, \uu_{\mathbf{n}}\boxast \uu_{\boldsymbol{\ell}}
    \end{align*}
    for appropriate~$a_{\mathbf{n},\boldsymbol{\ell}}(\boldsymbol{\mu})\in\Q$. Hence,
    \begin{align*}
        &\, \sum\limits_{(\mathbf{n},\boldsymbol{\ell})\in\indexset{z-1}{\dd-1}} a_{\mathbf{n},\boldsymbol{\ell}}(\boldsymbol{\mu})\, \uu_{\mathbf{n}}\boxast \uu_{\boldsymbol{\ell}}\uu_2
        =\, \uu_{\boldsymbol{\mu}} + \sum\limits_{(\mathbf{n},\boldsymbol{\ell})\in\indexset{z-1}{\dd-1}} a_{\mathbf{n},\boldsymbol{\ell}}(\boldsymbol{\mu}) \left(\uu_{(n_1,\dots,n_{s-1})}\boxast \uu_{\boldsymbol{\ell}}\right) \uu_{2+n_s}.
    \end{align*}
    The latter sum consists of words ending in some~$\uu_{\mu_\dd'}$ with~$\mu_\dd'\geq 3$. However, such words can be written as linear combinations of words ending in~$\uu_3$ and box products, similar to the proof of Lemma~\ref{lem:red1}, completing the proof.
\end{proof}

\begin{lemma}
\label{lem:red1rightend}
    Let be~$\textbf{n}\in\Z_{>0}^s,\boldsymbol{\ell}\in\Z_{>0}^\dd$ with~$1\leq s\leq\dd$. Then,~$\uu_\textbf{n}\boxast \uu_{\boldsymbol{\ell}}$ can be written as linear combination of non-trivial box products~$\uu_{\boldsymbol{n'}}\boxast \uu_{\boldsymbol{\ell'}}$ where~$\boldsymbol{\ell'}$ ends in~$1$.
\end{lemma}

\begin{proof}
    Writing~$\boldsymbol{\ell} = (\ell_1,\dots,\ell_\dd)$, we may assume~$\ell_\dd > 1$ since for~$\ell_\dd = 1$ there is nothing to prove. Then,
    \begin{align*}
        \uu_\textbf{n}\boxast \uu_{\boldsymbol{\ell}} =&\, \uu_\textbf{n}\boxast \left(\uu_{\ell_\dd-1}\boxast \uu_{(\ell_1,\dots,\ell_{\dd-1},1)} - \sum\limits_{j=1}^{\dd-1} \uu_{(\ell_1,\dots,\ell_j+\ell_\dd-1,\dots,\ell_{\dd-1},1)}\right)
        \\
        =&\, \left(\uu_\textbf{n}\ast \uu_{\ell_\dd-1}\right)\boxast \uu_{(\ell_1,\dots,\ell_{\dd-1},1)} - \sum\limits_{j=1}^{\dd-1} \uu_\textbf{n}\boxast \uu_{(\ell_1,\dots,\ell_j+\ell_\dd-1,\dots,\ell_{\dd-1},1)},
    \end{align*}
    where we used Lemma~\ref{lem:stuffle-boxstuffle} in the last step.
\end{proof}

A further result about the numbers~$\sdim{z}{\dd}$ is the following lemma that gives a lower bound.
\begin{lemma}
    For all~$z,\dd\in\Z_{>0}$, we have 
  ~$
        \sdim{z}{\dd} \geq \binom{z+\dd-2}{\dd-1}.
  ~$
\end{lemma}

\begin{proof}
    We prove by induction on~$z+\dd$. For~$z=1$, the claim is clear, since for all~$\dd\in\Z_{>0}$, we have~$0\neq u_1\boxast u_1^\dd\in\VSbox{1}{\dd}$, i.e.,
    \begin{align*}
        \sdim{1}{\dd} \geq 1 = \binom{1+\dd - 2}{\dd - 1}.
    \end{align*}
    For~$\dd =1$, we have for all~$z\in\Z_{>0}$ equality by Lemma~\ref{lem:rzleq8}. In particular, the base case~$z+\dd = 2$ is proven. Now, let be~$z,\dd\in\Z_{>1}$ and assume that the lemma is proven for all smaller values of~$z+\dd$. By Lemma~\ref{lem:dineq} and the induction hypothesis, we obtain
    \begin{align*}
        \sdim{z}{\dd} \geq \sdim{z}{\dd-1} + \sdim{z-1}{\dd}\geq \binom{z+\dd-3}{\dd - 2} + \binom{z+\dd-3}{\dd-1}= \binom{z+\dd-2}{\dd-1}.\tag*{\qedhere}
    \end{align*}
\end{proof}

We end this subsection with some remark on Conjecture~\ref{conj:systemApp} that is independent of the rest of the paper.
\begin{remark}
Using basic linear algebra, we obtain the following equivalent formulation of Conjecture~\ref{conj:systemApp} in the cases~$z\geq\dd$. Fix positive integers~$\dd$ and~$z$ with~$z\geq \dd$. Conjecture~\ref{conj:systemApp} is true for the pair~$(z,\dd)$ if and only if the~$\binom{z+\dd-1}{\dd-1}$ expressions
    \begin{align}
        \eqlabel{eq:system1}
        \left\{\sum\limits_{(\boldsymbol{n},\boldsymbol{\ell})\in\indexset{z}{\dd}} \epsilon_{\boldsymbol{n},\boldsymbol{\ell}}^{\boldsymbol{\mu}} \, \uu_{\boldsymbol{n}}\boxast \uu_{\boldsymbol{\ell}}\ \Bigg| \ \boldsymbol{\mu}\in\Z_{>0}^\dd,\, |\boldsymbol{\mu}| = z+\dd\right\}
    \end{align}
    are~$\Q$-linearly independent. Here,~$\epsilon_{\boldsymbol{n},\boldsymbol{\ell}}^{\boldsymbol{\mu}}$ denotes the multiplicity of~$\uu_{\boldsymbol{\mu}}$ in~$\uu_{\boldsymbol{n}}\boxast \uu_{\boldsymbol{\ell}}$.
\end{remark}\noproof{remark}

\section{Our approach to the refined Bachmann Conjecture~\ref{conj:mdbdstrongApp}}
\label{sec:idea}

In the following, we present the approach with which one is trying to make progress in proving the refined Bachmann Conjecture~\ref{conj:mdbdstrongApp}. The general idea is to prove by induction on~$\zero(\word)$ for~$\word\in\mathcal{U}^{\ast,\circ}$ that~$\fz{\word}\in\Zqz$. This is trivial for the base case~$\zero(\word) = 0$. Thus, we assume~$\zero(\word) >0$. Particularly - for proving the induction step - one has to write~$\fz{\word}$ as a linear combination of~$\fz{\word'}$'s with~$\word'\in\mathcal{U}^{\ast,\circ}$ and~$\zero(\word') < \zero(\word)$. In our approach, we refine the induction step by showing that for every word~$\word\in\mathcal{U}^{\ast,\circ}$ we can write~$\fz{\word}$ as a linear combination of~$\fz{\word'}$'s with~$\word'\in\mathcal{U}^{\ast,\circ}$ and~$\zero(\word') < \zero(\word)$, or
\begin{align*}
    \zero(\word') = \zero(\word) \quad \text{and}\quad  \dep(\word') + \wt(\word') < \dep(\word) + \wt(\word)
\end{align*}
(see the refined Bachmann Conjecture~\ref{conj:mdbdstrongApp}). The general observation of why the refined Bachmann Conjecture~\ref{conj:mdbdstrongApp} is of interest when studying Bachmann's Conjecture~\ref{conj:mdbdApp} is given in the following lemma. 

\begin{lemma}[Lemma~\ref{lem:genidea2App}]
\label{lem:genideaApp}
    Fix~$z,\dd,w\in\Z_{>0}$. If the refined Bachmann Conjecture~\ref{conj:mdbdstrongApp} is true for~$(z,\dd,w)$ and if Bachmann's Conjecture~\ref{conj:mdbdApp} is true for all~$(z',\dd',w')\in\Z_{>0}^3$ with~$z'+\dd'+w' < z+\dd+w$, then Bachmann's Conjecture~\ref{conj:mdbdApp} is true for~$(z,\dd,w)$. In particular, the refined Bachmann Conjecture~\ref{conj:mdbdstrongApp} implies Bachmann's Conjecture~\ref{conj:mdbdApp}.
\end{lemma}

\begin{proof}
    Fix~$z,\dd,w\in\Z_{>0}$ and assume that the refined Bachmann Conjecture~\ref{conj:mdbdstrongApp} is true for~$(z,\dd,w)$ and that Bachmann's Conjecture~\ref{conj:mdbdApp} is true for all triples~$(z',\dd',w')\in\Z_{>0}^3$ satisfying~$z'+\dd'+w' < z+\dd+w$. By definition of~$\F{z,\dd,w}$ and the second part of our assumption, it follows
    \begin{align*}
        \F{z,d,w} =& \, \fil{Z,D,W}{z,d,w-1}\Zq + \sum\limits_{\substack{z'+d' = z+d-1\\ 0\leq z'\leq z}} \fil{Z,D,W}{z',d',w}\Zq
        \\
        \subset& \, \fil{D,W}{z+d,w-1}\Zqz + \fil{D,W}{z+d-1,w}\Zqz\subset\fil{D,W}{z+d,w}\Zqz.
    \end{align*}
    Using the assumption~$\fil{Z,D,W}{z,d,w}\Zq \subset \F{z,d,w}$, we obtain~$\fil{Z,D,W}{z,d,w}\Zq\subset\fil{D,W}{z+d,w}\Zqz$, i.e., Bachmann's Conjecture~\ref{conj:mdbdApp} for~$(z,\dd,w)$.
\end{proof}

For given~$z\geq\dd$, our approach to the refined Bachmann Conjecture~\ref{conj:mdbdstrongApp} restricts - independent of the weight~$w$ - to prove Conjecture~\ref{conj:systemApp} for the pair~$(z,\dd)$ as the following theorem shows.

\begin{theorem}
    \label{thm:concl}
    Fix~$z,\dd\in\Z_{>0}$ with~$z\geq \dd$. If Conjecture~\ref{conj:systemApp} is true for the pair~$(z,\dd)$, then for all~$w\in\Z_{>0}$, we have
    \begin{align*}
        \fil{Z,D,W}{z,\dd,w}\Zq\subset \sum\limits_{\substack{z'+\dd'=z+\dd-1\\ 0\leq z'\leq z-1}}\fil{Z,D,W}{z',\dd',w} \Zq\subset\F{z,d,w}.
    \end{align*}
    In particular, the refined Bachmann Conjecture~\ref{conj:mdbdstrongApp} is true for the triples~$(z,\dd,w)\in\Z_{>0}^3$ with~$w$ arbitrary.
\end{theorem}
    
\begin{proof}
    Fix~$z,\dd\in\Z_{>0}$ with~$z\geq \dd$ and assume that Conjecture~\ref{conj:systemApp} is true for~$(z,\dd)$. This means~$\uu_{\mathbf{z}} \in\CP$ for all~$\mathbf{z}\in\Z_{>0}^\dd$ with~$|\mathbf{z}| = z+\dd$. Hence, the claim follows immediately from Corollary~\ref{cor:mainideaboxApp}.
\end{proof}

\begin{remark}
    Immediately from Theorems~\ref{thm:sys1} and~\ref{thm:concl} the following statement is obtained: If Conjecture~\ref{conj:systemApp} is true for all~$z=\dd$, then we have
    \begin{align*}
        \fil{Z}{z}\Zq \subset \fil{Z}{\dd-1}\Zq
    \end{align*}
    for all~$(z,\dd)\in\Z_{>0}^2$ with~$z\geq \dd$. More precise, then we have
    \begin{align*}
        \Zq = \Zqz + \sum\limits_{\substack{0\leq z\leq \dd-1\\ \dd\geq 1}} \fil{Z,D,W}{z,\dd,2z+\dd-1}\Zq.
    \end{align*}
\end{remark}

\begin{remark}
    For~$z\geq\dd$, our approach to Bachmann's Conjecture~\ref{conj:mdbdApp}, and the refined Bachmann Conjecture~\ref{conj:mdbdstrongApp}, is to study Conjecture~\ref{conj:systemApp} in more detail. We will explain this in Section~\ref{sec:zlessr}. For~$z<\dd$, this approach will not suffice since in this case, we have~$\VSbox{z}{\dd}\subsetneq\Vbox{z}{\dd}$ by Conjecture~\ref{conj:systemApp} which is numerically explicit verified for small values of~$z$ and~$\dd$ (see Lemma~\ref{lem:rzleq8}). Hence, we need to extend our approach. We make do with few explicit calculations to prove our main results in Section~\ref{sec:zlessr}. In the outlook, Section~\ref{sec:mdbd:outlook}, we abstract our calculations and leave it as an open question whether this generalization is sufficient.
\end{remark}\noproof{remark}

\section{Proof of our main results towards the refined Bachmann Conjecture~\ref{conj:mdbdstrongApp}}
\label{sec:zlessr}

In this section, we first provide the proof of our main results, namely, Theorems~\ref{thm:mainApp} and~\ref{thm:rleq4App}, where some particular statements are black-boxed. We deliver their proofs in Sections~\ref{sec:refBazd=23},~\ref{sec:refBazd=24}, and~\ref{sec:refBazd=34}.

\begin{proposition}
\label{prop:r=2}
    The refined Bachmann Conjecture~\ref{conj:mdbdstrongApp} is true for all~$(z,2,w)\in\Z_{>0}^3$.
\end{proposition}

\begin{proof}
    Due to case~$\dd=2$ of Lemma~\ref{lem:rzleq8}, Conjecture~\ref{conj:systemApp} is true for all~$(z,2)\in\Z_{>0}^2$ with~$z\geq 2$. Theorem~\ref{thm:concl} then implies~$\fil{Z,D,W}{z,2,w}\Zq\subset \F{z,2,w}$ for all~$z,w\in\Z_{>0}$ with~$z\geq 2$. Hence, it remains to prove case~$z=1$. However, this follows immediately from the special case~$\dd=2$ of Corollary~\ref{cor:z=1}. 
\end{proof}

\begin{proposition}
\label{prop:r=3}
    The refined Bachmann Conjecture~\ref{conj:mdbdstrongApp} is true for all~$(z,3,w)\in\Z_{>0}^3$.
\end{proposition}

\begin{proof}
    The case~$z=1$ is proven by Corollary~\ref{cor:z=1}, the case $z=2$ will follow from Theorem~\ref{thm:r=32} below, and the cases~$z\geq 3$ are proven by the~$z=3$ case of Lemma~\ref{lem:smallz=3}, Theorem~\ref{thm:sys1}, and Theorem~\ref{thm:concl}. 
\end{proof}

\begin{proposition}
\label{prop:r=4}
    The refined Bachmann Conjecture~\ref{conj:mdbdstrongApp} is true for all~$(z,4,w)\in\Z_{>0}^3$.
\end{proposition}
\begin{proof}
    While the case~$z=1$ is proven by Corollary~\ref{cor:z=1}, the case~$z=2$ will be obtained from Theorem~\ref{thm:r=42} below, and the case~$z=3$ will be obtained from Theorem~\ref{thm:r=43} below. Furthermore, the cases~$z\geq 4$ are proven by Proposition~\ref{prop:zgeqd4} and Theorem~\ref{thm:concl}, completing the claim.
\end{proof}

We are now able to prove one of our main theorems.
\begin{theorem}[Theorem~\ref{thm:mainApp}]
\label{thm:main2App}
    Bachmann's Conjecture~\ref{conj:mdbdApp} is true for all~$(z,\dd,w)\in\Z_{>0}^3$ with~$z+\dd\leq 6$.
\end{theorem}
\begin{proof}
    For~$z+\dd\leq 3$, the claim is an immediate consequence of Proposition~\ref{prop:dep1explicit} and Corollary~\ref{cor:z=1}. For~$z=\dd=2$, the claim follows by induction on~$w$, using the proven claim for~$z+\dd\leq 3$, Lemma~\ref{lem:genideaApp}, and Proposition~\ref{prop:r=2} in the induction step. Together with Proposition~\ref{prop:dep1explicit} and Corollary~\ref{cor:z=1}, the claim holds now for~$z+\dd\leq 4$. Again, inductively on~$w$, the claim for~$(z,\dd) \in \{(3,2),(2,3)\}$ follows from the already proven claim for~$z+\dd\leq 4$, Lemma~\ref{lem:genideaApp}, and Proposition~\ref{prop:r=2} (for~$(z,\dd) = (3,2)$), and Proposition~\ref{prop:r=3} (for~$(z,\dd) = (2,3)$). Now, using Proposition~\ref{prop:dep1explicit} and Corollary~\ref{cor:z=1}, the claim follows for~$z+\dd\leq 5$. Analogously, for~$(z,\dd)\in\{(4,2),(3,3),(2,4)\}$, the claim follows in each case inductively on~$w$, where we use in the induction step the already proven claim for~$z+\dd\leq 5$, Lemma~\ref{lem:genideaApp}, and Proposition~\ref{prop:r=2} (for~$(z,\dd) = (4,2)$),~\ref{prop:r=3} (for~$(z,\dd) = (3,3)$), and~\ref{prop:r=4} (for~$(z,\dd) = (2,4)$), respectively. Now, using Proposition~\ref{prop:dep1explicit} and Corollary~\ref{cor:z=1}, the theorem is proven for~$z+\dd\leq 6$ as well, completing the proof.
\end{proof}

Theorem~\ref{thm:main2App} is the main ingredient in the proof of the next main theorem.
\begin{theorem}[Theorem~\ref{thm:rleq4App}]
\label{thm:rleq42App}
     The refined Bachmann Conjecture~\ref{conj:mdbdstrongApp} is true for all triples of positive integers~$(z,\dd,w)\in\Z_{>0}^3$ with~$1\leq \dd\leq 4$.
\end{theorem}
\begin{proof}
    For~$1\leq z < \dd\leq 4$ and~$w\in\Z_{>0}$ arbitrary, we obtain the claim from Theorem~\ref{thm:main2App}. Furthermore, for~$1\leq\dd\leq 3$,~$z\geq\dd$ and~$w\in\Z_{>0}$ arbitrary, we obtain the claim from Corollary~\ref{cor:mainideaboxApp}, Lemma~\ref{lem:rzleq8}, and Theorem~\ref{thm:concl}. For~$\dd=z=4$ and~$w\in\Z_{>0}$ arbitrary, the claim follows from Proposition~\ref{prop:zgeqd4} and Corollary~\ref{cor:mainideaboxApp}. Hence, for~$z\geq \dd = 4$ and~$w\in\Z_{>0}$ arbitrary, the claim is a direct consequence of Corollary~\ref{cor:mainideaboxApp} and Theorem~\ref{thm:concl}, proving the theorem finally.
\end{proof}

\subsection{The refined Bachmann Conjecture~\ref{conj:mdbdstrongApp} for~$(z,\dd,w) = (2,3,w)$}
\label{sec:refBazd=23}

\begin{theorem}
\label{thm:r=32}
    The refined Bachmann Conjecture~\ref{conj:mdbdstrongApp} is true for all~$(2,3,w)\in\Z_{>0}^3$, i.e.,
    \begin{align}
    \eqlabel{eq:toshow23}        \fz{u_{k_1}u_0^{z_1}u_{k_2}u_0^{z_2}u_{k_3}u_0^{z_3}}\in\F{2,3,w}
    \end{align}
    for all~$k_j\in\Z_{>0},\, z_j\in\Z_{\geq 0}$, for $1\leq j\leq 3$, satisfying $z_1+z_2+z_3 = 2$ and $w = k_1+k_2+k_3+2$.
\end{theorem}

\begin{proof}
    For~$k_1 = k_2 = k_3 = 1$ and for all~$z_1,z_2,z_3\geq 0$ satisfying~$z_1+z_2+z_3=2$,~\eqref{eq:toshow23} is true since, after using~$\tau$-invariance of~$\zeta_q^{\mathrm{f}}$, we have
    \begin{align*}
        \fz{u_{1}u_0^{z_1}u_{1}u_0^{z_2}u_{1}u_0^{z_3}} = \fz{u_{z_3+1}u_{z_2+1}u_{z_3+1}} \in \F{2,3,w}.
    \end{align*}
    Furthermore, for~$k_2>1$,~\eqref{eq:toshow23} will follow from Lemma~\ref{lem:k2gg1}, for~$k_3>1$,~\eqref{eq:toshow23} will follow from Lemma~\ref{lem:k3gg1}, and for $k_1>1$,~\eqref{eq:toshow23} will follow from Lemma~\ref{lem:k1gg1}, completing the proof of the theorem.
\end{proof}

\begin{lemma}
\label{lem:r=3main}
    Let be~$k_1,k_2,k_3\in\Z_{>0}$ and write~$w = k_1+k_2+k_3 + 2$. We have
    \begin{align}
    \eqlabel{eq:311=122}
    \fz{u_{k_1} u_{k_2} u_{k_3} u_0 u_0},\ & \fz{u_{k_1} u_0 u_{k_2} u_0 u_{k_3}} &&\hspace{3mm}\in\F{2,3,w},\\
    \eqlabel{eq:131=212}
    \fz{u_{k_1} u_{k_2} u_0 u_0 u_{k_3}} \equiv\,& \fz{u_{k_1} u_0 u_{k_2} u_{k_3} u_0}&&\mod\F{2,3,w} \\
    \eqlabel{eq:113=221}
    \equiv\, -\fz{u_{k_1} u_0 u_0 u_{k_2} u_{k_3}} \equiv\,& -\fz{u_{k_1} u_{k_2} u_0 u_{k_3} u_0} &&\mod\F{2,3,w}.
\end{align}
In particular, for fixed~$k_1,k_2,k_3$, if one of the latter four formal Multiple Zeta Values is in~$\F{2,3,w}$, \eqref{eq:toshow23} is true for the corresponding choice of~$k_1,k_2,k_3$.
\end{lemma}

\begin{proof}
First note that \eqref{eq:311=122} is a consequence of Corollaries~\ref{cor:z1111111=02} and~\ref{cor:11111112}. Furthermore, after using Lemma~\ref{lem:mainideaboxApp} and~$\tau$-invariance of formal \qmzv s, with \eqref{eq:311=122}, we obtain
\begin{align}
    0\equiv\ &\fz{\Psi_{(k_1,k_2,k_3)}(u_1 u_1\boxast u_1 u_1 u_1)}
    \hspace{-5mm}&&\hspace{-9mm}\mod\F{2,3,w}
    \\
    \equiv\ &\fz{u_2 u_0^{k_3-1} u_2 u_0^{k_2-1} u_1 u_0^{k_1-1} + u_2 u_0^{k_3-1} u_1 u_0^{k_2-1} u_2 u_0^{k_1-1}} \hspace{-5mm}&&\hspace{-9mm}\mod\F{2,3,w}
    \\
    \eqlabel{eq:11-111}
    \equiv\ &\fz{u_{k_1} u_{k_2} u_0 u_{k_3} u_0} + \fz{u_{k_1} u_0 u_{k_2} u_{k_3} u_0} \hspace{-5mm}&&\hspace{-9mm}\mod\F{2,3,w},
    \\
    0\equiv\ &\fz{\Psi_{(k_1,k_2,k_3)}(u_1\boxast u_1 u_2 u_1)} 
    \hspace{-5mm}&&\hspace{-9mm}\mod\F{2,3,w}
    \\
    \equiv\ &\fz{u_2 u_0^{k_3-1} u_2 u_0^{k_2-1} u_1 u_0^{k_1-1} + u_1 u_0^{k_3-1} u_3 u_0^{k_2-1} u_1 u_0^{k_1-1}} \hspace{-5mm}&&\hspace{-9mm}\mod\F{2,3,w}
    \\
    \equiv\ &\fz{u_{k_1} u_{k_2} u_0 u_{k_3} u_0} + \fz{u_{k_1} u_{k_2} u_0 u_0 u_{k_3}} \hspace{-5mm}&&\hspace{-9mm}\mod\F{2,3,w},
    \\
    \eqlabel{eq:1-121}
    0\equiv\ &\fz{\Psi_{(k_1,k_2,k_3)}(u_1\boxast u_1 u_1 u_2)}
    \hspace{-5mm}&&\hspace{-9mm}\mod\F{2,3,w}
    \\
    \equiv\ &\fz{u_2 u_0^{k_3-1} u_1 u_0^{k_2-1} u_2 u_0^{k_1-1} + u_1 u_0^{k_3-1} u_1 u_0^{k_2-1} u_3 u_0^{k_1-1}} \hspace{-5mm}&&\hspace{-9mm}\mod\F{2,3,w}
    \\
    \eqlabel{eq:1-112}
    \equiv\ &\fz{u_{k_1} u_0 u_{k_2} u_{k_3} u_0} + \fz{u_{k_1} u_0 u_0 u_{k_2} u_{k_3}} \hspace{-5mm}&&\hspace{-9mm}\mod\F{2,3,w}.
\end{align}
We obtain \eqref{eq:131=212} and \eqref{eq:113=221}, by comparing~\eqref{eq:11-111},~\eqref{eq:1-121}, and~\eqref{eq:1-112}.
\end{proof}

\begin{lemma}
\label{lem:k2gg1}
    Equation~\eqref{eq:toshow23} is true for~$k_2>1$. 
\end{lemma}

\begin{proof}
Let be~$k_1,k_2,k_3\in\Z_{>0}$ and write~$w=k_1+k_2+k_3+3$. By~\eqref{eq:basicstuffle1}, we have
\begin{align*}
    u_2 u_1 \ast u_{k_1} u_{k_2} u_{k_3}\in\fil{Z,D,W}{0,5,w}\QB^\circ.
\end{align*}
Hence, and due to~$\tau$-invariance of formal \qmzv s, we have
\begin{align*}
    0\equiv\ & \frac{1}{k_2} \fz{\tau(u_2 u_1)\ast \tau\left(u_{k_1} u_{k_2} u_{k_3}\right)}&&\mod\F{2,3,w}
    \\
    \equiv\ &\frac{1}{k_2} \fz{u_1 u_1 u_0 \ast u_1 u_0^{k_3-1} u_1 u_0^{k_2-1} u_1 u_0^{k_1-1}}&&\mod\F{2,3,w}
    \\
    \equiv\ &\fz{u_2 u_0^{k_3-1} u_2 u_0^{k_2} u_1 u_0^{k_1-1}} + \frac{k_1}{k_2} \fz{u_1 u_1\ast u_1 u_0^{k_3-1} u_1 u_0^{k_2-1} u_1 u_0^{k_1}}&&\mod\F{2,3,w}
    \\
    \equiv\ &\fz{u_{k_1} u_{k_2+1} u_0 u_{k_3} u_0} + \fz{\Psi_{(k_1+1,k_2,k_3)}(u_1 u_1\boxast u_1 u_1 u_1)} &&\mod\F{2,3,w},
     \\
    \equiv\ &\fz{u_{k_1} u_{k_2+1} u_0 u_{k_3} u_0} &&\mod\F{2,3,w},
\end{align*}
where the last step is a consequence of Lemma~\ref{lem:mainideaboxApp}. Now, with Lemma~\ref{lem:r=3main}, \eqref{eq:toshow23} indeed is true for~$k_2>1$. 
\end{proof}

\begin{lemma}
\label{lem:k3gg1}
    Equation~\eqref{eq:toshow23} is true for~$k_3>1$.
\end{lemma}

\begin{proof}
    Let be~$k_1,k_2,k_3\in\Z_{>0}$ and write~$w = k_1 + k_2 + k_3 + 3$.  By~\eqref{eq:basicstuffle1}, we have
\begin{align*}
    u_2 \ast u_{k_1} u_0 u_{k_2} u_{k_3}\in\fil{Z,D,W}{1,4,w}\QB^\circ.
\end{align*}
Hence, and due to~$\tau$-invariance of formal \qmzv s, we have
    \begin{align*}
        0\equiv\ &\frac{1}{k_3} \fz{\tau(u_2) \ast \tau\left(u_{k_1} u_0 u_{k_2} u_{k_3}\right)} &&\hspace{-6mm}\mod\F{2,3,w}
        \\
        \equiv\ &\frac{1}{k_3} \fz{u_1 u_0 \ast u_1 u_0^{k_3-1} u_1 u_0^{k_2-1} u_2 u_0^{k_1-1}} &&\hspace{-6mm}\mod\F{2,3,w}
        \\
        \equiv\ & \fz{u_2 u_0^{k_3} u_1 u_0^{k_2-1} u_2 u_0^{k_1-1}} + \frac{k_2}{k_3} \fz{u_2 u_0^{k_3-1} u_1 u_0^{k_2} u_2 u_0^{k_1-1}} 
        \\ &+ \frac{k_2}{k_3} \fz{u_1 u_0^{k_3-1} u_2 u_0^{k_2} u_2 u_0^{k_1-1}} + \frac{k_1}{k_3}\fz{\Psi_{(k_1+1,k_2,k_3)}(u_1\boxast u_1 u_1 u_2} &&\hspace{-6mm}\mod\F{2,3,w}
        \\
        \equiv\ & \fz{u_{k_1} u_0 u_{k_2} u_{k_3+1} u_0} &&\hspace{-6mm}\mod\F{2,3,w}.
    \end{align*}
    The last step is obtained by Lemmas~\ref{lem:mainideaboxApp} and~\ref{lem:k2gg1}. Hence, the lemma is proven by Lemma~\ref{lem:r=3main}.
\end{proof}

Hence, for proving Proposition~\ref{prop:r=3}, the remaining case is~$k_2=k_3=1$. 
\begin{lemma}
\label{lem:k1gg1}
    Equation~\eqref{eq:toshow23} is true for~$k_1>1$.
\end{lemma}

\begin{proof}
    Let be~$k_1,k_2,k_3\in\Z_{>0}$ with~$k_1>1$ and write~$w = k_1 + k_2 + k_3 + 2$. Due to Lemmas~\ref{lem:k2gg1} and~\ref{lem:k3gg1}, we may assume~$k_1>1$ and~$k_2=k_3=1$, i.e.,~$w = k_1 + 4$ then. By Proposition~\ref{prop:r=2}, we have~$\fz{u_{k_1} u_0 u_1 u_0}\in\F{2,2,w-1}$ and thus~$\fz{u_1 \ast u_{k_1} u_0 u_1 u_0}\in\F{2,3,w}$. Multiplying out the latter product and using Proposition~\ref{prop:r=2},~\eqref{eq:311=122}, and Lemma~\ref{lem:k2gg1}, we see that
    \begin{align*}
        0\equiv\ &\fz{u_1 \ast u_{k_1} u_0 u_1 u_0} \equiv 2 \fz{u_{k_1} u_0 u_1 u_1 u_0 + u_{k_1} u_1 u_0 u_1 u_0} &&\mod \F{2,3,w}
        \\
        \equiv\ &\fz{u_{k_1} u_0 u_1 u_1 u_0} &&\mod \F{2,3,w},
    \end{align*}
    where the last congruence is obtained from \eqref{eq:11-111}. Thus, the proof of the lemma follows from Lemma~\ref{lem:r=3main}.
\end{proof}




\subsection{The refined Bachmann Conjecture~\ref{conj:mdbdstrongApp} for~$(z,\dd,w) = (2,4,w)$}
\label{sec:refBazd=24}




\begin{theorem}
\label{thm:r=42}
    The refined Bachmann Conjecture~\ref{conj:mdbdstrongApp} is true for all~$(2,4,w)\in\Z_{>0}^3$, i.e.,
    \begin{align}
    \eqlabel{eq:toshow24}        \fz{u_{k_1}u_0^{z_1}u_{k_2}u_0^{z_2}u_{k_3}u_0^{z_3}u_{k_4}u_0^{z_4}}\in\F{2,4,w}
    \end{align}
    for all~$k_j\in\Z_{>0},\, z_j\in\Z_{\geq 0}$, for $1\leq j\leq 4$, satisfying $z_1+z_2+z_3+z_4 = 2$ and $w = k_1+k_2+k_3+k_4+2$.
\end{theorem}

\begin{proof}
    In the case~$k_1 = k_2 = k_3 = k_4 = 1$,~\eqref{eq:toshow24} is true since for all~$z_1,\dots,z_4\geq 0$, we have by~$\tau$-invariance of~$\zeta_q^{\mathrm{f}}$ that
    \begin{align}
        \fz{u_{1}u_0^{z_1}u_{1}u_0^{z_2}u_{1}u_0^{z_3}u_{1}u_0^{z_4}} = \fz{u_{z_4+1}u_{z_3+1}u_{z_2+1}u_{z_1+1}}\in\fil{D,W}{4,w}\Zqz.
    \end{align}
    In the four cases~$k_{i_1},\ k_{i_2},\ k_{i_3} > 1$ with pairwise distinct~$i_1,i_2,i_3\in\{1,2,3,4\}$,~\eqref{eq:toshow24} will follow from Lemma~\ref{lem:r=4:k2k4}, Proposition~\ref{r=4:k1k3k4}, and Proposition~\ref{r=4:k1k2k3}. Furthermore, the six cases~$k_{i_1},k_{i_2} > 1$ for distinct~$i_1,i_2\in\{1,2,3,4\}$ (and the two other~$k_j$'s equal~$1$) then follow from Lemmas~\ref{lem:r=4:k2k4},~\ref{lem:r=4:k3k4},~\ref{lem:r=4:k2k3}, and~\ref{lem:r=4:k1kj}. Next, the four cases of~$k_i > 1$ ($i\in\{1,2,3,4\}$) (and the three other~$k_j$'s equal $1$), will follow from Lemmas~\ref{lem:r=4:k3},~\ref{lem:r=4:k4},~\ref{lem:r=4:k2}, and~\ref{lem:r=4:k1}. This completes the proof of the theorem.    
\end{proof}


In the following three lemmas, we state some congruences that are true independently of the several cases we might consider.

\begin{lemma}
\label{lem:r=42first}
    Let be~$k_1,\dots,k_4\in\Z_{>0}$ and write~$w=k_1+\cdots+k_4+2$. We have
    \begin{align}
    \eqlabel{eq:kkkk00=0}
        0\equiv\ &\fz{u_{k_1}u_{k_2}u_{k_3}u_{k_4}u_0u_0} &&\mod\F{2,4,w},\quad 
        \\
        \eqlabel{eq:k0k0kk=0}
        0\equiv\ &\fz{u_{k_1}u_0u_{k_2}u_0u_{k_3}u_{k_4}} &&\mod\F{2,4,w},\quad 
        \\
        \eqlabel{eq:kkk0k0+kkk00k=0}
        0\equiv\ &\fz{u_{k_1}u_{k_2}u_{k_3}u_0u_{k_4}u_0} + \fz{u_{k_1}u_{k_2}u_{k_3}u_0u_0u_{k_4}} &&\mod\F{2,4,w}.
    \end{align}
\end{lemma}

\begin{proof}
    Note that \eqref{eq:kkkk00=0} is a direct consequence of Corollary~\ref{cor:z1111111=02}, while \eqref{eq:k0k0kk=0} follows from Corollary~\ref{cor:11111112}. Last, \eqref{eq:kkk0k0+kkk00k=0} follows from \eqref{eq:kkkk00=0} and the special case~$\dd =4$,~$z=2$,~$j=3$ of Corollary~\ref{cor:sumtozApp}. 
\end{proof}

\begin{lemma}
\label{lem:r=4:dual}
    Let be~$k_1,\dots,k_4\in\Z_{>0}$ and write~$w=k_1+\cdots+k_4+2$. We have 
    \begin{align}
    \begin{split}
    \eqlabel{eq:1-2111new1}
   0 \equiv\ & \fz{u_{k_1}u_{k_2}u_{k_3}u_0u_{k_4}u_0} + \fz{u_{k_1}u_{k_2}u_0u_{k_3}u_{k_4}u_0} \\ &+ \fz{u_{k_1}u_0u_{k_2}u_{k_3}u_{k_4}u_0}\hspace{37.5mm}\mod \F{2,4,w},
   \end{split}
    \\
    \eqlabel{eq:1-1211new3}
    0 \equiv\ & \fz{u_{k_1}u_{k_2}u_0u_{k_3}u_0u_{k_4}} + \fz{u_{k_1}u_0u_{k_2}u_{k_3}u_0u_{k_4}}\mod \F{2,4,w},
    \\
    \begin{split}
    \eqlabel{eq:1-1121new1}
    0 \equiv\ & \fz{u_{k_1}u_{k_2}u_0u_{k_3}u_{k_4}u_0} + \fz{u_{k_1}u_{k_2}u_0u_{k_3}u_0u_{k_4}} \\ &+ \fz{u_{k_1}u_{k_2}u_0u_0u_{k_3}u_{k_4}}\hspace{37.5mm}\mod \F{2,4,w},
    \end{split}
    \\
    \begin{split}
    \eqlabel{eq:1-1112new1}
    0 \equiv\ & \fz{u_{k_1}u_0u_{k_2}u_{k_3}u_{k_4}u_0} + \fz{u_{k_1}u_0u_{k_2}u_{k_3}u_0u_{k_4}} \\ &+ \fz{u_{k_1}u_0u_0u_{k_2}u_{k_3}u_{k_4}}\hspace{37.5mm}\mod \F{2,4,w},
    \end{split}
    \\
    \begin{split}
    \eqlabel{eq:2-1111new1}
    0 \equiv\ & \fz{u_{k_1}u_{k_2}u_{k_3}u_0u_0u_{k_4}} + \fz{u_{k_1}u_{k_2}u_0u_0u_{k_3}u_{k_4}} \\ &+ \fz{u_{k_1}u_0u_0u_{k_2}u_{k_3}u_{k_4}}\hspace{37.5mm}\mod \F{2,4,w},
    \end{split}
    \\
    \begin{split}
    \eqlabel{eq:11-1111new1}
    0 \equiv\ & \fz{u_{k_1}u_{k_2}u_{k_3}u_0u_{k_4}u_0} + \fz{u_{k_1}u_{k_2}u_0u_{k_3}u_{k_4}u_0} \\ &+ \fz{u_{k_1}u_0u_{k_2}u_{k_3}u_{k_4}u_0} +\fz{u_{k_1}u_{k_2}u_0u_{k_3}u_0u_{k_4}} \\ &+ \fz{u_{k_1}u_0u_{k_2}u_{k_3}u_0u_{k_4}}\hspace{37.5mm}\mod \F{2,4,w}.
    \end{split}
\end{align}
\end{lemma}

\begin{proof}
    All relations are, by Lemma~\ref{lem:mainideaboxApp}, a consequence of
    \begin{align*} 
        0\equiv \fz{\tau(\Psi_{\mathbf{k}}(\uu_{\mathbf{n}}\boxast \uu_{\boldsymbol{\ell}}))}\mod\F{2,4,w}
    \end{align*} with~$\mathbf{k} = (k_1,\dots,k_4)$ each and~$(\mathbf{n},\boldsymbol{\ell})\in\indexset{2}{4}$, where Lemma~\ref{lem:r=42first} was applied. Precisely, for~\eqref{eq:1-2111new1}, we used~$(\mathbf{n},\boldsymbol{\ell}) = ((1),(2,1,1,1))$, for~\eqref{eq:1-1211new3}, we used~$(\mathbf{n},\boldsymbol{\ell}) = ((1),(1,2,1,1))$, for~\eqref{eq:1-1121new1}, we used~$(\mathbf{n},\boldsymbol{\ell}) = ((1),(1,1,2,1))$, for~\eqref{eq:1-1112new1}, we used~$(\mathbf{n},\boldsymbol{\ell}) = ((1),(1,1,1,2)$, for~\eqref{eq:2-1111new1}, we used~$(\mathbf{n},\boldsymbol{\ell}) = ((2),(1,1,1,1))$. Furthermore, for~\eqref{eq:11-1111new1}, we used the element~$(\mathbf{n},\boldsymbol{\ell}) = ((1,1),(1,1,1,1))$ of~$\indexset{2}{4}$.
\end{proof}

\begin{lemma}
\label{lem:k4cons}
    Let be~$k_1,\dots,k_4\in\Z_{>0}$ and write~$w=k_1+\cdots+k_4+3$. We have
    \begin{align}
    \eqlabel{eq:10-1112new1}
    \begin{split}
    0 \equiv\ & k_4\fz{u_{k_1}u_0u_{k_2}u_{k_3}u_{k_4+1}u_0} - k_3\fz{u_{k_1}u_0u_0u_{k_2}u_{k_3+1}u_{k_4}}
    \\
    &-k_2\fz{u_{k_1}u_0u_0u_{k_2+1}u_{k_3}u_{k_4}}\hspace{42.8mm}\mod\F{2,4,w},
    \end{split}
    \\
    \eqlabel{eq:10-1121new1}
    0 \equiv\ & k_4\fz{u_{k_1}u_{k_2}u_0u_{k_3}u_{k_4+1}u_0} - k_3\fz{u_{k_1}u_{k_2}u_0u_0u_{k_3+1}u_{k_4}}\hspace{-2.5mm}\mod\F{2,4,w},
    \\
    \eqlabel{eq:10-1211new1}
    0 \equiv\ & k_4\fz{u_{k_1}u_{k_2}u_{k_3}u_0u_{k_4+1}u_0} + k_2\fz{u_{k_1}u_{k_2+1}u_0u_{k_3}u_0u_{k_4}}\hspace{-2.5mm}\mod\F{2,4,w},
    \\
    \eqlabel{eq:10-2111new1}
    0 \equiv\ & k_3\fz{u_{k_1}u_{k_2}u_{k_3+1}u_0u_{k_4}u_0} - k_2\fz{u_{k_1}u_0u_{k_2+1}u_{k_3}u_{k_4}u_0}\hspace{-2.5mm}\mod\F{2,4,w},
    \\
    \eqlabel{eq:20-1111new1}
    0 \equiv\ & k_3\fz{u_{k_1}u_{k_2}u_{k_3+1}u_0u_0u_{k_4}} - k_2\fz{u_{k_1}u_0u_0u_{k_2+1}u_{k_3}u_{k_4}}\hspace{-2.5mm}\mod\F{2,4,w}.
\end{align}
\end{lemma}

\begin{proof}
    We use~$\tau$-invariance of formal \qmzv s and Corollary~\ref{cor:z=1} to see in the following calculations that each of the formal \qmzv s of stuffle products in the first line indeed is an element of~$\F{2,4,w}$ in the following.
    
    Now, by \eqref{eq:k0k0kk=0} and \eqref{eq:1-1112new1}, we have
    \begin{align}
    \eqlabel{eq:10-1112}
    0\equiv\ & \fz{\tau(u_2)\ast\tau\left(u_{k_1} u_0 u_{k_2} u_{k_3} u_{k_4}\right)}&&\hspace{-1.7cm}\mod\F{2,4,w}
    \\
    \equiv\ &\fz{u_1 u_0 \ast u_1 u_0^{k_4-1} u_1 u_0^{k_3-1} u_1 u_0^{k_2-1} u_2 u_0^{k_1-1}} &&\hspace{-1.7cm}\mod\F{2,4,w}
    \\ \nonumber 
    \equiv\ &k_4 \fz{u_2 u_0^{k_4} u_1 u_0^{k_3-1} u_1 u_0^{k_2-1} u_2 u_0^{k_1-1}} - k_3 \fz{u_1 u_0^{k_4-1} u_1 u_0^{k_3} u_1 u_0^{k_2-1} u_3 u_0^{k_1-1}}
    \\ &- k_2 \fz{u_1 u_0^{k_4-1} u_1 u_0^{k_3-1} u_1 u_0^{k_2} u_3 u_0^{k_1-1}} &&\hspace{-1.7cm}\mod\F{2,4,w}
    \\ \nonumber 
    \equiv\ &k_4 \fz{u_{k_1} u_0 u_{k_2} u_{k_3} u_{k_4+1} u_0} - k_3 \fz{u_{k_1} u_0 u_0 u_{k_2} u_{k_3+1} u_{k_4}}
    \\ &- k_2 \fz{u_{k_1} u_0 u_0 u_{k_2+1} u_{k_3} u_{k_4}} &&\hspace{-1.7cm}\mod\F{2,4,w},
\end{align}
proving \eqref{eq:10-1112new1}. Furthermore, using \eqref{eq:k0k0kk=0}, we have
\begin{align}
    \eqlabel{eq:10-1121}
    0\equiv\ & \fz{\tau(u_2)\ast\tau\left(u_{k_1} u_{k_2} u_0 u_{k_3} u_{k_4}\right)}&&\mod\F{2,4,w}
    \\
    \equiv\ &\fz{u_1 u_0 \ast u_1 u_0^{k_4-1} u_1 u_0^{k_3-1} u_2 u_0^{k_2-1} u_1 u_0^{k_1-1}} &&\mod\F{2,4,w}
    \\ \nonumber \equiv\ &k_4 \fz{u_2 u_0^{k_4} u_1 u_0^{k_3-1} u_2 u_0^{k_2-1} u_1 u_0^{k_1-1}} \\ &- k_3 \fz{u_1 u_0^{k_4-1} u_1 u_0^{k_3} u_3 u_0^{k_2-1} u_1 u_0^{k_1-1}} &&\mod\F{2,4,w}
    \\ \nonumber \equiv\ &k_4 \fz{u_{k_1}u_{k_2}u_0u_{k_3}u_{k_4+1}u_0} - k_3 \fz{u_{k_1}u_{k_2}u_0u_0u_{k_3+1}u_{k_4}} &&\mod\F{2,4,w},
\end{align}
proving \eqref{eq:10-1121new1}. Now, applying \eqref{eq:kkk0k0+kkk00k=0} yields
\begin{align}
    \eqlabel{eq:10-1211}
    0\equiv\ & \fz{\tau(u_2)\ast\tau\left(u_{k_1} u_{k_2} u_{k_3} u_0 u_{k_4}\right)}&&\mod\F{2,4,w}
    \\
    \equiv\ &\fz{u_1 u_0 \ast u_1 u_0^{k_4-1} u_2 u_0^{k_3-1} u_1 u_0^{k_2-1} u_1 u_0^{k_1-1}} &&\mod\F{2,4,w}
    \\ \nonumber
    \equiv\ &k_4 \fz{u_2 u_0^{k_4} u_2 u_0^{k_3-1} u_1 u_0^{k_2-1} u_1 u_0^{k_1-1}} \\ &+ k_2 \fz{u_1 u_0^{k_4-1} u_2 u_0^{k_3-1} u_2 u_0^{k_2} u_1 u_0^{k_1-1}} &&\mod\F{2,4,w}
    \\ \nonumber \equiv\ &k_4 \fz{u_{k_1}u_{k_2}u_{k_3}u_0u_{k_4+1}u_0} + k_2 \fz{u_{k_1}u_{k_2+1}u_0u_{k_3}u_0u_{k_4}} &&\mod\F{2,4,w},
\end{align}
proving \eqref{eq:10-1211new1}. Next, use \eqref{eq:kkkk00=0} and \eqref{eq:1-2111new1} to obtain 
\begin{align}
    \eqlabel{eq:10-2111}
    0\equiv\ & \fz{\tau(u_2)\ast\tau\left(u_{k_1} u_{k_2} u_{k_3} u_{k_4} u_0\right)}&&\mod\F{2,4,w}
    \\
    \equiv\ &\fz{u_1 u_0 \ast u_2 u_0^{k_4-1} u_1 u_0^{k_3-1} u_1 u_0^{k_2-1} u_1 u_0^{k_1-1}} &&\mod\F{2,4,w}
    \\ \nonumber 
    \equiv\ & k_3 \fz{u_2 u_0^{k_4-1} u_2 u_0^{k_3} u_1 u_0^{k_2-1} u_1 u_0^{k_1-1}} \\ &- k_2 \fz{u_2 u_0^{k_4-1} u_1 u_0^{k_3-1} u_1 u_0^{k_2} u_2 u_0^{k_1-1}} &&\mod\F{2,4,w}
    \\ \nonumber \equiv\ &k_3 \fz{u_{k_1}u_{k_2}u_{k_3+1}u_0u_{k_4}u_0} - k_2 \fz{u_{k_1}u_0u_{k_2+1}u_{k_3}u_{k_4}u_0} &&\mod\F{2,4,w},
\end{align}
proving \eqref{eq:10-2111new1}. Now, \eqref{eq:kkkk00=0} and \eqref{eq:2-1111new1} imply
\begin{align}
    \eqlabel{eq:20-1111}
    0\equiv\ & \fz{\tau(u_2 u_0)\ast\tau\left(u_{k_1} u_{k_2} u_{k_3} u_{k_4}\right)}&&\mod\F{2,4,w}
    \\
    \equiv\ &\fz{u_2 u_0 \ast u_1 u_0^{k_4-1} u_1 u_0^{k_3-1} u_1 u_0^{k_2-1} u_1 u_0^{k_1-1}} &&\mod\F{2,4,w}
    \\
    \equiv\ &k_3 \fz{u_1 u_0^{k_4-1} u_3 u_0^{k_3} u_1 u_0^{k_2-1} u_1 u_0^{k_1-1}} \\ &- k_2 \fz{u_1 u_0^{k_4-1} u_1 u_0^{k_3-1} u_1 u_0^{k_2} u_3 u_0^{k_1-1}} &&\mod\F{2,4,w}
    \\ \nonumber \equiv\ &k_3 \fz{u_{k_1}u_{k_2}u_{k_3+1}u_0u_0u_{k_4}} - k_2 \fz{u_{k_1}u_0u_0u_{k_2+1}u_{k_3}u_{k_4}} &&\mod\F{2,4,w},
\end{align}
proving \eqref{eq:20-1111new1}. 
This completes the proof of the lemma.
\end{proof}

\begin{corollary}
\label{cor:k4cons}
    Let be~$k_1,\dots,k_4\in\Z_{>0}$ and write~$w=k_1+\cdots+k_4+3$. We have
    \begin{align}
        \eqlabel{eq:k2:1212=0}
        0\equiv\ &\fz{u_{k_1}u_0u_{k_2+1}u_{k_3}u_0u_{k_4}} &&\mod\F{2,4,w},
        \\
        \eqlabel{eq:k2:1221=0}
        0\equiv\ &\fz{u_{k_1}u_{k_2+1}u_0u_{k_3}u_0u_{k_4}} &&\mod\F{2,4,w},
        \\
        \eqlabel{eq:k4:2211=0}
        0 \equiv\ &\fz{u_{k_1}u_{k_2}u_{k_3}u_0u_{k_4+1}u_0} &&\mod\F{2,4,w},
        \\
         \eqlabel{eq:k4:1311=0}
        0 \equiv\ &\fz{u_{k_1}u_{k_2}u_{k_3}u_0u_0u_{k_4+1}} &&\mod\F{2,4,w}.
    \end{align}
\end{corollary}

\begin{proof}
    Adding \eqref{eq:10-2111new1} and \eqref{eq:20-1111new1}, yields, applying \eqref{eq:kkk0k0+kkk00k=0},
    \begin{align*}
        0\equiv\ & -k_2\left(\fz{u_{k_1}u_0u_{k_2+1}u_{k_3}u_{k_4}u_0} + \fz{u_{k_1}u_0u_0u_{k_2+1}u_{k_3}u_{k_4}}\right) &&\mod\F{2,4,w}
        \\
        \equiv\ & k_2 \fz{u_{k_1}u_0u_{k_2+1}u_{k_3}u_0u_{k_4}} &&\mod\F{2,4,w},
    \end{align*}
    where the last step follows from \eqref{eq:1-1112new1}. Hence, \eqref{eq:k2:1212=0} is proven. Furthermore, \eqref{eq:k2:1221=0} is deducted from \eqref{eq:1-1211new3} and \eqref{eq:k2:1212=0}. Now, \eqref{eq:k4:2211=0} follows from \eqref{eq:10-1211new1} and \eqref{eq:k2:1221=0}. Since \eqref{eq:k4:1311=0} is a consequence of \eqref{eq:k4:2211=0} and \eqref{eq:kkk0k0+kkk00k=0}, the corollary is proven.
\end{proof}

\begin{lemma}
\label{lem:r=4:k2k4}
    Equation \eqref{eq:toshow24} is true for~$k_2,k_4 > 1$.
\end{lemma}

\begin{proof}
    Let be~$k_1,k_2,k_3,k_4\in\Z_{>0}$ and write~$w = k_1 + k_2 + k_3 + k_4 + 4$. By \eqref{eq:kkkk00=0}, \eqref{eq:1-2111new1}, and \eqref{eq:k4:2211=0}, we have
    \begin{align}
        \eqlabel{eq:1001-1111}
    0\equiv\ &\fz{\tau(u_1 u_3)\ast\tau\left(u_{k_1} u_{k_2} u_{k_3} u_{k_4}\right)}&&\mod\F{2,4,w}
    \\
    \equiv\ &\fz{u_1 u_0 u_0 u_1 \ast u_1 u_0^{k_4-1} u_1 u_0^{k_3-1} u_1 u_0^{k_2-1} u_1 u_0^{k_1-1}} &&\mod\F{2,4,w}
    \\ \nonumber
    \equiv\ & k_4k_2\fz{u_2 u_0^{k_4} u_1 u_0^{k_3-1} u_1 u_0^{k_2} u_2 u_0^{k_1-1}} 
    \\
    &+ \fz{u_1 u_0^{k_4-1} \left(u_1 u_0 u_0 u_1 \ast u_1 u_0^{k_3-1} u_1 u_0^{k_2-1} u_1 u_0^{k_1-1}\right)}
    \\
    &+ \fz{u_2 u_0^{k_4-1}u_1 \left(u_0 u_0 u_1 \ast u_0^{k_3-1} u_1 u_0^{k_2-1} u_1 u_0^{k_1-1}\right)} &&\mod\F{2,4,w}.
    \end{align}
    Now, by \eqref{eq:1-1211new3}, \eqref{eq:1-1112new1}, \eqref{eq:11-1111new1}, and \eqref{eq:k2:1212=0}, the latter is, modulo~$\F{2,4,w}$, congruent
    \begin{align}
    & k_4 k_2 \fz{u_{k_1} u_0 u_{k_2+1} u_{k_3} u_{k_4+1} u_0} - \binom{k_3+1}{2} \fz{u_{k_1} u_{k_2} u_{k_3+2} u_0 u_{k_4} u_0}
    \\
    &+ k_3 k_2 \fz{u_{k_1} u_0 u_{k_2+1} u_{k_3+1} u_{k_4} u_0} - \binom{k_2+1}{2} \fz{u_{k_1} u_0 u_0 u_{k_2+2} u_{k_3} u_{k_4}}.
    \end{align}
    Using \eqref{eq:10-1121new1}, \eqref{eq:10-1211new1}, \eqref{eq:10-2111new1}, and \eqref{eq:20-1111new1}, the latter is, modulo~$\F{2,4,w}$, congruent
    \begin{align}
        &- k_2 k_3 \fz{u_{k_1} u_{k_2+1} u_0 u_0 u_{k_3+1} u_{k_4}}
        - \frac12 k_2 k_3 \fz{u_{k_1} u_0 u_{k_2+1} u_{k_3+1} u_{k_4} u_0}
    \\ \nonumber
    &+ k_3 k_2 \fz{u_{k_1} u_0 u_{k_2+1} u_{k_3+1} u_{k_4} u_0}
    - \frac12 k_2 k_3 \fz{u_{k_1} u_{k_2+1} u_{k_3+1} u_0 u_0 u_{k_4}}
    \\ \nonumber
    \equiv\ & k_2 k_3 \left(- \fz{u_{k_1} u_{k_2+1} u_0 u_0 u_{k_3+1} u_{k_4}}
    -  \frac12 \fz{u_{k_1} u_{k_2+1} u_{k_3+1} u_0 u_0 u_{k_4}}
    \right.
    \\ \nonumber 
    &\left. + \frac12 \fz{u_{k_1} u_0 u_{k_2+1} u_{k_3+1} u_{k_4} u_0}
    \right) &\hspace{-2.3cm}\mod\F{2,4,w}.
    \end{align}
    With \eqref{eq:1-1211new3}, \eqref{eq:1-1121new1}, \eqref{eq:11-1111new1}, and \eqref{eq:k2:1221=0}, one obtains so
    \begin{align}
    \label{k2k3:2121=0}
        0\equiv\ & \fz{u_{k_1} u_{k_2+1} u_0 u_{k_3+1} u_{k_4} u_0}
        \mod\F{2,4,w}.
    \end{align}
    Now, this, together with \eqref{eq:1-1121new1} and \eqref{eq:k2:1221=0} imply
    \begin{align}
    \label{k2k3:1131=0}
        0\equiv\ & \fz{u_{k_1} u_{k_2+1} u_0 u_0 u_{k_3+1} u_{k_4}}
        \mod\F{2,4,w}.
    \end{align}
    Furthermore, \eqref{eq:10-1121new1} and \eqref{k2k3:1131=0} imply
    \begin{align}
        \label{k2k4:2121=0}
        0\equiv\ & \fz{u_{k_1} u_{k_2+1} u_0 u_{k_3} u_{k_4+1} u_0}
        \mod\F{2,4,w}.
    \end{align}
    Note that by Lemma~\ref{lem:r=42first}, Corollary~\ref{cor:k4cons}, and the congruences in Lemma~\ref{lem:r=4:dual}, the claim follows. 
\end{proof}

\begin{proposition}
\label{r=4:k1k3k4}
    Equation \eqref{eq:toshow24} is true for~$k_1,k_3,k_4 > 1$.
\end{proposition}

\begin{proof}
    Let be~$k_1,k_2,k_3,k_4\in\Z_{>0}$ with~$k_1,k_3,k_4>1$ and write~$w = k_1 + k_2 + k_3 + k_4 + 2$. For all~$z_2,z_3,z_4\geq 0$ with~$z_2+z_3+z_4 = 2$, using Theorem~\ref{thm:r=32} in the first step and Lemma~\ref{lem:r=4:k2k4} additionally in the second step, we have
    \begin{align*}
        0 \equiv\ \fz{u_{k_1}\ast u_{k_2} u_0^{z_2} u_{k_3} u_0^{z_3} u_{k_4} u_0^{z_4}} 
        \equiv\ \fz{u_{k_1} u_{k_2} u_0^{z_2} u_{k_3} u_0^{z_3} u_{k_4} u_0^{z_4}} \mod\F{2,4,w}.
    \end{align*}
    Using this observation, for~$z_1\geq 1$,~$z_2,z_3,z_4\geq 0$ with~$z_1+\cdots +z_4 = 2$, we have, using Corollary~\ref{cor:z=1} in the first step due to~$z_2+z_3+z_4\leq 1$,
    \begin{align*}
        0\equiv\ & \fz{u_{z_1}\ast \tau\left(u_{k_1} u_{k_2} u_0^{z_2} u_{k_3} u_0^{z_3} u_{k_4} u_0^{z_4}\right)}&&\mod\F{2,4,w}
        \\
        \equiv \ & \fz{u_{z_1} \ast u_{z_4+1} u_0^{k_4-1} u_{z_3+1} u_0^{k_3-1} u_{z_2+1} u_0^{k_2-1} u_1 u_0^{k_1-1}} &&\mod\F{2,4,w}
        \\
        \equiv\ & \fz{u_{z_4+1} u_0^{k_4-1} u_{z_3+1} u_0^{k_3-1} u_{z_2+1} u_0^{k_2-1} u_{z_1+1} u_0^{k_1-1}} &&\mod\F{2,4,w}
        \\
        \equiv\ & \fz{u_{k_1} u_0^{z_1} u_{k_2} u_0^{z_2} u_{k_3} u_0^{z_3} u_{k_4} u_0^{z_4}} &&\mod\F{2,4,w}.
    \end{align*}
    This completes the proof of the proposition.
\end{proof}

\begin{proposition}
    \label{r=4:k1k2k3}
    Equation \eqref{eq:toshow24} is true for~$k_1,k_2,k_3 > 1$.
\end{proposition}

\begin{proof}
    Let be~$k_1,k_2,k_3,k_4\in\Z_{>0}$ with~$k_1,k_2,k_3>1$ and write~$w = k_1 + k_2 + k_3 + k_4 + 2$. Using Lemma~\ref{lem:r=4:k2k4} and Proposition~\ref{r=4:k1k3k4}, we obtain for~$z_1,z_2,z_4\geq 0$ with~$z_1+z_2+z_4=2$ that
    \begin{align*}
        0\equiv\ \fz{u_{k_4} u_0^{z_4} \ast u_{k_1} u_0^{z_1} u_{k_2} u_0^{z_2} u_{k_3}}
        \equiv\ \fz{u_{k_1} u_0^{z_1} u_{k_2} u_0^{z_2} u_{k_3}u_{k_4} u_0^{z_4}}  \mod\F{2,4,w},
    \end{align*}
    where we used Proposition~\ref{prop:dep1explicit} and Proposition~\ref{prop:r=3} for the first congruence. Now, for all~$z_1,\dots,z_4\geq 0$ with~$z_1+\cdots+z_4 = 2$ and~$z_3 > 0$, we have
    \begin{align*}
        0\equiv\ & \fz{u_{z_3}\ast \tau\left(u_{k_1} u_0^{z_1} u_{k_2} u_0^{z_2} u_{k_3} u_{k_4} u_0^{z_4}\right)}&&\mod\F{2,4,w}
        \\
        \equiv\ & \fz{u_{z_3} \ast u_{z_4+1} u_0^{k_4-1} u_{1} u_0^{k_3-1} u_{z_2+1} u_0^{k_2-1} u_{z_1+1} u_0^{k_1-1}} &&\mod\F{2,4,w}
        \\
        \equiv\ & \fz{u_{z_4+1} u_0^{k_4-1} u_{z_3+1} u_0^{k_3-1} u_{z_2+1} u_0^{k_2-1} u_{z_1+1} u_0^{k_1-1}} &&\mod\F{2,4,w}
        \\
        \equiv\ & \fz{u_{k_1} u_0^{z_1} u_{k_2} u_0^{z_2} u_{k_3} u_0^{z_3} u_{k_4} u_0^{z_4}} &&\mod\F{2,4,w}.
    \end{align*}
    This completes the proof of the proposition.
\end{proof}

Lemma~\ref{lem:r=4:k2k4} and Propositions~\ref{r=4:k1k3k4} and~\ref{r=4:k1k2k3}, show that Theorem~\ref{thm:r=42} is true when three of the~$k_j$ are larger than~$1$. Hence, in the following, we will prove the remaining cases that two of the~$k_j$'s are larger~$1$.

\begin{lemma}
\label{lem:r=4:k3k4}
    Equation \eqref{eq:toshow24} is true for~$k_3,k_4>1$.
\end{lemma}

\begin{proof}
    Let be~$k_1,k_2,k_3,k_4\in\Z_{>0}$ with~$k_3,k_4>1$ and write~$w = k_1 + k_2 + k_3 + k_4 + 2$. According to Lemma~\ref{lem:r=4:k2k4} and Proposition~\ref{r=4:k1k3k4}, we may assume~$k_1=k_2=1$. Using Proposition~\ref{prop:r=3} for the first two steps in the following calculation, while using~\eqref{eq:k0k0kk=0}, \eqref{eq:1-1211new3}, and \eqref{eq:k2:1212=0} for the last step, we have
    \begin{align}
        \nonumber
        0\equiv\ &  \fz{u_1\ast u_1 u_0 u_{k_3} u_0 u_{k_4}} &&\mod\F{2,4,w}
        \\ \nonumber
        \equiv\ & 2 \fz{u_1 u_1 u_0 u_{k_3} u_0 u_{k_4}} + \fz{u_1 u_0 u_1 u_{k_3} u_0 u_{k_4}} \\ \nonumber &+ \fz{u_1 u_0 u_{k_3} u_1 u_0 u_{k_4}} 
        + \fz{u_1 u_0 u_{k_3} u_0 u_1 u_{k_4}} \\ &+ \fz{u_1 u_0 u_{k_3} u_0 u_{k_4} u_1} &&\mod\F{2,4,w}
        \\
        \label{k3k4:1221=0}
        \equiv\ & \fz{u_1 u_1 u_0 u_{k_3} u_0 u_{k_4}} &&\mod\F{2,4,w}.\qquad
    \end{align}
    This implies, with \eqref{eq:1-1211new3} again,
    \begin{align}
        \label{k3k4:1212=0}
        0 \equiv\ & \fz{u_1 u_0 u_1 u_{k_3} u_0 u_{k_4}} \mod\F{2,4,w}.
    \end{align}
    Now, using Proposition~\ref{prop:r=3} for the first step, then using \eqref{eq:k0k0kk=0}, \eqref{eq:k2:1212=0}, and Lemma~\ref{lem:r=4:k2k4} for the second step, then applying \eqref{k3k4:1221=0} and \eqref{k3k4:1212=0}, we obtain
    \begin{align}
        \nonumber
        0\equiv\ & \frac14\left(\fz{u_1 u_0 \ast u_1 u_0 u_{k_3} u_{k_4}} -  \fz{u_1 \ast u_1 u_0 u_{k_3} u_{k_4} u_0}\right) &&\mod\F{2,4,w}
        \\ \nonumber
        \equiv\ & \fz{u_1 u_1 u_0 u_0 u_{k_3} u_{k_4}} + \frac12 \fz{u_1 u_1 u_0 u_{k_3} u_0 u_{k_4}} \\ &+ \frac14 \fz{u_1 u_0 u_1 u_{k_3} u_0 u_{k_4}} &&\mod\F{2,4,w}
        \\ \label{k3k4:1131=0}
        \equiv\ & \fz{u_1 u_1 u_0 u_0 u_{k_3} u_{k_4}} &&\mod\F{2,4,w}.
    \end{align}
    The lemma follows using the relations in Lemma~\ref{lem:r=4:dual}.
\end{proof}

\begin{lemma}
\label{lem:r=4:k2k3}
    Equation~\eqref{eq:toshow24} is true for~$k_2,k_3 > 1$.
\end{lemma}

\begin{proof}
    Let be~$k_1,k_2,k_3,k_4\in\Z_{>0}$ with~$k_2,k_3>1$ and write~$w = k_1 + k_2 + k_3 + k_4 + 2$. According to Proposition~\ref{r=4:k1k2k3} and Lemma~\ref{lem:r=4:k2k4}, we may assume~$k_1=k_4=1$. Using Proposition~\ref{prop:r=3} for the first step and Lemmas~\ref{lem:r=4:k2k4} and~\ref{lem:r=4:k3k4} for the second one, we obtain
    \begin{align}
        \label{k2k3:1113=0}
        0\equiv\ & \fz{u_1\ast u_1 u_0 u_0 u_{k_2} u_{k_3}}
        \equiv\, \fz{u_1 u_0 u_0 u_{k_2} u_{k_3} u_1}  \mod\F{2,4,w},
    \end{align}
    giving by \eqref{eq:2-1111new1}, respectively by \eqref{eq:1-1112new1} and \eqref{eq:k2:1212=0},
    \begin{align}
        \label{k2k3:1311=0}
        0\equiv\ &  \fz{u_1 u_{k_2} u_{k_3} u_0 u_0 u_1} \mod\F{2,4,w},
        \\
        \label{k2k3:2112=0}
        0\equiv\ &  \fz{u_1 u_0 u_{k_2} u_{k_3} u_1 u_0} \mod\F{2,4,w}.
    \end{align}
    Note that \eqref{k2k3:1311=0} implies by \eqref{eq:kkk0k0+kkk00k=0}
    \begin{align}
        \label{k2k3:2211=0}
        0\equiv\ &  \fz{u_1 u_{k_2} u_{k_3} u_0 u_1 u_0} \mod\F{2,4,w},
    \end{align}
    completing, together with \eqref{eq:kkkk00=0}, \eqref{eq:k0k0kk=0}, \eqref{eq:k2:1212=0}, \eqref{eq:k2:1221=0}, \eqref{k2k3:2121=0}, and \eqref{k2k3:1131=0}, the proof of the lemma.
\end{proof}

\begin{lemma}
\label{lem:r=4:k1kj}
    Equation \eqref{eq:toshow24} is true for~$k_1>1$ and one of~$k_2,k_3,k_4$ larger~$1$.
\end{lemma}

\begin{proof}
    Let be~$k_1,k_2,k_3,k_4\in\Z_{>0}$ with~$k_1>1$ and write~$w = k_1 + k_2 + k_3 + k_4 + 2$. First, assume that one of~$k_3,k_4$ larger~$1$ as well. For~$z_2,z_3,z_4\geq 0$ with~$z_2+z_3+z_4 = 2$, we have, using Proposition~\ref{prop:r=3} in the first step and Lemmas~\ref{lem:r=4:k2k4},~\ref{lem:r=4:k3k4}, and~\ref{lem:r=4:k2k3} for the second one,
    \begin{align}
        0\equiv \fz{u_{k_1}\ast u_{k_2} u_0^{z_2} u_{k_3} u_0^{z_3} u_{k_4} u_0^{z_4}} \equiv\ \fz{u_{k_1} u_{k_2} u_0^{z_2} u_{k_3} u_0^{z_3} u_{k_4} u_0^{z_4}} \mod\F{2,4,w}.
    \end{align}
    Now, for all~$z_1>0,z_2,z_3,z_4\geq 0$ with~$z_1+\cdots+z_4 = 2$, using Corollary~\ref{cor:z=1}, we obtain
    \begin{align}
        0\equiv\ & \fz{u_{z_1} \ast \tau\left(u_{k_1} u_{k_2} u_0^{z_2} u_{k_3} u_0^{z_3} u_{k_4} u_0^{z_4}\right)} &&\mod\F{2,4,w}
        \\
        \equiv \ & \fz{u_{z_1} \ast u_{z_4+1} u_0^{k_4-1} u_{z_3+1} u_0^{k_3-1} u_{z_2+1} u_0^{k_2-1} u_{1} u_0^{k_1-1} } &&\mod\F{2,4,w}
        \\
        \equiv\ & \fz{u_{z_4+1} u_0^{k_4-1} u_{z_3+1} u_0^{k_3-1} u_{z_2+1} u_0^{k_2-1} u_{z_1+1} u_0^{k_1-1}} &&\mod\F{2,4,w}
        \\
        \equiv\ & \fz{u_{k_1} u_0^{z_1} u_{k_2} u_0^{z_2} u_{k_3} u_0^{z_3} u_{k_4} u_0^{z_4}} &&\mod\F{2,4,w},
    \end{align}
    showing that \eqref{eq:toshow24} holds for~$k_1,k_3 > 1$, and for~$k_1,k_4 > 1$ as well.

    It remains considering the case of~$k_1,k_2 > 1$ with~$k_3,k_4\in\Z_{>0}$ arbitrary. Note that for~$z_3,z_4\geq 0$ with~$z_3+z_4 = 2$, we have by the previous results of this proof and Lemmas~\ref{lem:r=4:k2k4},~\ref{lem:r=4:k3k4}, and~\ref{lem:r=4:k2k3},
    \begin{align}
        0\equiv\ &\fz{u_{k_3} u_0^{z_3} u_{k_4} u_0^{z_4} \ast u_{k_1} u_{k_2}}
        \equiv\, \fz{u_{k_1} u_{k_2} u_{k_3} u_0^{z_3} u_{k_4} u_0^{z_4}} \mod\F{2,4,w}.
    \end{align}
    By Corollary~\ref{cor:z=1} for the first congruence and for the second, again, by the previous results of this proof and Lemmas~\ref{lem:r=4:k2k4},~\ref{lem:r=4:k3k4}, and~\ref{lem:r=4:k2k3}, we have
    \begin{align}
        0\equiv\ &\fz{u_{k_3} u_{k_4} u_0 \ast u_{k_1} u_0 u_{k_2}}
        \equiv\, \fz{u_{k_1} u_0 u_{k_2} u_{k_3} u_{k_4} u_0} \mod\F{2,4,w}.
    \end{align}
    Using the previous results of this proof and \eqref{eq:kkkk00=0}, \eqref{eq:k0k0kk=0}, \eqref{eq:k2:1221=0}, \eqref{eq:k2:1212=0}, and Lemma~\ref{lem:r=4:dual}, we obtain that \eqref{eq:toshow24} also holds true for~$k_1,k_2 > 1$, completing the proof.
\end{proof}

As in the proof of Theorem~\ref{thm:r=42} mentioned, for completing the proof of Theorem~\ref{thm:r=42}, it remains to consider the cases where one of the~$k_j$'s is larger~$1$ while the other three equal~$1$.

\begin{lemma}
    \label{lem:r=4:k3}
    Equation~\eqref{eq:toshow24} is true for~$k_3 > 1$.
\end{lemma}

\begin{proof}
     Let be~$k_1,k_2,k_3,k_4\in\Z_{>0}$ with~$k_3>1$ and write~$w = k_1 + k_2 + k_3 + k_4 + 2$. According to Lemmas~\ref{lem:r=4:k3k4},~\ref{lem:r=4:k2k3},~\ref{lem:r=4:k1kj}, we may assume~$k_1 = k_2 = k_4 = 1$, i.e.,~$w = k_3 + 5$. Using Proposition~\ref{prop:r=3} for the first congruence, Corollary~\ref{cor:z=1} and Proposition~\ref{prop:r=3} for the second one, and \eqref{eq:k0k0kk=0}, \eqref{eq:2-1111new1}, and \eqref{eq:k4:1311=0} for the third one, we have
    \begin{align}
        \nonumber
        0\equiv\ & \fz{u_1 \ast u_1 u_0 u_0 u_1 u_{k_3}} &&\mod\F{2,4,w}\quad
        \\ \nonumber
        \equiv\ & 2 \fz{u_1 u_1 u_0 u_0 u_1 u_{k_3}} + \fz{u_1 u_0 u_1 u_0 u_1 u_{k_3}} 
        \\
        &+ 2 \fz{u_1 u_0 u_0 u_1 u_1 u_{k_3}} + \fz{u_1 u_0 u_0 u_1 u_{k_3} u_1} &&\mod\F{2,4,w}\quad
        \\
        \label{k31:1113=0}
        \equiv\ & \fz{u_1 u_0 u_0 u_1 u_{k_3} u_1} &&\mod\F{2,4,w}.\qquad 
    \end{align}
    Furthermore, using Proposition~\ref{prop:r=3} for the first congruence, Corollary~\ref{cor:z=1}, Proposition~\ref{prop:r=3} and Equations \eqref{eq:k0k0kk=0}, \eqref{eq:1-1211new3}, and \eqref{eq:k2:1212=0} for the second congruence, gives
    \begin{align}
        \label{k31:1221=0}
        0\equiv\ \fz{u_1\ast u_1 u_0 u_{k_3} u_0 u_1} 
        \equiv\ \fz{u_1 u_1 u_0 u_{k_3} u_0 u_1} \mod\F{2,4,w},
    \end{align}
    and so, by \eqref{eq:1-1211new3} again,
    \begin{align}
        \label{k31:1212=0}
        0 \equiv\ & \fz{u_1 u_0 u_1 u_{k_3} u_0 u_1} \mod\F{2,4,w}.
    \end{align}
    Now, \eqref{eq:1-1112new1} in combination with \eqref{k31:1113=0} and \eqref{k31:1212=0} implies
    \begin{align}
        \label{k31:2112=0}
        0 \equiv\ & \fz{u_1 u_0 u_1 u_{k_3} u_1 u_0} \mod\F{2,4,w}.
    \end{align}
    Using Corollary~\ref{cor:z=1} for the first congruence and, for the second one, Corollary~\ref{cor:z=1}, Propositions~\ref{prop:r=2} and~\ref{prop:r=3} and Equations~\eqref{eq:1-2111new1}, \eqref{k4:2221=0}, \eqref{eq:k0k0kk=0}, \eqref{k31:1221=0}, and \eqref{k31:1212=0}, we obtain
    \begin{align}
        \nonumber
        0\equiv\ & \frac14\left( \fz{u_1 u_0 u_1\ast u_1 u_0 u_{k_3}} - \fz{u_1 u_0 \ast u_1 u_0 u_1 u_{k_3}}\right) &&\mod\F{2,4,w}
        \\
        \label{k31:1131=0}
        \equiv\ & \fz{u_1 u_1 u_0 u_0 u_{k_3} u_1} &&\mod\F{2,4,w}.\qquad\ \ 
    \end{align}
    The remaining proof follows directly from Lemma~\ref{lem:r=4:dual}.
\end{proof}

\begin{lemma}
     \label{lem:r=4:k4}
    Equation~\eqref{eq:toshow24} is true for~$k_4 > 1$.
\end{lemma}

\begin{proof}
     Let be~$k_1,k_2,k_3,k_4\in\Z_{>0}$ with~$k_4>1$ and write~$w = k_1 + k_2 + k_3 + k_4 + 2$. According to Lemmas~\ref{lem:r=4:k3k4},~\ref{lem:r=4:k2k4},~\ref{lem:r=4:k1kj}, we may assume~$k_1 = k_2 = k_3 = 1$, i.e.,~$w = k_4 + 5$. Using Corollary~\ref{cor:z=1} for the first congruence, and for the second one, Corollary~\ref{cor:z=1}, Proposition~\ref{prop:r=3}, and Equations \eqref{eq:11-1111new1}, \eqref{eq:k4:2211=0}, and \eqref{k31:2112=0}, we have
    \begin{align}
        \label{k41:1131=0}
        0\equiv\ & \frac14\left( \fz{u_1 u_0 \ast u_1 u_0 u_1 u_{k_4}} \right)
        \equiv\  \fz{u_1 u_1 u_0 u_0 u_1 u_{k_4}} \mod\F{2,4,w}.
    \end{align}
    This, \eqref{eq:k4:1311=0}, and \eqref{k31:1131=0}, gives, together with Proposition~\ref{prop:r=3},
    \begin{align}
        \label{k41:1221=0}
        0\equiv\ & \fz{u_1 \ast u_1 u_1 u_0 u_0 u_{k_4}}
        \equiv\  \fz{u_1 u_1 u_0 u_1 u_0 u_{k_4}} \mod\F{2,4,w}.
    \end{align}
    The remaining part of the proof follows from Equations~\eqref{eq:kkkk00=0}, \eqref{eq:k0k0kk=0}, \eqref{eq:k4:2211=0}, \eqref{eq:k4:1311=0}, and Lemma~\ref{lem:r=4:dual}.
\end{proof}

\begin{lemma}
     \label{lem:r=4:k2}
    Equation~\eqref{eq:toshow24} is true for~$k_2 > 1$.
\end{lemma}

\begin{proof}
     Let be~$k_1,k_2,k_3,k_4\in\Z_{>0}$ with~$k_2>1$ and write~$w = k_1 + k_2 + k_3 + k_4 + 2$. According to Lemmas~\ref{lem:r=4:k2k4},~\ref{lem:r=4:k2k3},~\ref{lem:r=4:k1kj}, we may assume~$k_1 = k_3 = k_4 = 1$, i.e.,~$w = k_2 + 5$. First note that, by Proposition~\ref{prop:r=3} and Lemma~\ref{lem:r=4:k3}, one has
    \begin{align}
        \label{k21:1113=0}
        0\equiv\ & \frac12 \fz{u_1 \ast u_1 u_0 u_0 u_{k_2} u_1}
        \equiv\  \fz{u_1 u_0 u_0 u_{k_2} u_1 u_1} \mod\F{2,4,w},
    \end{align}
    giving, with \eqref{eq:1-1112new1} and \eqref{eq:k2:1212=0},
    \begin{align}
        \label{k21:2112=0}
        0 \equiv\ & \fz{u_1 u_0 u_{k_2} u_1 u_1 u_0} \mod\F{2,4,w},
    \end{align}
    Furthermore, by Proposition~\ref{prop:r=3} for the first congruence, Corollary~\ref{cor:z=1}, Proposition~\ref{prop:r=3} and Lemma~\ref{lem:r=4:k3}, Equations~\eqref{eq:1-2111new1}, \eqref{eq:k2:1221=0}, and \eqref{k21:2112=0} for the second congruence, we obtain
    \begin{align}
        \label{k21:2121=0}
        0\equiv\ \fz{u_1 \ast u_1 u_{k_2} u_0 u_1 u_0}
        \equiv\ \fz{u_1 u_{k_2} u_0 u_1 u_1 u_0}\mod\F{2,4,w}.
    \end{align}
    The remaining part of the proof follows from \eqref{eq:kkkk00=0}, \eqref{eq:k0k0kk=0}, \eqref{eq:k2:1221=0}, \eqref{eq:k2:1212=0}, and Lemma~\ref{lem:r=4:dual}, immediately.
\end{proof}

\begin{lemma}
     \label{lem:r=4:k1}
    Equation~\eqref{eq:toshow24} is true for~$k_1 > 1$.
\end{lemma}

\begin{proof}
     Let be~$k_1,k_2,k_3,k_4\in\Z_{>0}$ with~$k_1>1$ and write~$w = k_1 + k_2 + k_3 + k_4 + 2$. According to Lemma~\ref{lem:r=4:k1kj}, we may assume~$k_2 = k_3 = k_4 = 1$, i.e.,~$w = k_1 + 5$. For any~$z_2,z_3,z_4\geq 0$ with~$z_2+z_3+z_4=2$, we have, using Corollary~\ref{cor:z=1}, Proposition~\ref{prop:r=3}, and Lemmas~\ref{lem:r=4:k3},~\ref{lem:r=4:k4}, and~\ref{lem:r=4:k2} for the third congruence,
    \begin{align}
        \nonumber
        0\equiv\ &\fz{u_{k_1}\ast\tau\left(u_{z_4+1} u_{z_3+1} u_{z_2+1}\right)}  &&\mod\F{2,4,w}
        \\
        \equiv\ &\fz{u_{k_1} \ast u_1 u_0^{z_2} u_1 u_0^{z_3} u_1 u_0^{z_4}}  &&\mod\F{2,4,w}
        \\
        \equiv\ &\fz{u_{k_1} u_1 u_0^{z_2} u_1 u_0^{z_3} u_1 u_0^{z_4}} &&\mod\F{2,4,w}.
    \end{align}
    This implies, for any~$z_1>0,z_2,z_3,z_4\geq 0$ with~$z_1+\cdots +z_4=2$, using Proposition~\ref{prop:r=3} for the first congruence and, additionally, Corollary~\ref{cor:z=1} for the third congruence,
    \begin{align*}
        0\equiv\ &\fz{u_{z_1}\ast\tau\left(u_{k_1} u_1 u_0^{z_2} u_1 u_0^{z_3} u_1 u_0^{z_4}\right)} &&\mod\F{2,4,w}
        \\
        \equiv\ &\fz{u_{z_1}\ast u_{z_4+1} u_{z_3+1} u_{z_2+1} u_1 u_0^{k_1-1}} &&\mod\F{2,4,w}
        \\
        \equiv \ &\fz{u_{z_4+1} u_{z_3+1} u_{z_2+1} u_{z_1+1} u_0^{k_1-1}} &&\mod\F{2,4,w}
        \\
        \equiv\ &\fz{u_{k_1} u_0^{z_1} u_1 u_0^{z_2} u_1 u_0^{z_3} u_1 u_0^{z_4}}&&\mod\F{2,4,w},
    \end{align*}
    completing the proof of the lemma.
\end{proof}

\subsection{The refined Bachmann Conjecture~\ref{conj:mdbdstrongApp} for~$(z,\dd,w) = (3,4,w)$}
\label{sec:refBazd=34}


\begin{theorem}
\label{thm:r=43}
    The refined Bachmann Conjecture~\ref{conj:mdbdstrongApp} is true for all~$(3,4,w)\in\Z_{>0}^3$, i.e.,
    \begin{align}
    \eqlabel{eq:toshow34}        \fz{u_{k_1}u_0^{z_1}u_{k_2}u_0^{z_2}u_{k_3}u_0^{z_3}u_{k_4}u_0^{z_4}}\in\F{3,4,w}
    \end{align}
    for all integers~$k_j\in\Z_{>0},\, z_j\in\Z_{\geq 0}$, for~$1\leq j\leq 4$, satisfying $z_1+z_2+z_3+z_4 = 3$ and~$w = k_1+k_2+k_3+k_4+3$.
\end{theorem}

\begin{proof}
    In the case~$k_1=k_2=k_3=k_4 = 1$,~\eqref{eq:toshow34} is true since, by~$\tau$-invariance of $\zeta_q^{\mathrm{f}}$, for any~$z_1,\dots,z_4\geq 0$, we have
    \begin{align*}
        \fz{u_{1} u_0^{z_1} u_{1} u_0^{z_2} u_{1} u_0^{z_3} u_{1} u_0^{z_4}} \equiv \fz{u_{z_4+1} u_{z_3+1} u_{z_2+1} u_{z_1+1}} \in\Zqz.
    \end{align*}
    For~$k_3>1$,~\eqref{eq:toshow34} will follow from Lemma~\ref{lem:k3done}, for $k_4>1$,~\eqref{eq:toshow34} will follow from Lemma~\ref{lem:k4done}, for $k_2>1$,~\eqref{eq:toshow34} will follow from Lemma~\ref{lem:k2done}, and for $k_1>1$,~\eqref{eq:toshow34} will follow from Lemma~\ref{lem:k1done}. This completes the proof of the theorem.
\end{proof}

First, we will consider some relations we need more than once.


\begin{lemma}
    \label{lem:r=43first}
Let be~$k_1,\dots,k_4\in\Z_{>0}$ and write~$w=k_1+\cdots+k_4+3$. We have
\begin{align}
    \eqlabel{eq:kkkk000=0}
    0\equiv\ & \fz{u_{k_1} u_{k_2} u_{k_3} u_{k_4} u_0 u_0 u_0} \mod\F{3,4,w},
    \\
    \eqlabel{eq:k0k0k0k=0}
    0\equiv\ & \fz{u_{k_1} u_0 u_{k_2} u_0 u_{k_3} u_0 u_{k_4}} \mod\F{3,4,w}.
\end{align}
\end{lemma}
\begin{proof}
    Congruence \eqref{eq:kkkk000=0} is a special case of Corollary~\ref{cor:z1111111=0}. Setting~$\mathbf{k} := (k_1,\dots,k_4)$, \eqref{eq:k0k0k0k=0} follows from Lemma~\ref{lem:mainideaboxApp} and~\eqref{eq:kkkk000=0} via
    \begin{align*}
    \fz{\Psi_{\mathbf{k}}\left(\uu_1\uu_2^3\right)} \equiv\ \sum\limits_{j=0}^3 (-1)^{j-1} \fz{\Psi_{\mathbf{k}}\left(\uu_1^j\boxast \uu_{4-j}\uu_1^3\right)}\equiv\ 0 
    \mod\F{3,4,w}.\tag*{\qedhere}
\end{align*}
\end{proof}

Next, we consider relations coming from products with no~$u_0$ in one of the factors. 

\begin{lemma}
\label{lem:r=43:dual}
Let be~$k_1,\dots,k_4\in\Z_{>0}$ and write~$w=k_1+\cdots+k_4+3$. We have
\begin{align}
    \begin{split}
    \eqlabel{eq:1-3111}
    0 
    \equiv \ & \fz{u_{k_1} u_{k_2} u_{k_3} u_0 u_{k_4} u_0 u_0} + \fz{u_{k_1} u_{k_2} u_0 u_{k_3} u_{k_4} u_0 u_0} \\ &+ \fz{u_{k_1} u_0 u_{k_2} u_{k_3} u_{k_4} u_0 u_0} \hspace{46.6mm}\mod \F{3,4,w},
    \end{split}
    \\
    \eqlabel{eq:1-1113}
    \begin{split}
    0\equiv\ & \fz{u_{k_1} u_0 u_0 u_{k_2} u_{k_3} u_{k_4} u_0} + \fz{u_{k_1} u_0 u_0 u_{k_2} u_{k_3} u_0 u_{k_4}} \\  &+ \fz{u_{k_1} u_0 u_0 u_{k_2} u_0 u_{k_3} u_{k_4}} + \fz{u_{k_1} u_0 u_0 u_0 u_{k_2} u_{k_3} u_{k_4}} \mod \F{3,4,w},
    \end{split}
    \\
    \begin{split}
    \eqlabel{eq:1-2211}
    0\equiv\ & \fz{u_{k_1} u_{k_2} u_{k_3} u_0 u_{k_4} u_0 u_0} + \fz{u_{k_1} u_{k_2} u_{k_3} u_0 u_0 u_{k_4} u_0} \\  &+ \fz{u_{k_1} u_{k_2} u_0 u_{k_3} u_0 u_{k_4} u_0} + \fz{u_{k_1} u_0 u_{k_2} u_{k_3} u_0 u_{k_4} u_0}  \mod \F{3,4,w},
    \end{split}
    \\
    \begin{split}
    \eqlabel{eq:1-2112}
    0\equiv\ & \fz{u_{k_1} u_0 u_{k_2} u_{k_3} u_{k_4} u_0 u_0} + \fz{u_{k_1} u_0 u_{k_2} u_{k_3} u_0 u_{k_4} u_0} \\ &+  \fz{u_{k_1} u_0 u_{k_2} u_0 u_{k_3} u_{k_4} u_0} + \fz{u_{k_1} u_0 u_0 u_{k_2} u_{k_3} u_{k_4} u_0} \mod \F{3,4,w},
    \end{split}
    \\
    \begin{split}
    \eqlabel{eq:1-1212}
    0\equiv\ & \fz{u_{k_1} u_0 u_{k_2} u_{k_3} u_0 u_{k_4} u_0} + \fz{u_{k_1} u_0 u_{k_2} u_{k_3} u_0 u_0 u_{k_4}} \\  &+\fz{u_{k_1} u_0 u_{k_2} u_0 u_{k_3} u_0 u_{k_4}} + \fz{u_{k_1} u_0 u_0 u_{k_2} u_{k_3} u_0 u_{k_4}} \mod \F{3,4,w},
    \end{split}
    \\
    \begin{split}
    \eqlabel{eq:1-1122}
    0\equiv\ & \fz{u_{k_1} u_0 u_{k_2} u_0 u_{k_3} u_{k_4} u_0} + \fz{u_{k_1} u_0 u_{k_2} u_0 u_{k_3} u_0 u_{k_4}} \\  &+\fz{u_{k_1} u_0 u_{k_2} u_0 u_0 u_{k_3} u_{k_4}} + \fz{u_{k_1} u_0 u_0 u_{k_2} u_0 u_{k_3} u_{k_4}} \mod \F{3,4,w},
    \end{split}
    \\
    \begin{split}
    \eqlabel{eq:2-2111}
    0\equiv\ & \fz{u_{k_1} u_{k_2} u_{k_3} u_{k_4} u_0 u_0 u_0} + \fz{u_{k_1} u_{k_2} u_{k_3} u_0 u_0 u_{k_4} u_0} \\ &+ \fz{u_{k_1} u_{k_2} u_0 u_0 u_{k_3} u_{k_4} u_0} + \fz{u_{k_1} u_0 u_0 u_{k_2} u_{k_3} u_{k_4} u_0} \mod\F{3,4,w},
    \end{split}
    \\
    \begin{split}
    \eqlabel{eq:2-1211}
    0\equiv\ & \fz{u_{k_1} u_{k_2} u_{k_3} u_0 u_{k_4} u_0 u_0} + \fz{u_{k_1} u_{k_2} u_{k_3} u_0 u_0 u_0 u_{k_4}} \\ &+ \fz{u_{k_1} u_{k_2} u_0 u_0 u_{k_3} u_0 u_{k_4}} + \fz{u_{k_1} u_0 u_0 u_{k_2} u_{k_3} u_0 u_{k_4}} \mod\F{3,4,w},
    \end{split}
     \\
    \begin{split}
    \eqlabel{eq:2-1121}
    0\equiv\ & \fz{u_{k_1} u_{k_2} u_0 u_{k_3} u_{k_4} u_0 u_0} + \fz{u_{k_1} u_{k_2} u_0 u_{k_3} u_0 u_0 u_{k_4}} \\ &+ \fz{u_{k_1} u_{k_2} u_0 u_0 u_0 u_{k_3} u_{k_4}} + \fz{u_{k_1} u_0 u_0 u_{k_2} u_0 u_{k_3} u_{k_4}} \mod\F{3,4,w},
    \end{split}
    \\
    \begin{split}
    \eqlabel{eq:2-1112}
    0\equiv\ & \fz{u_{k_1} u_0 u_{k_2} u_{k_3} u_{k_4} u_0 u_0} + \fz{u_{k_1} u_0 u_{k_2} u_{k_3} u_0 u_0 u_{k_4}} \\ &+ \fz{u_{k_1} u_0 u_{k_2} u_0 u_0 u_{k_3} u_{k_4}} + \fz{u_{k_1} u_0 u_0 u_0 u_{k_2} u_{k_3} u_{k_4}} \mod\F{3,4,w},
    \end{split}
    \\
    \begin{split}
    \eqlabel{eq:11-1112}
    0\equiv\ & \fz{u_{k_1} u_0 u_{k_2} u_{k_3} u_0 u_{k_4} u_0} + \fz{u_{k_1} u_0 u_{k_2} u_0 u_{k_3} u_{k_4} u_0}\\ &+ \fz{u_{k_1} u_0 u_0 u_{k_2} u_{k_3} u_{k_4} u_0} + \fz{u_{k_1} u_0 u_0 u_{k_2} u_0 u_{k_3} u_{k_4}} \\ &+ \fz{u_{k_1} u_0 u_0 u_0 u_{k_2} u_{k_3} u_{k_4}} \hspace{46.6mm}\mod\F{3,4,w},
    \end{split}
    \\
    \begin{split}
    \eqlabel{eq:3-1111}
    0\equiv\ & \fz{u_{k_1} u_{k_2} u_{k_3} u_0 u_0 u_0 u_{k_4}} + \fz{u_{k_1} u_{k_2} u_0 u_0 u_0  u_{k_3} u_{k_4}} \\ &+ \fz{u_{k_1} u_0 u_0 u_0 u_{k_2} u_{k_3} u_{k_4}} \hspace{46.6mm}\mod\F{3,4,w},
    \end{split}
    \\
    \begin{split}
    \eqlabel{eq:21-1111}
    0\equiv\ & \fz{u_{k_1} u_{k_2} u_{k_3} u_0 u_{k_4} u_0 u_0} + \fz{u_{k_1} u_{k_2} u_0 u_{k_3} u_{k_4} u_0 u_0} \\ &+ \fz{u_{k_1} u_0 u_{k_2} u_{k_3} u_{k_4} u_0 u_0} + \fz{u_{k_1} u_{k_2} u_0 u_{k_3} u_0 u_0 u_{k_4}}\\ &+ \fz{u_{k_1} u_0 u_{k_2} u_{k_3} u_0 u_0 u_{k_4}} + \fz{u_{k_1} u_0 u_{k_2} u_0 u_0 u_{k_3} u_{k_4}} \mod\F{3,4,w},
    \end{split}
    \\
    \begin{split}
    \eqlabel{eq:12-1111}
    0\equiv\ & \fz{u_{k_1} u_{k_2} u_{k_3} u_0 u_0 u_{k_4} u_0} + \fz{u_{k_1} u_{k_2} u_0 u_0 u_{k_3} u_{k_4} u_0} \\ &+ \fz{u_{k_1} u_0 u_0 u_{k_2} u_{k_3} u_{k_4} u_0} + \fz{u_{k_1} u_{k_2} u_0 u_0 u_{k_3} u_0 u_{k_4}}\\ &+ \fz{u_{k_1} u_0 u_0 u_{k_2} u_{k_3} u_0 u_{k_4}} + \fz{u_{k_1} u_0 u_0 u_{k_2} u_0 u_{k_3} u_{k_4}} \mod\F{3,4,w},
    \end{split}
    \\
    \begin{split}
    \eqlabel{eq:111-1111}
    0\equiv\ & \fz{u_{k_1} u_{k_2} u_0 u_{k_3} u_0 u_{k_4} u_0} + \fz{u_{k_1} u_0 u_{k_2} u_{k_3} u_0 u_{k_4} u_0}\\ &+ \fz{u_{k_1} u_0 u_{k_2} u_0 u_{k_3} u_{k_4} u_0} \hspace{46.6mm}\mod\F{3,4,w}.
    \end{split}
\end{align}
\end{lemma}

\begin{proof}
    All relations are a consequence of  Lemma~\ref{lem:r=43first} and, by Lemma~\ref{lem:mainideaboxApp},
    \begin{align*} 
        0\equiv \fz{\tau(\Psi_{\mathbf{k}}(\uu_{\mathbf{n}}\boxast \uu_{\boldsymbol{\ell}}))}\mod\F{3,4,w}
    \end{align*}
    with~$\mathbf{k} = (k_1,\dots,k_4)$ and~$(\mathbf{n},\boldsymbol{\ell})\in\indexset{3}{4}$ each. Precisely, we used~$(\mathbf{n},\boldsymbol{\ell}) = ((1),(3,1,1,1))$ for~\eqref{eq:1-3111}, ~$(\mathbf{n},\boldsymbol{\ell}) = ((1),(1,1,1,3))$ for~\eqref{eq:1-1113}, ~$(\mathbf{n},\boldsymbol{\ell}) = ((1),(2,2,1,1))$ for~\eqref{eq:1-2211}, ~$(\mathbf{n},\boldsymbol{\ell}) = ((1),(2,1,1,2))$ for~\eqref{eq:1-2112}, ~$(\mathbf{n},\boldsymbol{\ell}) = ((1),(1,2,1,2))$ for~\eqref{eq:1-1212}. Furthermore, we used~$(\mathbf{n},\boldsymbol{\ell}) = ((1),(1,1,2,2))$ for~\eqref{eq:1-1122}. Furthermore, we used~$(\mathbf{n},\boldsymbol{\ell}) = ((2),(2,1,1,1))$ for~\eqref{eq:2-2111}, ~$(\mathbf{n},\boldsymbol{\ell}) = ((2),(1,2,1,1))$ for~\eqref{eq:2-1211}, ~$(\mathbf{n},\boldsymbol{\ell}) = ((2),(1,1,2,1))$ for~\eqref{eq:2-1121}, ~$(\mathbf{n},\boldsymbol{\ell}) = ((2),(1,1,1,2))$ for~\eqref{eq:2-1112}, ~$(\mathbf{n},\boldsymbol{\ell}) = ((1,1),(1,1,1,2))$ for~\eqref{eq:11-1112}. Furthermore, we used~$(\mathbf{n},\boldsymbol{\ell}) = ((3),(1,1,1,1))$ for~\eqref{eq:3-1111}, ~$(\mathbf{n},\boldsymbol{\ell}) = ((2,1),(1,1,1,1))$ for~\eqref{eq:21-1111}, ~$(\mathbf{n},\boldsymbol{\ell}) = ((1,2),(1,1,1,1))$ for~\eqref{eq:12-1111}, and~$(\mathbf{n},\boldsymbol{\ell}) = ((1,1,1),(1,1,1,1))$ for~\eqref{eq:111-1111}.
\end{proof}

Note that we have the following conclusions.

\begin{lemma}
\label{lem:23}
Let be~$\mathbf{k} = (k_1,\dots,k_4)\in\Z^4_{>0}$ and write~$w=|\mathbf{k}|+3$. For all~$1\leq j\leq 4$, we have
\begin{align}
    \eqlabel{eq:j1411=-2122}
    0\equiv\ &\fz{\Psi_{\mathbf{k}}\left(\uu_1^{j-1} \uu_4 \uu_1^{4-j} + \uu_2^{j-1}\uu_1 \uu_2^{4-j}\right)}&\hspace{-7mm}\mod\F{3,4,w},\hspace{7mm}
    \\
    \eqlabel{eq:1222:23}
    0\equiv\ & \fz{u_{k_1}u_{k_2}u_{k_3}u_0u_0u_{k_4}u_0} + \fz{u_{k_1}u_{k_2}u_0u_0u_{k_3}u_{k_4}u_0} \\ &+\fz{u_{k_1}u_{k_2}u_0u_0u_{k_3}u_0u_{k_4}} &\hspace{-7mm}\mod\F{3,4,w},\hspace{7mm}
    \\
    \eqlabel{eq:1222:32}
    0\equiv\ & \fz{u_{k_1}u_{k_2}u_{k_3}u_0u_{k_4}u_0u_0} + \fz{u_{k_1}u_{k_2}u_0u_{k_3}u_{k_4}u_0u_0}\\ &+\fz{u_{k_1}u_{k_2}u_0u_{k_3}u_0u_0u_{k_4}} &\hspace{-7mm}\mod\F{3,4,w},\hspace{7mm}
    \\
    \eqlabel{eq:2311=-1123}
    0\equiv\ &\fz{u_{k_1}u_{k_2}u_0u_0u_{k_3}u_0u_{k_4}} + \fz{u_{k_1}u_0u_0u_{k_2}u_0u_{k_3}u_{k_4}} &\hspace{-7mm}\mod\F{3,4,w},\hspace{7mm}
    \\
    \eqlabel{eq:3211=-1132}
    0\equiv\ &\fz{u_{k_1}u_{k_2}u_{k_3}u_0u_{k_4}u_0u_0}+ \fz{u_{k_1}u_0u_{k_2}u_0u_0u_{k_3}u_{k_4}} &\hspace{-7mm}\mod\F{3,4,w},\hspace{7mm}
    \\
    \eqlabel{eq:3112=-1231}
    0\equiv\ & \fz{u_{k_1}u_0u_{k_2}u_{k_3}u_{k_4}u_0u_0} + \fz{u_{k_1}u_{k_2}u_0u_0u_{k_3}u_0u_{k_4}} &\hspace{-7mm}\mod\F{3,4,w},\hspace{7mm}
    \\
    \eqlabel{eq:2113=-1321}
    0\equiv\ &\fz{u_{k_1}u_0u_0u_{k_2}u_{k_3}u_{k_4}u_0} + \fz{u_{k_1}u_{k_2}u_0u_{k_3}u_0u_0u_{k_4}} &\hspace{-7mm}\mod\F{3,4,w}.\hspace{6.5mm}
\end{align}

\end{lemma}

\begin{proof}
    The proof of \eqref{eq:j1411=-2122} is obtained from Lemma~\ref{lem:mainideaboxApp} and the direct calculation
    \begin{align*}
        0 \equiv\ &\sum\limits_{p=1}^3 (-1)^p \fz{\Psi_{\mathbf{k}}\left(\uu_1^p \boxast \uu_1^{j-1} \uu_{4-p} \uu_1^{4-j}\right)} &&\mod\F{3,4,w} 
        \\
        \equiv\ &\fz{\Psi_{\mathbf{k}}\left(\uu_1^{j-1} \uu_4 \uu_1^{4-j} + \uu_2^{j-1}\uu_1 \uu_2^{4-j}\right)}&&\mod\F{3,4,w}.
    \end{align*} 
    Note that \eqref{eq:2311=-1123} is a consequence of \eqref{eq:kkkk000=0}, \eqref{eq:1-2211}, \eqref{eq:2-1211}, \eqref{eq:2-2111}, when using \eqref{eq:12-1111}. Analogously, \eqref{eq:3211=-1132} is a consequence of \eqref{eq:kkkk000=0}, \eqref{eq:1-1122}, \eqref{eq:2-1121}, \eqref{eq:2-1112}, using \eqref{eq:21-1111}. Furthermore, we obtain \eqref{eq:3112=-1231} with \eqref{eq:21-1111}, \eqref{eq:111-1111}, and case~$j=3$ of \eqref{eq:j1411=-2122}, in a similar way, using Lemma~\ref{lem:mainideaboxApp},
\begin{align}
    \nonumber 0\equiv\ & \fz{\Psi_{\mathbf{k}}(\uu_1\boxast \uu_1\uu_2\uu_1\uu_2)} - \fz{\Psi_{\mathbf{k}}(\uu_2\boxast \uu_1\uu_2\uu_1\uu_1)}\\ &-\fz{\Psi_{\mathbf{k}}(\uu_2\boxast \uu_1\uu_1\uu_1\uu_2)} - \fz{u_1 u_0^{k_4-1} u_2 u_0^{k_3-1} u_3 u_0^{k_2-1} u_1 u_0^{k_1-1}} \mod\F{3,4,w}
    \\
    \equiv\ & \fz{u_{k_1}u_0u_{k_2}u_{k_3}u_{k_4}u_0u_0} + \fz{u_{k_1}u_{k_2}u_0u_0u_{k_3}u_0u_{k_4}} \mod\F{3,4,w}
\end{align}
and we obtain \eqref{eq:2113=-1321} with \eqref{eq:12-1111}, \eqref{eq:111-1111}, and with case~$j=3$ of \eqref{eq:j1411=-2122},
\begin{align}
    0\equiv\ & \fz{\Psi_{\mathbf{k}}(\uu_1\boxast \uu_2\uu_1\uu_1\uu_1)} - \fz{\Psi_{\mathbf{k}}(\uu_2\boxast \uu_1\uu_1\uu_2\uu_1)} \\ &- \fz{\Psi_{\mathbf{k}}(\uu_2\boxast \uu_2\uu_1\uu_1\uu_1)}  \mod\F{3,4,w}
    \\
    \equiv\ &\fz{u_{k_1}u_0u_0u_{k_2}u_{k_3}u_{k_4}u_0} + \fz{u_{k_1}u_{k_2}u_0u_{k_3}u_0u_0u_{k_4}} \hspace{-2mm}\mod\F{3,4,w},
\end{align}
completing the proof of the lemma.
\end{proof}

For the proof of Theorem~\ref{thm:r=43}, it remains to consider the cases where we have for one~$j\in\{1,2,3,4\}$ that~$k_j>1$. 

\begin{lemma}
    \label{lem:k3done}
    Equation~\eqref{eq:toshow34} is true for~$k_3>1$.
\end{lemma}

\begin{proof}
    
Let be~$k_1,k_2,k_3,k_4\in\Z_{>0}$ and write~$w = k_1+k_2+k_3+k_4+4$. By~\eqref{eq:k0k0k0k=0} and~\eqref{eq:111-1111}, we obtain 
\begin{align}
    \eqlabel{eq:1011-1111}
    \hspace{-5mm}0\equiv\ & \frac{1}{k_3}\fz{\tau(u_1 u_1 u_2)\ast \tau\left(u_{k_1}u_{k_2}u_{k_3}u_{k_4}\right)}&&\hspace{-15mm} \mod\F{3,4,w}
    \\
    \equiv\ & \frac{1}{k_3} \fz{u_1 u_0 u_1 u_1\ast u_1 u_0^{k_4-1} u_1 u_0^{k_3-1} u_1 u_0^{k_2-1} u_1 u_0^{k_1-1}}&&\hspace{-15mm} \mod\F{3,4,w}
    \\ 
    \equiv\ & \fz{u_2 u_0^{k_4-1} u_1 u_0^{k_3} u_2 u_0^{k_2-1} u_2 u_0^{k_1-1}}&&\hspace{-15mm} \mod\F{3,4,w}
    \\ \label{k3:2122=0}
    \equiv\ & \fz{u_{k_1}u_0u_{k_2}u_0u_{k_3+1}u_{k_4}u_0}&&\hspace{-15mm} \mod\F{3,4,w}.
\end{align}

Similar, using \eqref{eq:k0k0k0k=0}, \eqref{eq:111-1111}, \eqref{k3:2122=0}, we have
\begin{align}
    \eqlabel{eq:1101-1111}
    \hspace{-7mm} 0\equiv\ & -\frac{1}{k_2}\fz{\tau(u_1 u_2 u_1)\ast \tau\left(u_{k_1}u_{k_2}u_{k_3}u_{k_4}\right)}&&\hspace{-15mm} \mod\F{3,4,w}
    \\
    \equiv\ & -\frac{1}{k_2} \fz{u_1 u_1 u_0 u_1\ast u_1 u_0^{k_4-1} u_1 u_0^{k_3-1} u_1 u_0^{k_2-1} u_1 u_0^{k_1-1}}&&\hspace{-16mm} \mod\F{3,4,w}
    \\ \equiv\ & \fz{u_2 u_0^{k_4-1} u_2 u_0^{k_3-1} u_2 u_0^{k_2} u_1 u_0^{k_1-1}}&&\hspace{-16mm} \mod\F{3,4,w}
    \\ \label{k2:2221=0} 
    \equiv\ & \fz{u_{k_1}u_{k_2+1}u_0u_{k_3}u_0u_{k_4}u_0}&&\hspace{-16mm} \mod\F{3,4,w}.
\end{align}

Furthermore, using \eqref{eq:111-1111}, \eqref{k2:2221=0}, \eqref{eq:11-1112}, we have
\begin{align}
    \eqlabel{eq:110-1112}
    0\equiv\ & \frac{1}{k_3}\fz{\tau(u_2 u_1)\ast \tau\left(u_{k_1}u_0u_{k_2}u_{k_3}u_{k_4}\right)}&&\hspace{-15mm} \mod\F{3,4,w}
    \\
    \equiv\ & \frac{1}{k_3} \fz{u_1 u_1 u_0\ast u_1 u_0^{k_4-1} u_1 u_0^{k_3-1} u_1 u_0^{k_2-1} u_2 u_0^{k_1-1}}&&\hspace{-15mm} \mod\F{3,4,w}
    \\ 
    \equiv\ & \fz{u_2 u_0^{k_4-1} u_2 u_0^{k_3} u_1 u_0^{k_2-1} u_2 u_0^{k_1-1}}&&\hspace{-15mm} \mod\F{3,4,w}
    \\ \label{k3:2212=0}
    \equiv\ & \fz{u_{k_1}u_0u_{k_2}u_{k_3+1}u_0u_{k_4}u_0}&&\hspace{-15mm} \mod\F{3,4,w}.
\end{align}
This implies by \eqref{eq:111-1111} and \eqref{k3:2122=0}
\begin{align}
    \label{k3:2221=0}
    0 \equiv\ & \fz{u_{k_1}u_{k_2}u_0u_{k_3+1}u_0u_{k_4}u_0}
    \mod\F{3,4,w}.
\end{align}

Now, by equations \eqref{eq:111-1111}, \eqref{eq:1222:23}, \eqref{eq:1222:32}, \eqref{k3:2221=0}, \eqref{k2:2221=0}, we have
\begin{align}
    \eqlabel{eq:110-1121}
    \hspace{-3mm}0 \equiv\ & -\frac{1}{k_2}\fz{\tau(u_2 u_1)\ast \tau\left(u_{k_1}u_{k_2}u_0u_{k_3}u_{k_4}\right)}&&\hspace{-15mm} \mod\F{3,4,w}
    \\
    \equiv\ & -\frac{1}{k_2} \fz{u_1 u_1 u_0\ast u_1 u_0^{k_4-1} u_1 u_0^{k_3-1} u_2 u_0^{k_2-1} u_1 u_0^{k_1-1}}&&\hspace{-15mm} \mod\F{3,4,w}
    \\ 
    \equiv\ & \fz{u_2 u_0^{k_4-1} u_3 u_0^{k_3-1} u_1 u_0^{k_2} u_1 u_0^{k_1-1}}&&\hspace{-15mm} \mod\F{3,4,w}
    \\ \label{k2:2311=0} 
    \equiv\ & \fz{u_{k_1}u_{k_2+1}u_{k_3}u_0u_0u_{k_4}u_0}&&\hspace{-15mm} \mod\F{3,4,w}.
\end{align}

Considering \eqref{eq:2311=-1123}, this implies
\begin{align}
    \label{k2:1123=0}
    0 \equiv\ & \fz{u_{k_1}u_0u_0u_{k_2+1}u_0u_{k_3}u_{k_4}}
    \mod\F{3,4,w}.
\end{align}

Next, we consider, using Corollary~\ref{cor:z=1},~\eqref{eq:12-1111}, \eqref{eq:1222:23}, \eqref{eq:1222:32},
\begin{align}
    \eqlabel{eq:120-1111}
    0\equiv\ & \frac{1}{k_3}\fz{\tau(u_2u_0u_1)\ast \tau\left(u_{k_1}u_{k_2}u_{k_3}u_{k_4}\right)}&&\hspace{-13mm} \mod\F{3,4,w}
    \\
    \equiv\ & \frac{1}{k_3} \fz{u_1 u_2 u_0\ast u_1 u_0^{k_4-1} u_1 u_0^{k_3-1} u_1 u_0^{k_2-1} u_1 u_0^{k_1-1}}&&\hspace{-13mm} \mod\F{3,4,w}
    \\ 
    \equiv\ & \fz{u_2 u_0^{k_4-1} u_3 u_0^{k_3} u_1 u_0^{k_2-1} u_1 u_0^{k_1-1}}&&\hspace{-13mm} \mod\F{3,4,w}
    \\ \label{k3:2311=0}
    \equiv\ & \fz{u_{k_1}u_{k_2}u_{k_3+1}u_0u_0u_{k_4}u_0}&&\hspace{-13mm} \mod\F{3,4,w}.
\end{align}
Now, a consequence of \eqref{eq:2311=-1123} is
\begin{align}
    \label{k3:1123=0}
    0 \equiv\ & \fz{u_{k_1}u_0u_0u_{k_2}u_0u_{k_3+1}u_{k_4}}
    \mod\F{3,4,w}.
\end{align}

In a similar way, we obtain by Corollary~\ref{cor:z=1},~\eqref{eq:21-1111}, \eqref{eq:1222:23}, \eqref{eq:1222:32},
\begin{align}
    \eqlabel{eq:210-1111}
    0\equiv\ & \frac{1}{k_3}\fz{\tau(u_2 u_1 u_0)\ast \tau\left(u_{k_1}u_{k_2}u_{k_3}u_{k_4}\right)}&&\hspace{-13mm} \mod\F{3,4,w}
    \\
    \equiv\ & \frac{1}{k_3} \fz{u_2 u_1 u_0\ast u_1 u_0^{k_4-1} u_1 u_0^{k_3-1} u_1 u_0^{k_2-1} u_1 u_0^{k_1-1}}&&\hspace{-13mm} \mod\F{3,4,w}
    \\ 
    \equiv\ & \fz{u_3 u_0^{k_4-1} u_2 u_0^{k_3} u_1 u_0^{k_2-1} u_1 u_0^{k_1-1}}&&\hspace{-13mm} \mod\F{3,4,w}
    \\ \label{k3:3211=0}
    \equiv\ & \fz{u_{k_1}u_{k_2}u_{k_3+1}u_0u_{k_4}u_0u_0}&&\hspace{-13mm} \mod\F{3,4,w}.
\end{align}
By \eqref{eq:3211=-1132}, one obtains 
\begin{align} 
    \label{k3:1132=0}
    0 \equiv\ & \fz{u_{k_1}u_0u_{k_2}u_0u_0u_{k_3+1}u_{k_4}}
    \mod\F{3,4,w}.
\end{align}

From Corollary~\ref{cor:z=1},~\eqref{k3:2311=0}, \eqref{k2:2311=0}, and \eqref{k2:2221=0} we immediately get
\begin{align}
    \eqlabel{eq:110-1211}
    \hspace{-5mm}0\equiv\ & \frac{1}{k_2}\fz{\tau(u_2 u_1)\ast \tau\left(u_{k_1}u_{k_2}u_{k_3}u_0u_{k_4}\right)}&&\hspace{-15mm} \mod\F{3,4,w}
    \\
    \equiv\ & \frac{1}{k_2} \fz{u_1 u_1 u_0\ast u_1 u_0^{k_4-1} u_2 u_0^{k_3-1} u_1 u_0^{k_2-1} u_1 u_0^{k_1-1}}&&\hspace{-15mm} \mod\F{3,4,w}
    \\ 
    \equiv\ & \fz{u_1 u_0^{k_4-1} u_3 u_0^{k_3-1} u_2 u_0^{k_2} u_1 u_0^{k_1-1}}&&\hspace{-15mm} \mod\F{3,4,w}
    \\ \label{k2:1321=0}
    \equiv\ & \fz{u_{k_1}u_{k_2+1}u_0u_{k_3}u_0u_0u_{k_4}}&&\hspace{-15mm} \mod\F{3,4,w},
\end{align}
and so, by \eqref{eq:2113=-1321},
\begin{align}
    \label{k2:2113=0}
    0 \equiv\ & \fz{u_{k_1}u_0u_0u_{k_2+1}u_{k_3}u_{k_4}u_0}
    \mod\F{3,4,w}.
\end{align}
This implies, using \eqref{eq:1-1113}, \eqref{k2:1123=0}, \eqref{k2:2221=0},
\begin{align}
    \label{k2:1213=0}
    0 \equiv\ & \fz{u_{k_1}u_0u_0u_{k_2+1}u_{k_3}u_0u_{k_4}}
    \mod\F{3,4,w}.
\end{align}

Also, from \eqref{eq:2-2111}, using \eqref{k2:2311=0}, \eqref{k2:2113=0}, and \eqref{eq:kkkk000=0}, we obtain
\begin{align}
    \label{k2:2131=0}
    0 \equiv\ & \fz{u_{k_1}u_{k_2+1}u_0u_0u_{k_3}u_{k_4}u_0}
    \mod\F{3,4,w}.
\end{align}

This leads to, using \eqref{eq:12-1111}, \eqref{k2:2311=0}, \eqref{k2:1123=0}, \eqref{k2:2113=0}, \eqref{k2:1213=0}, 
\begin{align}
    \label{k2:1231=0}
    0 \equiv\ & \fz{u_{k_1}u_{k_2+1}u_0u_0u_{k_3}u_0u_{k_4}}
    \mod\F{3,4,w}.
\end{align}
A consequence of \eqref{eq:3112=-1231} then is
\begin{align}
    \label{k2:3112=0}
    0 \equiv\ & \fz{u_{k_1}u_0u_{k_2+1}u_{k_3}u_{k_4}u_0u_0}
    \mod\F{3,4,w}.
\end{align}

By Corollary~\ref{cor:z=1},~\eqref{eq:1-2112}, \eqref{eq:1-1113}, and~\eqref{k3:2122=0}, we have
\begin{align}
    \eqlabel{eq:101-1112}
    \hspace{-5mm}0\equiv\ & -\frac{1}{k_4}\fz{\tau(u_1 u_2)\ast \tau\left(u_{k_1}u_0u_{k_2}u_{k_3}u_{k_4}\right)}&&\hspace{-15mm} \mod\F{3,4,w}
    \\
    \equiv\ & -\frac{1}{k_4} \fz{u_1 u_0 u_1\ast u_1 u_0^{k_4-1} u_1 u_0^{k_3-1} u_1 u_0^{k_2-1} u_2 u_0^{k_1-1}}&&\hspace{-15mm} \mod\F{3,4,w}
    \\
    \equiv\ & \fz{u_3 u_0^{k_4} u_1 u_0^{k_3-1} u_1 u_0^{k_2-1} u_2 u_0^{k_1-1}}&&\hspace{-15mm} \mod\F{3,4,w}
    \\ \label{k4:3112=0} 
    \equiv\ & \fz{u_{k_1}u_0u_{k_2}u_{k_3}u_{k_4+1}u_0u_0}&&\hspace{-15mm} \mod\F{3,4,w}.
\end{align}

Hence, by Theorem~\ref{thm:r=42} for the first congruence and by applying \eqref{eq:1-2112}, \eqref{eq:1-1113}, \eqref{k3:2212=0} afterwards, we see that
\begin{align}
    \eqlabel{eq:10-2112}
    \hspace{-5mm}0\equiv\ & -\frac{1}{k_3}\fz{\tau(u_2)\ast \tau\left(u_{k_1}u_0u_{k_2}u_{k_3}u_{k_4}u_0\right)}&&\hspace{-16mm} \mod\F{3,4,w}
    \\
    \equiv\ & -\frac{1}{k_3}\fz{u_1 u_0\ast u_2 u_0^{k_4-1} u_1 u_0^{k_3-1} u_1 u_0^{k_2-1} u_2 u_0^{k_1-1}}&&\hspace{-16mm} \mod\F{3,4,w}
    \\ \equiv\ & \fz{u_3 u_0^{k_4-1} u_1 u_0^{k_3} u_1 u_0^{k_2-1} u_2 u_0^{k_1-1}}&&\hspace{-16mm} \mod\F{3,4,w}
    \\ \label{k3:3112=0} 
    \equiv\ & \fz{u_{k_1}u_0u_{k_2}u_{k_3+1}u_{k_4}u_0u_0}&&\hspace{-16mm} \mod\F{3,4,w}.
\end{align}

Now, \eqref{eq:3112=-1231} yields
\begin{align}
    \label{k3:1231=0}
    0 \equiv\ & \fz{u_{k_1}u_{k_2}u_0u_0u_{k_3+1}u_0u_{k_4}}
    \mod\F{3,4,w}.
\end{align}

Furthermore, \eqref{k3:3112=0} implies with \eqref{eq:2-1112} and \eqref{k3:1132=0}, respectively \eqref{eq:1-3111} and \eqref{k3:3211=0},
\begin{align}
    \label{k3:1312=0}
    0 \equiv\ & \fz{u_{k_1}u_0u_{k_2}u_{k_3+1}u_0u_0u_{k_4}}
    \mod\F{3,4,w},
\end{align}
respectively,
\begin{align}
    \label{k3:3121=0}
    0 \equiv\ & \fz{u_{k_1}u_{k_2}u_0u_{k_3+1}u_{k_4}u_0u_0}
    \mod\F{3,4,w}.
\end{align}
The latter implies by using \eqref{eq:21-1111} for the first congruence, then \eqref{eq:2113=-1321} for the second one, \eqref{eq:2-2111} for the third one, and \eqref{eq:12-1111} for the last one,
\begin{align}
    \label{k3:1321=0}
    0 \equiv\ & \fz{u_{k_1}u_{k_2}u_0u_{k_3+1}u_0u_0u_{k_4}}
    \mod\F{3,4,w},
    \\
    \label{k3:2113=0}
    0 \equiv\ & \fz{u_{k_1}u_0u_0u_{k_2}u_{k_3+1}u_{k_4}u_0}
    \mod\F{3,4,w},
    \\
    \label{k3:2131=0}
    0 \equiv\ & \fz{u_{k_1}u_{k_2}u_0u_0u_{k_3+1}u_{k_4}u_0}
    \mod\F{3,4,w},
    \\
    \label{k3:1213=0}
    0 \equiv\ & \fz{u_{k_1}u_0u_0u_{k_2}u_{k_3+1}u_0u_{k_4}}
    \mod\F{3,4,w}.
\end{align}
This completes the proof of the lemma.
\end{proof}

\begin{lemma}
    \label{lem:k4done}
    Equation~\eqref{eq:toshow34} is true for~$k_4>1$.
\end{lemma}

\begin{proof}
    Let be~$k_1,k_2,k_3,k_4\in\Z_{>0}$ and write~$w = k_1+k_2+k_3+k_4+4$. From \eqref{k4:3112=0}, we obtain by \eqref{eq:3112=-1231}
\begin{align}
    \label{k4:1231=0}
    0 \equiv\ & \fz{u_{k_1}u_{k_2}u_0u_0u_{k_3}u_0u_{k_4+1}}
    \mod\F{3,4,w}.
\end{align}

From Corollary~\ref{cor:z=1},~ Lemma~\ref{lem:k3done}, \eqref{eq:1-1122}, \eqref{k2:1123=0}, one sees
\begin{align}
    \eqlabel{eq:101-1121}
    \hspace{-5mm}0\equiv\ & -\frac{1}{k_4}\fz{\tau(u_1 u_2)\ast \tau\left(u_{k_1}u_{k_2}u_0u_{k_3}u_{k_4}\right)}&&\hspace{-13mm} \mod\F{3,4,w}
    \\
    \equiv\ & -\frac{1}{k_4} \fz{u_1 u_0 u_1\ast u_1 u_0^{k_4-1} u_1 u_0^{k_3-1} u_2 u_0^{k_2-1} u_1 u_0^{k_1-1}}&&\hspace{-13mm} \mod\F{3,4,w}
    \\ 
    \equiv\ & \fz{u_3 u_0^{k_4} u_1 u_0^{k_3-1} u_2 u_0^{k_2-1} u_1 u_0^{k_1-1}} &&\hspace{-13mm} \mod\F{3,4,w}
    \\ \label{k4:3121=0} 
    \equiv\ & \fz{u_{k_1}u_{k_2}u_0u_{k_3}u_{k_4+1}u_0u_0}&&\hspace{-13mm} \mod\F{3,4,w},
\end{align}

With \eqref{eq:1-3111}, this implies
\begin{align}
    \label{k4:3211=0}
    0 \equiv\ & \fz{u_{k_1}u_{k_2}u_{k_3}u_0u_{k_4+1}u_0u_0}
    \mod\F{3,4,w},
    \\
    \label{k4:1132=0}
    0 \equiv\ & \fz{u_{k_1}u_0u_{k_2}u_0u_0u_{k_3}u_{k_4+1}}
    \mod\F{3,4,w}.
\end{align}
The second congruence is a consequence of the first one and \eqref{eq:3211=-1132}.

Furthermore, Theorem~\ref{thm:r=42} for the first congruence, Lemma~\ref{lem:k3done} and \eqref{eq:1-1212} for the third one, give
\begin{align}
    \eqlabel{eq:10-1212}
    0\equiv\ & \frac{1}{k_4}\fz{\tau(u_2)\ast \tau\left(u_{k_1}u_0u_{k_2}u_{k_3}u_0u_{k_4}\right)}&&\hspace{-13mm} \mod\F{3,4,w}
    \\
    \equiv\ &\frac{1}{k_4} \fz{u_1 u_0\ast u_1 u_0^{k_4-1} u_2 u_0^{k_3-1} u_1 u_0^{k_2-1} u_2 u_0^{k_1-1}}&&\hspace{-13mm} \mod\F{3,4,w}
    \\  \equiv\ & \fz{u_2 u_0^{k_4} u_2 u_0^{k_3-1} u_1 u_0^{k_2-1} u_2 u_0^{k_1-1}}&&\hspace{-13mm} \mod\F{3,4,w}
    \\ \label{k4:2212=0}
    \equiv\ & \fz{u_{k_1}u_0u_{k_2}u_{k_3}u_0u_{k_4+1}u_0}&&\hspace{-13mm} \mod\F{3,4,w},
\end{align}
and so, applying case~$j=3$ of \eqref{eq:j1411=-2122},
\begin{align}
    \label{k4:1141=0}
    0 \equiv\ & \fz{u_{k_1}u_{k_2}u_0u_0u_0u_{k_3}u_{k_4+1}}
    \mod\F{3,4,w}.
\end{align}

Furthermore, by Theorem~\ref{thm:r=42} for the first congruence and by Lemma~\ref{lem:k3done} and \eqref{eq:1-1113} for the third one, we observe
\begin{align}
    \eqlabel{eq:10-1113}
    0\equiv\ & \frac{1}{k_4}\fz{\tau(u_2)\ast \tau\left(u_{k_1}u_0u_0u_{k_2}u_{k_3}u_{k_4}\right)}&&\hspace{-13mm} \mod\F{3,4,w}
    \\
    \equiv\ &\frac{1}{k_4} \fz{u_1 u_0\ast u_1 u_0^{k_4-1} u_1 u_0^{k_3-1} u_1 u_0^{k_2-1} u_3 u_0^{k_1-1}}&&\hspace{-13mm} \mod\F{3,4,w}
    \\ 
    \equiv\ & \fz{u_2 u_0^{k_4} u_1 u_0^{k_3-1} u_1 u_0^{k_2-1} u_3 u_0^{k_1-1}}&&\hspace{-13mm} \mod\F{3,4,w}
    \\ \label{k4:2113=0}
    \equiv\ & \fz{u_{k_1}u_0u_0u_{k_2}u_{k_3}u_{k_4+1}u_0}&&\hspace{-13mm} \mod\F{3,4,w}.
\end{align}

This implies, using \eqref{eq:2113=-1321}, and \eqref{eq:21-1111} for the second congruence additionally,
\begin{align}
    \label{k4:1321=0}
    0 \equiv\ & \fz{u_{k_1}u_{k_2}u_0u_{k_3}u_0u_0u_{k_4+1}}
    \mod\F{3,4,w},
    \\
    \label{k4:1312=0}
    0 \equiv\ & \fz{u_{k_1}u_0u_{k_2}u_{k_3}u_0u_0u_{k_4+1}}
    \mod\F{3,4,w}.
\end{align}
Now, \eqref{eq:2-1121}, \eqref{k4:2212=0}, \eqref{k4:3121=0}, \eqref{k4:1321=0} yield
\begin{align}
    \label{k4:1123=0}
    0 \equiv\ & \fz{u_{k_1}u_0u_0u_{k_2}u_0u_{k_3}u_{k_4+1}}
    \mod\F{3,4,w},
    \\
    \label{k4:2311=0}
    0 \equiv\ & \fz{u_{k_1}u_{k_2}u_{k_3}u_0u_0u_{k_4+1}u_0}
    \mod\F{3,4,w}.
\end{align}
The second congruence is a consequence of the first one and \eqref{eq:2311=-1123}. Using \eqref{k4:2311=0} and equations \eqref{eq:2-2111} and \eqref{k4:2113=0}, we see that
\begin{align}
    \label{k4:2131=0}
    0 \equiv\ & \fz{u_{k_1}u_{k_2}u_0u_0u_{k_3}u_{k_4+1}u_0}
    \mod\F{3,4,w},
    \\
    \label{k4:1213=0}
    0 \equiv\ & \fz{u_{k_1}u_0u_0u_{k_2}u_{k_3}u_0u_{k_4+1}}
    \mod\F{3,4,w},
\end{align}
where the second congruence is implied by the first one and \eqref{eq:12-1111}.

Combining \eqref{eq:2-1211}, \eqref{k4:3211=0}, \eqref{k4:1231=0}, \eqref{k4:1213=0}, we have
\begin{align}
    \label{k4:1411=0}
    0 \equiv\ & \fz{u_{k_1}u_{k_2}u_{k_3}u_0u_0u_0u_{k_4+1}}
    \mod\F{3,4,w},
    \\
    \label{k4:2122=0}
    0 \equiv\ & \fz{u_{k_1}u_0u_{k_2}u_0u_{k_3}u_{k_4+1}u_0}
    \mod\F{3,4,w}.
\end{align}
The second congruence is a consequence of the first one and case~$j=2$ of \eqref{eq:j1411=-2122}.

Now, \eqref{eq:3-1111}, \eqref{k4:1411=0}, \eqref{k4:1141=0}, \eqref{eq:kkkk000=0} give
\begin{align}
    \label{k4:1114=0}
    0 \equiv\ & \fz{u_{k_1}u_0u_0u_0u_{k_2}u_{k_3}u_{k_4+1}}
    \mod\F{3,4,w},
    \\
    \label{k4:2221=0}
    0 \equiv\ & \fz{u_{k_1}u_{k_2}u_0u_{k_3}u_0u_{k_4+1}u_0}
    \mod\F{3,4,w}.
\end{align}
The second congruence is a consequence of the first one and case~$j=4$ of \eqref{eq:j1411=-2122} additionally. This completes the proof of the Lemma.
\end{proof}

\begin{lemma}
    \label{lem:k2done}
    Equation~\eqref{eq:toshow34} is true for~$k_2>1$.
\end{lemma}

\begin{proof}
Let be~$k_1,k_2,k_3,k_4\in\Z_{>0}$ and write~$w = k_1+k_2+k_3+k_4+4$. Note that by Theorem~\ref{thm:r=42} for the first congruence and by Lemmas~\ref{lem:k3done} and~\ref{lem:k4done}, and equations \eqref{k2:2131=0} and \eqref{k2:2221=0} for the third congruence, we have
\begin{align}
    \eqlabel{eq:10-2121}
    0\equiv\ & \frac{1}{k_2}\fz{\tau(u_2)\ast \tau\left(u_{k_1}u_{k_2}u_0u_{k_3}u_{k_4}u_0\right)}&&\hspace{-13mm} \mod\F{3,4,w}
    \\
    \equiv\, &\frac{1}{k_2} \fz{u_1 u_0\ast u_2 u_0^{k_4-1} u_1 u_0^{k_3-1} u_2 u_0^{k_2-1} u_1 u_0^{k_1-1}}&&\hspace{-13mm} \mod\F{3,4,w} 
    \\
    \equiv\ & \fz{u_3 u_0^{k_4-1} u_1 u_0^{k_3-1} u_2 u_0^{k_2} u_1 u_0^{k_1-1}}&&\hspace{-13mm} \mod\F{3,4,w}
    \\ \label{k2:3121=0}
    \equiv\ & \fz{u_{k_1}u_0u_0u_{k_2+1}u_{k_3}u_0u_{k_4}}&&\hspace{-13mm} \mod\F{3,4,w}.
\end{align}

By \eqref{eq:1-3111} and \eqref{k2:3112=0}, this yields
\begin{align}
    \label{k2:3211=0}
    0 \equiv\ & \fz{u_{k_1}u_{k_2+1}u_{k_3}u_0u_{k_4}u_0u_0}
    \mod\F{3,4,w},
\end{align}
leading to, by using \eqref{eq:3211=-1132} and then \eqref{eq:21-1111},
\begin{align}
    \label{k2:1132=0}
    0 \equiv\ & \fz{u_{k_1}u_0u_{k_2+1}u_0u_0u_{k_3}u_{k_4}}
    \mod\F{3,4,w},
    \\
    \label{k2:1312=0}
    0 \equiv\ & \fz{u_{k_1}u_0u_{k_2+1}u_{k_3}u_0u_0u_{k_4}}
    \mod\F{3,4,w}.
\end{align}
Combining \eqref{eq:2-1211}, \eqref{k2:3211=0}, \eqref{k2:1231=0}, and \eqref{k2:1213=0}, we obtain
\begin{align}
    \label{k2:1411=0}
    0 \equiv\ & \fz{u_{k_1}u_{k_2+1}u_{k_3}u_0u_0u_0u_{k_4}}
    \mod\F{3,4,w},
    \\
    \label{k2:2122=0}
    0 \equiv\ & \fz{u_{k_1}u_0u_{k_2+1}u_0u_{k_3}u_{k_4}u_0}
    \mod\F{3,4,w}.
\end{align}
The second congruence is a consequence of the first one and case~$j=2$ of \eqref{eq:j1411=-2122}.

Furthermore, combining \eqref{eq:2-1121}, \eqref{k2:3121=0}, \eqref{k2:1321=0}, and \eqref{k2:1123=0}, we obtain
\begin{align}
    \label{k2:1141=0}
    0 \equiv\ & \fz{u_{k_1}u_{k_2+1}u_0u_0u_0u_{k_3}u_{k_4}}
    \mod\F{3,4,w}.
    \\
    \label{k2:2212=0}
    0 \equiv\ & \fz{u_{k_1}u_0u_{k_2+1}u_{k_3}u_0u_{k_4}u_0}
    \mod\F{3,4,w}.
\end{align}
The second congruence is a consequence of the first one and case~$j=3$ of \eqref{eq:j1411=-2122}. This completes the proof of the lemma.
\end{proof}

\begin{lemma}
    \label{lem:k1done}
    Equation~\eqref{eq:toshow34} is true for~$k_1>1$.
\end{lemma}

\begin{proof}
    Let be~$k_1,k_2,k_3,k_4\in\Z_{>0}$ with~$k_1>1$ and write~$w = k_1+k_2+k_3+k_4+3$. Using Proposition~\ref{prop:r=3} for the first congruence and Lemmas~\ref{lem:k3done},~\ref{lem:k4done},~\ref{lem:k2done} afterwards, for all~$z_2,z_3,z_4\geq 0$ with~$z_2+z_3+z_4 = 3$, we obtain 
    \begin{align}
    \nonumber
        0&\equiv\ \fz{u_{k_1}\ast u_{k_2} u_0^{z_2} u_{k_3} u_0^{z_3} u_{k_4} u_0^{z_4}} &&\mod\F{3,4,w}
        \\
        \eqlabel{eq:r4d1=0}
        &\equiv\ \fz{u_{k_1} u_{k_2} u_0^{z_2} u_{k_3} u_0^{z_3} u_{k_4} u_0^{z_4}} &&\mod\F{3,4,w}.
    \end{align}
    Now, choose~$z_1\geq 1,z_2,z_3,z_4\geq 0$ with~$z_1+\cdots + z_4 = 3$. Then, we obtain by Theorem~\ref{thm:r=42} (in case~$z_1=1$), Corollary~\ref{cor:z=1} (in case~$z_1=2$), and~\eqref{eq:r4d1=0},
    \begin{align*}
        0\equiv\ &\fz{u_{z_1}\ast \tau\left(u_{k_1} u_{k_2} u_0^{z_2} u_{k_3} u_0^{z_3} u_{k_4} u_0^{z_4}\right)}&&\mod\F{3,4,w}
        \\
        \equiv\ &\fz{u_{z_1}\ast u_{z_4 + 1} u_0^{k_4-1} u_{z_3 + 1} u_0^{k_3-1} u_{z_2 + 1} u_0^{k_2-1} u_{1} u_0^{k_1-1}}&&\mod\F{3,4,w}
        \\
        \equiv\ &\fz{u_{z_4 + 1} u_0^{k_4-1} u_{z_3 + 1} u_0^{k_3-1} u_{z_2 + 1} u_0^{k_2-1} u_{z_1 + 1} u_0^{k_1-1}}&&\mod\F{3,4,w}
        \\
        \equiv\ &\fz{u_{k_1} u_0^{z_1} u_{k_2} u_0^{z_2} u_{k_3} u_0^{z_3} u_{k_4} u_0^{z_4}} &&\mod\F{3,4,w}.
    \end{align*}
    This completes the proof of the lemma.
\end{proof}



\section{Conclusion and outlook}
\label{sec:mdbd:outlook}

With $\fil{Z,D,W}{z,\dd,w}\Zq\subset\F{z,\dd,w}$ for all $(z,\dd,w)\in\Z_{>0}^3$ (the refined Bachmann Conjecture~\ref{conj:mdbdstrongApp}), we gave a refinement of Bachmann's Conjecture~\ref{conj:mdbdApp} and proved several cases. For~$z\geq\dd$, we gave a strategy for a general proof. Furthermore, for~$z<\dd$, we were also able to prove the cases~$1\leq\dd\leq 4$. One can generalize our approach as described in the following paragraph. 

\paragraph{\textbf{Approach to the refined Bachmann Conjecture~\ref{conj:mdbdstrongApp} in case~$z<\dd$.}} 
We fix positive integers~$z,\dd,w\in\Z_{>0}$ with~$z<\dd$ in the following and assume throughout the whole paragraph that 
\begin{align} 
\eqlabel{eq:mainassump}
    \fil{Z,D,W}{\tilde{z},\tilde{\dd},\tilde{w}}\Zq\subset \F{\tilde{z},\tilde{\dd},\tilde{w}}
\end{align} 
for~$\tilde{z}\leq z,\, \tilde{\dd}<\dd,\ \tilde{w}<w$ is proven already. Note that the approach from case~$z\geq \dd$ will not suffice for the case~$z<\dd$ since~$\VSbox{z}{\dd}\subsetneq\Vbox{z}{\dd}$ in this case by Conjecture~\ref{conj:systemApp}. Therefore, we extend this approach as follows. Fix throughout this paragraph~$\mathbf{k} = (k_1,\dots,k_\dd)\in\Z_{>0}^\dd$ with~$|\mathbf{k}| = w-z$. Besides 
\begin{align*}
    S^{(1)}_{z,\dd,\mathbf{k}} := \left\{\fz{\Psi_{\mathbf{k}}(\uu_{\mathbf{n}}\boxast \uu_{\boldsymbol{\ell}})}\mid (\mathbf{n},\boldsymbol{\ell})\in\indexset{z}{\dd}\right\}\subset \fil{Z,D,W}{z,\dd,w}\Zq
\end{align*}
(the inclusion follows from Lemma~\ref{lem:mainideaboxApp}), we consider
\begin{align*}
    S^{(2)}_{z,\dd,\mathbf{k}} := \left\{\fz{\tau(\tau(\word_{\mathbf{n},\mathbf{m}}) \ast \tau(\word_{\boldsymbol{\ell},\mathbf{k'}}))}\Bigg|\substack{ (\mathbf{n},\boldsymbol{\ell})\in\indexset{z}{\dd},\, \mathbf{m}\in\Z_{\geq 0}^{\len(\mathbf{n})},\, |\mathbf{m}|\leq\len(\mathbf{n}) + \dd -z,\\ \mathbf{k'}\in\Z_{>0}^\dd, \, k_j\geq k_j'\geq 1\ (1\leq j\leq \dd),\\ |\mathbf{m}|+|\mathbf{k'}| = s + |\mathbf{k}|,\, \wt(\word_{\mathbf{n},\mathbf{m}}) + \wt(\word_{\boldsymbol{\ell},\mathbf{k'}})=w} \right\},
\end{align*}
where
\begin{align*}
    \word_{\mathbf{n},\mathbf{m}} := u_{m_1} u_0^{n_s-1}\cdots u_{m_s} u_0^{n_1-1},\quad \word_{\boldsymbol{\ell},\mathbf{k'}} = u_{k_1'} u_0^{\ell_\dd-1}\cdots u_{k_\dd'} u_0^{\ell_1-1}.
\end{align*}
\begin{remark}
    Note that we have~$S^{(1)}_{z,\dd,\mathbf{k}}\subset S^{(2)}_{z,\dd,\mathbf{k}}$ for all~$z,\dd\in\Z_{>0}$ with~$z<\dd$ and~$\mathbf{k}\in\Z_{>0}^\dd$.
\end{remark}
Furthermore, we consider
\begin{align*}
    S^{(3)}_{z,\dd,\mathbf{k}} := \left\{\fz{u_{\sigma(k_1)}u_0^{e_1}\cdots u_{\sigma(k_{s'})}u_0^{e_{s'}}\ast u_{\sigma(k_{s'+1})}u_0^{e_{s'+1}}\cdots u_{\sigma(k_\dd)}u_0^{e_\dd}}\Bigg| \substack{\sigma\in\mathcal{S}_{\mathbf{k}},\, 1\leq s'\leq \dd-1,\\ \mathbf{e} = (e_1,\dots,e_\dd)\in\Z_{\geq 0}^\dd,\, |\mathbf{e}| = z}\right\},
\end{align*}
where~$\mathcal{S}_{\mathbf{k}}$ is the set of permutations on~$\{k_j\mid 1\leq j\leq \dd\}$.

Similarly to the proof of Lemma~\ref{lem:mainideaboxApp}, we can show the following.
\begin{lemma}
\label{lem:approachzlessd1App}
    Fix~$z,\dd,w\in\Z_{>0}$. For all~$(\mathbf{n},\boldsymbol{\ell})\in\indexset{z}{\dd}$, $\mathbf{k},\, \mathbf{k'}\in\Z_{>0}^\dd$, and~$\mathbf{m}\in\Z_{\geq 0}^{s}$, where~$s=\len(\mathbf{n})$, satisfying~$|\mathbf{k}| = w-z$,~$|\mathbf{m}|\leq\len(\mathbf{n}) + \dd -z$ and~$k_j\geq k_j'\geq 1$ for all~$1\leq j\leq \dd,\ |\mathbf{m}|+|\mathbf{k'}| = s + |\mathbf{k}|,\ \wt(\word_{\mathbf{n},\mathbf{m}}) + \wt(\word_{\boldsymbol{\ell},\mathbf{k'}}) = w$, we have
    \begin{align}
        \fz{\tau(\tau(\word_{\mathbf{n},\mathbf{m}}) \ast \tau(\word_{\boldsymbol{\ell},\mathbf{k'}}))}\in\sum\limits_{1\leq s'\leq s}\fil{Z,D,W}{z-s',\dd+s',w}\Zq.
    \end{align}
    In particular, we have~$S^{(2)}_{z,\dd,\mathbf{k}}\subset\F{z,\dd,w}$.
\end{lemma}\noproof{lemma}

Let us consider an example for illustration of Lemma~\ref{lem:approachzlessd1App}.

\begin{example}
\label{ex:relsexpl}
    Denote~$w = k_1'+k_2'+k_3'+2$ in the following and choose 
    \begin{align*}
        \mathbf{n} = (1),\quad \mathbf{m} = (2),\quad \boldsymbol{\ell} = (1,1,1),\quad \mathbf{k'} = (k_1',k_2',k_3')\in\Z_{>0}^3
    \end{align*}
    in the notation of Lemma~\ref{lem:approachzlessd1App}. First, we see that~$\word_{\mathbf{n},\mathbf{m}} \ast \word_{\boldsymbol{\ell},\mathbf{k'}} = u_2\ast u_{k_1'} u_{k_2'} u_{k_3'} \in \mathcal{F}$, where~$\mathcal{F} = \fil{Z,D,W}{0,4,w}\QB^\circ + \fil{Z,D,W}{1,3,w-1}\QB^\circ$. Furthermore, we have
    \begin{align*}
        &\,\tau(\tau(u_2) \ast \tau(u_{k_1'} u_{k_2'} u_{k_3'}))
        \\
        =&\,\tau\left(u_1 u_0\ast u_1 u_0^{k_3'-1} u_1 u_0^{k_2'-1} u_1 u_0^{k_1'-1}\right)
        \\
        \equiv&\, \tau\left(k_3' u_2 u_0^{k_3'} u_1 u_0^{k_2'-1} u_1 u_0^{k_1'-1} + k_2' u_2 u_0^{k_3'-1} u_1 u_0^{k_2'} u_1 u_0^{k_1'-1}\right. \\ &\left.+ k_2' u_1 u_0^{k_3'-1} u_2 u_0^{k_2'} u_1 u_0^{k_1'-1} + k_1' u_2 u_0^{k_3'-1} u_1 u_0^{k_2'-1} u_1 \right. \\ &\left. + k_1' u_1 u_0^{k_3'-1} u_2 u_0^{k_2'-1} u_1 + k_1' u_1 u_0^{k_3'-1} u_1 u_0^{k_2'-1} u_2 u_0^{k_1'}\right)&&\mod\mathcal{F}
        \\
        \equiv&\, k_3' u_{k_1'} u_{k_2'} u_{k_3'+1} u_0 + k_2' u_{k_1'} u_{k_2'+1} u_{k_3'} u_0 +  k_2' u_{k_1'} u_{k_2'+1} u_0 u_{k_3'} \\ &+ k_1' u_{k_1'+1} u_{k_2'} u_{k_3'} u_0 + k_1' u_{k_1'+1} u_{k_2'} u_0 u_{k_3'} + k_1' u_{k_1'+1} u_0 u_{k_2'} u_{k_3'} &&\mod\mathcal{F}.
    \end{align*}
    Hence,
    \begin{align*}
        &k_3' \fz{u_{k_1'} u_{k_2'} u_{k_3'+1} u_0} + k_2' \fz{u_{k_1'} u_{k_2'+1} u_{k_3'} u_0} +  k_2' \fz{u_{k_1'} u_{k_2'+1} u_0 u_{k_3'}} \\ &+ k_1' \fz{u_{k_1'+1} u_{k_2'} u_{k_3'} u_0} + k_1' \fz{u_{k_1'+1} u_{k_2'} u_0 u_{k_3'}} + k_1' \fz{u_{k_1'+1} u_0 u_{k_2'} u_{k_3'}}\in\F{2,3,w}.
    \end{align*}
    Compared to the linear combinations in~$S^{(1)}_{z,\dd,\mathbf{k}}$, it stands out that the latter linear combination is not a linear combination of words with the same multiplicity and the same non-$u_0$ letters in the same order. Nevertheless, all occurring words~$u_{k_1}u_0^{z_1}u_{k_2}u_0^{z_2}u_{k_3}u_0^{z_3}$ satisfy~$k_j\geq k_j'$ and~$\sum\limits_{j=1}^3 (k_j-k_j') = 1 = |\mathbf{m}| - s = \dd-z$.
\end{example}

Furthermore, we have the following.

\begin{lemma}
\label{lem:approachzlessd2App}
    Fix~$z,\dd,w\in\Z_{>0}$ with~$z<\dd$ and assume that~$\fil{Z,D,W}{z',\dd',w'}\Zq\subset\F{z',\dd',w'}$ is proven already for~$z'\leq z,\, \dd'<\dd,\, w'< w$. Then, for every index~$\mathbf{k} = (k_1,\dots,k_\dd)\in\Z_{>0}^\dd$ and for all permutations~$\sigma$ on~$\{k_1,\dots,k_\dd\}$,~$1\leq s'\leq \dd-1$, and~$\mathbf{e} = (e_1,\dots,e_\dd)\in\Z_{\geq 0}^\dd$ satisfying~$|\mathbf{e}| = z$, we have
    \begin{align*}
        \fz{u_{\sigma(k_1)}u_0^{e_1}\cdots u_{\sigma(k_{s'})}u_0^{e_{s'}}\ast u_{\sigma(k_{s'+1})}u_0^{e_{s'+1}}\cdots u_{\sigma(k_\dd)}u_0^{e_\dd}}\in \F{z,\dd,w}.
    \end{align*}
    In particular, we have~$S^{(3)}_{z,\dd,\mathbf{k}}\subset\F{z,\dd,w}$.
\end{lemma}\noproof{lemma}

With the proofs of Theorems~\ref{thm:mainApp} and~\ref{thm:rleq4App}, we gave evidence for the following conjecture for~$d\leq 4$.

\begin{conjecture}
\label{conj:approachzlessdApp}
    Fix~$z,\dd,w\in\Z_{>0}$ with~$z<\dd$ and assume that~$\fil{Z,D,W}{z',\dd',w'}\Zq\subset\F{z',\dd',w'}$ is proven already for all~$z'\leq z,\, \dd'<\dd,\, w'< w$. Then, for every~$\mathbf{k} = (k_1,\dots,k_\dd)\in\Z_{>0}^\dd$ and for every word~$\word = u_{k_1}u_0^{z_1}\cdots u_{k_\dd}u_0^{z_\dd}\in\mathcal{U}^{\ast,\circ}$ satisfying~$\zero(\word) = z$,~$\dep(\word) = \dd$, and~$\wt(\word) = w$, we have
    \begin{align}
    \eqlabel{eq:approachzlessdApp}
        \fz{\word}\in \operatorname{span}_\Q \left(S^{(2)}_{z,\dd,\mathbf{k}}\cup S^{(3)}_{z,\dd,\mathbf{k}}\right) + \F{z,\dd,w} \subset\F{z,\dd,w}.
    \end{align}
    In particular, then we have~$\fil{Z,D,W}{z,\dd,w}\Zq\subset\F{z,\dd,w}$.
\end{conjecture}

\begin{remark}
    Note that the inclusion in~\eqref{eq:approachzlessdApp} follows from Lemmas~\ref{lem:approachzlessd1App} and~\ref{lem:approachzlessd2App}.
\end{remark}

\begin{remark}
\label{rem:approachzlessdApp}
    We can refine our approach to Conjecture~\ref{conj:approachzlessdApp} as follows. First, we will use for~$\mathbf{k}\in\Z_{>0}^\dd$ satisfying~$\#\{k_j > 1\}\geq \dd - z$ the linear combinations from~$S^{(2)}_{z,\dd,\mathbf{k}}$ only to show~\eqref{eq:approachzlessdApp}. For the remaining cases, we then may assume without loss of generality that~$\#\{k_j = 1\}\geq z$ and use both,~$S^{(2)}_{z,\dd,\mathbf{k}}$ and~$S^{(3)}_{z,\dd,\mathbf{k}}$ to prove~\eqref{eq:approachzlessdApp}. More precise, we consider the cases of~$j_0 := \#\{k_j = 1\}$ with increasing~$j_0\geq z$. The intuitive reason for this is that, for given~$j_0$, on the one hand we may assume that the cases for smaller values of~$j_0$ are proven, making the linear combinations from~$S^{(2)}_{z,\dd,\mathbf{k}}$ easier to handle since parts of them are in~$\F{z,\dd,w}$ already. On the other hand, the more entries of~$\mathbf{k}$ are the same (for our purposes: one), the less formal Multiple Zeta Values of different words occur in the linear combinations from~$S^{(3)}_{z,\dd,\mathbf{k}}$.
\end{remark}

\paragraph{\textbf{Conclusion.}} For $z<\dd$, our strategy also works in the small cases~$1\leq\dd\leq 4$ as shown, but there is still much to do for the general proof. More concretely, we conclude with the following open questions:
\begin{enumerate}
    \item How can one prove Conjecture~\ref{conj:systemApp} in general?
    \item Conjecturally, Conjecture~\ref{conj:systemApp} can be proven via induction on~$z$, $\dd$, or $z+\dd$.
    \item Regarding Conjecture~\ref{conj:systemApp}, we conjecturally have~$\sdim{z}{\dd} = \sdim{\dd}{z}$ for all $z,\dd\in\Z_{>0}$. Can one prove this equality?
    \item How to prove Conjecture~\ref{conj:boxkernel2} in general?
    \item How can one prove~$\fil{Z,D,W}{z,\dd,w}\Zq\subset\F{z,\dd,w}$ for $z<\dd$ in general?
    \item Similar to Proposition~\ref{prop:dep1explicit}, our approach for showing~$\fil{Z,D,W}{z,\dd,w}\Zq\subset\F{z,\dd,w}$ is suitable to obtain for all words~$\word\in\mathcal{U}^{\ast,\circ}$ an explicit formula~$\fz{\word} = \fz{\mathcal{L}}$, where~$\mathcal{L}$ is a linear combination of products of elements in~$\Zqz$. With some engagement following our calculations, this already can be done now for all words~$\word\in\mathcal{U}^{\ast,\circ}$ satisfying~$\zero(\word)+\dep(\word)\leq 6$. What do they look like? Can one find some systematics such that one can derive such formulas also for~$\zero(\word)+\dep(\word) > 6$ (which would prove Bachmann's Conjecture~\ref{conj:mdbdApp} in particular)? 
\end{enumerate}

\appendix


The numerical calculations in the paper were done using Python. In this appendix, the original source code is presented. 

\section{Computations regarding Lemma~\ref{lem:rzleq8}}
\subsection{Setup and basic functions}
\label{ssec:code:setup}
We begin with the required packages.

\begin{lstlisting}[language=python]
import numpy as np
import itertools
import math
from ast import literal_eval
\end{lstlisting}

\noindent The first definitions were elementary for the main calculations.

\begin{function}
    The function \verb|d(z,d,s)| returns~$\binom{z+\dd-1}{z-s}$ for~$z,\dd,s\in\Z_{>0}$ with~$s\leq z\leq\dd$, which is conjecturally~$\sdim{z}{\dd,s}$ (see Conjecture~\ref{conj:conjrefinement}).
\begin{lstlisting}[language=Python]
def d(z,d,s):
    if (z <= d) and (s <= z):
        return(math.comb(z+d-1,z-s))
    elif (z <= d) and (s > z):
        return(0)
\end{lstlisting}
\end{function}

\begin{function}
    The function  \verb|part(r,s)| returns the list of all ordered partitions of $r$ into exactly $s$ non-negative integers.
\begin{lstlisting}[language=python]
def part(r,s):
    if s<=0:
        return([[]])
    else:
        P = []
        for S in set(itertools.combinations(range(r+s-1), s-1)):
            p = []
            I = [-1] + list(S) + [r+s-1]
            for i in range(len(I)):
                if i > 0:
                    p.append(I[i]-I[i-1]-1)
            P.append(p)
        return(P)
\end{lstlisting}
\end{function}

\begin{function}
    The function \verb|ppart(r,s)| returns the list all ordered partitions into exactly~$s$ positive integers.
\begin{lstlisting}[language=python]
def ppart(r,s):
    if s<=0 or r<s:
        return([[]])
    else:
        P = []
        for p in part(r-s,s):
            q = p
            for j in range(len(p)):
                q[j] += 1
            P.append(q)
        P.sort()
        return(P)
\end{lstlisting}
\end{function}

\begin{function}
    The function \verb|Indices(z,d)| returns the list of all indices~$\boldsymbol{\mu}\in\Z_{>0}^\dd$ with~$|\boldsymbol{\mu}| = z+\dd$.
\begin{lstlisting}[language=python]
def Indices(z,d):
    if z==0:
        return([d*[1]])
    else:
        I = []
        for index in Indices(z-1,d):
            for k in range(d):
                indi = index[:k] + [index[k]+1] + index[k+1:]
                if indi not in I:
                    I.append(indi)
        I.sort()
        return(I)
\end{lstlisting}
\end{function}

\subsection{The box product}
In this section, we implement the box product as linear combination of words~$\uu_{\boldsymbol{\mu}}\in\left(\mathcal{U}\backslash\{u_0\}\right)^\ast$. Furthermore, for a set of box products, we implement the adjacency matrix whichs rows will correspond to the linear combinations and the columns to the words $\uu_{\boldsymbol{\mu}}$, i.e., the entries are the coefficient of a word in a linear combination of box products.

We begin with the box product.
\begin{function}
    The function \verb|box(index1,index2)| returns~$\uu_{\verb|index1|}\boxast \uu_{\verb|index2|}$ as follows. It returns a dictionary \verb|D| containing as keys the indices \verb|ind| satisfying that $\uu_{\verb|ind|}$ occurs in the box product $\uu_{\verb|index1|}\boxast \uu_{\verb|index2|}$ with multiplicity $\neq 0$; the value \verb|D[ind]| then is the multiplicity of $\uu_{\verb|ind|}$ in $\uu_{\verb|index1|}\boxast \uu_{\verb|index2|}$.
\begin{lstlisting}[language=Python]
def box(index1,index2):
    D = {}
    s = len(index1)
    d = len(index2)
    if s>d:
        return(D)
    elif index1 == []:
        D[str(index2)] = 1
    else:
        for S in set(itertools.combinations(range(d), s)):
            L = list(S)
            L.sort()
            ind = []
            for k in range(d):
                if k in L:
                    ind.append(index2[k]+index1[L.index(k)])
                else:
                    ind.append(index2[k])
            D[str(ind)] = 1
        return(D)
\end{lstlisting}
\end{function}

Based on \verb|box|, we introduce the following function representing~$\uu_{\verb|index1|}\boxast \uu_{\verb|index2|}$ as dictionary \verb|D| with keys $\verb|ind|\in\Z_{>0}^{\len(\verb|index2|)}$, satisfying
\begin{equation*}
    |\verb|ind|| = |{\verb|index1|}|+|{\verb|index2|}|,
\end{equation*}
and with \verb|D[ind]| being the multiplicity of $\uu_{\verb|ind|}$ in the box product $\uu_{\verb|index1|}\boxast \uu_{\verb|index2|}$.

\begin{lstlisting}[language=Python]
def BOX(index1,index2):
    s = len(index1)
    d = len(index2)
    z = sum(index1)+sum(index2)-d
    I = Indices(z,d)
    D = {}
    for ind in I:
        D[str(ind)] = 0
    if s>d or sum(index1)+sum(index2) != z+d:
        return(D)
    elif index1 == [] and sum(index2) == z+d:
        D[str(index2)] = 1
    else:
        for ind in box(index1,index2):
            D[ind] = box(index1,index2)[ind]
    return(D)
\end{lstlisting}

Let us consider an example to see the difference between the functions~\verb|box| and~\verb|BOX|.

\begin{example}
    We have
    \begin{align*}
        \uu_2\boxast \uu_1\uu_1\uu_1 = \uu_3\uu_1\uu_1 + \uu_1\uu_3\uu_1 + \uu_1\uu_1\uu_3.
    \end{align*}
    Now, \verb|box([2],[1,1,1])| returns
\begin{lstlisting}[language=python]
{'[3, 1, 1]': 1, '[1, 3, 1]': 1, '[1, 1, 3]': 1}
\end{lstlisting}
and \verb|BOX([2],[1,1,1])| returns
\begin{lstlisting}[language=python]
{'[1, 1, 3]': 1,
 '[1, 2, 2]': 0,
 '[1, 3, 1]': 1,
 '[2, 1, 2]': 0,
 '[2, 2, 1]': 0,
 '[3, 1, 1]': 1}.
\end{lstlisting}
\end{example}

\subsection{Dimension of spaces spanned by box products}

We considered in the paper the dimension of spaces spanned by several box products (in particular, $\VSbox{z}{\dd}$). Numerically, we will obtain such dimensions as the rank of the coefficient matrix of the box products that span the space we consider, interpreted as linear combination of words~$\uu_{\boldsymbol{\mu}}\in\left(\mathcal{U}\backslash\{u_0\}\right)^\ast$. For this, we introduce the function~\verb|MATR|.
\begin{function}
    The function~\verb|Dim(P)| takes a list \verb|P| of box products, given in shape of~\verb|BOX(index1,index2)|, and returns the dimension of the space they span. This is done via computing the rank of the coefficient matrix (as list of lists) of these box products with rows corresponding to the box products, columns corresponding to the coefficient of words~$\uu_{\boldsymbol{\mu}}\in\left(\mathcal{U}\backslash\{u_0\}\right)^\ast$.
\begin{lstlisting}[language=python]
def Dim(P):
    M = []
    for prod in P:
        I = []
        for index in prod:
            I.append(prod[index])
        M.append(I)
    rk = np.linalg.matrix_rank(M)
    return(rk)
\end{lstlisting}
\end{function}

\subsection{\LaTeX-Output}
We will consider subspaces of~$\VSbox{z}{\dd}$ for several $z,\dd\in\Z_{>0}$. Usually, we skip the cases of~$z=1$ or~$\dd=1$ since we already know the dimension of the corresponding subspace in these cases. The function~\verb|MatLatex| produces the \LaTeX-code of a table in which we collect our calculations.

\begin{function}
\label{func:table}
The function~\verb|MatLatex(M,cap)| gives the \LaTeX-code of the table with caption \verb|cap| and three entries in each cell. Here, \verb|M| is a list of lists with four entries each. They are all of shape
\begin{align*}
    [z,\dd,\operatorname{rk},\operatorname{dim}],
\end{align*} 
where~$z$ defines the column,~$\dd$ defines the row,~$\operatorname{rk}$ is the (numerical) dimension of the subspace of~$\VSbox{z}{\dd}$ we consider, while~$\operatorname{dim}$ is the corresponding conjectured dimension each. Every cell consists of two numbers, where the first one in black is the (numerically obtained) dimension of the subspace of~$\VSbox{z}{\dd}$ we consider and the second number is in \textcolor{blue}{blue} the conjectured dimension of the subspace of~$\VSbox{z}{\dd}$ we consider.
\begin{lstlisting}[language=python]
def MatLatex(M,cap):
    dmin = M[0][0]
    dmax = M[-1][0]
    zmin = M[0][1]
    zmax = M[-1][1]
    B = "\\begin{figure}[h]\n \\centering\n \\caption{"+cap+"}\n \\begin{tabular}{|" + "c|".join("" for j in range(zmin,zmax+2)) + "c|}\n \\hline\n"
    E = "\\end{tabular}\n \\end{figure}"
    newM = (dmax - dmin + 1)*[(zmax - zmin + 1)*["&-"]]
    S = "d$\\backslash$ z&" + "&".join(str(j) for j in range(zmin,zmax+1)) + "\\\\ \\hline\n"
    for result in M:
        helpstr = "&" + str(result[2]) + "\\ \\textcolor{blue}{"+str(result[3])+"} 
        dact = result[0] - dmin
        zact = result[1] - zmin
        rowact = newM[dact]
        newM = newM[:(result[0] - dmin)] + [rowact[:zact] + [helpstr] + rowact[zact+1:]] + newM[(result[0] - dmin+1):]
    for j in range(dmax - dmin + 1):
        S = S + str(dmin + j)
        for k in range(zmax-zmin+1):
            S = S + newM[j][k]
        S = S + "\\\\ \\hline\n"
    return(B+S+E)
\end{lstlisting}
\end{function}

Next, we produce the function giving the desired table for the dimension of~$\VSbox{z}{d,s_{\min}}$ for some $s_{\min}$ and $2\leq z,\dd$ up to an upper bound we declare in the input.

\begin{function}
Choosing $zmax,dmax,smin\in\Z_{>0}$, the following function returns the tabular according to Function~\ref{func:table} where in black the computed dimension of the space~$\VSbox{zmax}{dmax,smin}$ is displayed, while in \textcolor{blue}{blue} the conjectured dimension (coming from Conjecture~\ref{conj:conjrefinement}) appears.
\begin{lstlisting}[language=python]
def Tabular(zmax,dmax,smin):
    M = []
    for z in range(2,zmax+1):
        for d in range(2,dmax+1):
            P = []
            for k in range(smin,min(d,z)+1):
                S = ppart(d+z,d+k)
                for partition in S:
                    P.append(BOX(partition[:k],partition[k:]))
            rk = Dim(P)
            M.append([d,z,rk,d(z,d,smin)])
    if smin != 1:
        cap = "Dimension of $\\mathcal{S}_{z,d,"+str(smin)+"}$."
    else:
        cap = "Dimension of $\\mathcal{S}_{z,d}$."
    return(MatLatex(M,cap))
\end{lstlisting}
\end{function}

\subsection{Results}
In the following, we present several results of our calculations. Recall that every cell of the following tables consists of two numbers, where the first one in black is the (numerically obtained) dimension of the subspace of~$\VSbox{z}{\dd}$ we consider and the second number is in \textcolor{blue}{blue} the conjectured dimension from Conjecture~\ref{conj:conjrefinement}.

\begin{remark}

\begin{enumerate}
    \item Using \verb|Tabular(8,8,1)|, we obtain that Conjecture~\ref{conj:conjrefinement} is true for~$2\leq z\leq\dd\leq 8$ and $s_{\min} = 1$, i.e., Conjecture~\ref{conj:systemApp} is true for~$z,d\leq 8$:

\begin{figure}[H]
 \centering
 \caption{Dimension of $\mathcal{S}_{z,d}$.}
 \label{fig:smin=1}
 \begin{tabular}{|c|c|c|c|c|c|c|c|}
 \hline
d$\backslash$ z&2&3&4&5&6&7&8\\ \hline
2&3\ \textcolor{blue}{3} &-&-&-&-&-&-\\ \hline
3&4\ \textcolor{blue}{4} &10\ \textcolor{blue}{10} &-&-&-&-&-\\ \hline
4&5\ \textcolor{blue}{5} &15\ \textcolor{blue}{15} &35\ \textcolor{blue}{35} &-&-&-&-\\ \hline
5&6\ \textcolor{blue}{6} &21\ \textcolor{blue}{21} &56\ \textcolor{blue}{56} &126\ \textcolor{blue}{126} &-&-&-\\ \hline
6&7\ \textcolor{blue}{7} &28\ \textcolor{blue}{28} &84\ \textcolor{blue}{84} &210\ \textcolor{blue}{210} &462\ \textcolor{blue}{462} &-&-\\ \hline
7&8\ \textcolor{blue}{8} &36\ \textcolor{blue}{36} &120\ \textcolor{blue}{120} &330\ \textcolor{blue}{330} &792\ \textcolor{blue}{792} &1716\ \textcolor{blue}{1716} &-\\ \hline
8&9\ \textcolor{blue}{9} &45\ \textcolor{blue}{45} &165\ \textcolor{blue}{165} &495\ \textcolor{blue}{495} &1287\ \textcolor{blue}{1287} &3003\ \textcolor{blue}{3003} &6435\ \textcolor{blue}{6435} \\ \hline
\end{tabular}
 \end{figure}

\item Using \verb|Tabular(8,8,2)|, we obtain that Conjecture~\ref{conj:conjrefinement} is true for~$2\leq z\leq \dd\leq 8$ and~$s_{\min} = 2$:

\begin{figure}[H]
 \centering
 \caption{Dimension of $\mathcal{S}_{z,d,2}$.}
 \begin{tabular}{|c|c|c|c|c|c|c|c|}
 \hline
d$\backslash$ z&2&3&4&5&6&7&8\\ \hline
2&1\ \textcolor{blue}{1} &-&-&-&-&-&-\\ \hline
3&1\ \textcolor{blue}{1} &5\ \textcolor{blue}{5} &-&-&-&-&-\\ \hline
4&1\ \textcolor{blue}{1} &6\ \textcolor{blue}{6} &21\ \textcolor{blue}{21} &-&-&-&-\\ \hline
5&1\ \textcolor{blue}{1} &7\ \textcolor{blue}{7} &28\ \textcolor{blue}{28} &84\ \textcolor{blue}{84} &-&-&-\\ \hline
6&1\ \textcolor{blue}{1} &8\ \textcolor{blue}{8} &36\ \textcolor{blue}{36} &120\ \textcolor{blue}{120} &330\ \textcolor{blue}{330} &-&-\\ \hline
7&1\ \textcolor{blue}{1} &9\ \textcolor{blue}{9} &45\ \textcolor{blue}{45} &165\ \textcolor{blue}{165} &495\ \textcolor{blue}{495} &1287\ \textcolor{blue}{1287} &-\\ \hline
8&1\ \textcolor{blue}{1} &10\ \textcolor{blue}{10} &55\ \textcolor{blue}{55} &220\ \textcolor{blue}{220} &715\ \textcolor{blue}{715} &2002\ \textcolor{blue}{2002} &5005\ \textcolor{blue}{5005} \\ \hline
\end{tabular}
 \end{figure}

\item Using \verb|Tabular(8,8,3)|, we obtain that Conjecture~\ref{conj:conjrefinement} is true for~$2\leq z\leq \dd\leq 8$ and~$s_{\min} = 3$:

\begin{figure}[H]
 \centering
 \caption{Dimension of $\mathcal{S}_{z,d,3}$.}
 \begin{tabular}{|c|c|c|c|c|c|c|c|}
 \hline
d$\backslash$ z&2&3&4&5&6&7&8\\ \hline
2&0\ \textcolor{blue}{0} &-&-&-&-&-&-\\ \hline
3&0\ \textcolor{blue}{0} &1\ \textcolor{blue}{1} &-&-&-&-&-\\ \hline
4&0\ \textcolor{blue}{0} &1\ \textcolor{blue}{1} &7\ \textcolor{blue}{7} &-&-&-&-\\ \hline
5&0\ \textcolor{blue}{0} &1\ \textcolor{blue}{1} &8\ \textcolor{blue}{8} &36\ \textcolor{blue}{36} &-&-&-\\ \hline
6&0\ \textcolor{blue}{0} &1\ \textcolor{blue}{1} &9\ \textcolor{blue}{9} &45\ \textcolor{blue}{45} &165\ \textcolor{blue}{165} &-&-\\ \hline
7&0\ \textcolor{blue}{0} &1\ \textcolor{blue}{1} &10\ \textcolor{blue}{10} &55\ \textcolor{blue}{55} &220\ \textcolor{blue}{220} &715\ \textcolor{blue}{715} &-\\ \hline
8&0\ \textcolor{blue}{0} &1\ \textcolor{blue}{1} &11\ \textcolor{blue}{11} &66\ \textcolor{blue}{66} &286\ \textcolor{blue}{286} &1001\ \textcolor{blue}{1001} &3003\ \textcolor{blue}{3003} \\ \hline
\end{tabular}
 \end{figure}

\item Using \verb|Tabular(8,8,4)|, we obtain that Conjecture~\ref{conj:conjrefinement} is true for~$2\leq z\leq \dd\leq 8$ and~$s_{\min} = 4$:

\begin{figure}[H]
 \centering
 \caption{Dimension of $\mathcal{S}_{z,d,4}$.}
 \begin{tabular}{|c|c|c|c|c|c|c|c|}
 \hline
d$\backslash$ z&2&3&4&5&6&7&8\\ \hline
2&0\ \textcolor{blue}{0} &-&-&-&-&-&-\\ \hline
3&0\ \textcolor{blue}{0} &0\ \textcolor{blue}{0} &-&-&-&-&-\\ \hline
4&0\ \textcolor{blue}{0} &0\ \textcolor{blue}{0} &1\ \textcolor{blue}{1} &-&-&-&-\\ \hline
5&0\ \textcolor{blue}{0} &0\ \textcolor{blue}{0} &1\ \textcolor{blue}{1} &9\ \textcolor{blue}{9} &-&-&-\\ \hline
6&0\ \textcolor{blue}{0} &0\ \textcolor{blue}{0} &1\ \textcolor{blue}{1} &10\ \textcolor{blue}{10} &55\ \textcolor{blue}{55} &-&-\\ \hline
7&0\ \textcolor{blue}{0} &0\ \textcolor{blue}{0} &1\ \textcolor{blue}{1} &11\ \textcolor{blue}{11} &66\ \textcolor{blue}{66} &286\ \textcolor{blue}{286} &-\\ \hline
8&0\ \textcolor{blue}{0} &0\ \textcolor{blue}{0} &1\ \textcolor{blue}{1} &12\ \textcolor{blue}{12} &78\ \textcolor{blue}{78} &364\ \textcolor{blue}{364} &1365\ \textcolor{blue}{1365} \\ \hline
\end{tabular}
 \end{figure}

\item Using \verb|Tabular(8,8,5)|, we obtain that Conjecture~\ref{conj:conjrefinement} is true for~$2\leq z\leq \dd\leq 8$ and~$s_{\min} = 5$:

\begin{figure}[H]
 \centering
 \caption{Dimension of $\mathcal{S}_{z,d,5}$.}
 \begin{tabular}{|c|c|c|c|c|c|c|c|}
 \hline
d$\backslash$ z&2&3&4&5&6&7&8\\ \hline
2&0\ \textcolor{blue}{0} &-&-&-&-&-&-\\ \hline
3&0\ \textcolor{blue}{0} &0\ \textcolor{blue}{0} &-&-&-&-&-\\ \hline
4&0\ \textcolor{blue}{0} &0\ \textcolor{blue}{0} &0\ \textcolor{blue}{0} &-&-&-&-\\ \hline
5&0\ \textcolor{blue}{0} &0\ \textcolor{blue}{0} &0\ \textcolor{blue}{0} &1\ \textcolor{blue}{1} &-&-&-\\ \hline
6&0\ \textcolor{blue}{0} &0\ \textcolor{blue}{0} &0\ \textcolor{blue}{0} &1\ \textcolor{blue}{1} &11\ \textcolor{blue}{11} &-&-\\ \hline
7&0\ \textcolor{blue}{0} &0\ \textcolor{blue}{0} &0\ \textcolor{blue}{0} &1\ \textcolor{blue}{1} &12\ \textcolor{blue}{12} &78\ \textcolor{blue}{78} &-\\ \hline
8&0\ \textcolor{blue}{0} &0\ \textcolor{blue}{0} &0\ \textcolor{blue}{0} &1\ \textcolor{blue}{1} &13\ \textcolor{blue}{13} &91\ \textcolor{blue}{91} &455\ \textcolor{blue}{455} \\ \hline
\end{tabular}
 \end{figure}

\item Using \verb|Tabular(8,8,6)|, we obtain that Conjecture~\ref{conj:conjrefinement} is true for~$2\leq z\leq \dd\leq 8$ and~$s_{\min} = 6$:

\begin{figure}[H]
 \centering
 \caption{Dimension of $\mathcal{S}_{z,d,6}$.}
 \begin{tabular}{|c|c|c|c|c|c|c|c|}
 \hline
d$\backslash$ z&2&3&4&5&6&7&8\\ \hline
2&0\ \textcolor{blue}{0} &-&-&-&-&-&-\\ \hline
3&0\ \textcolor{blue}{0} &0\ \textcolor{blue}{0} &-&-&-&-&-\\ \hline
4&0\ \textcolor{blue}{0} &0\ \textcolor{blue}{0} &0\ \textcolor{blue}{0} &-&-&-&-\\ \hline
5&0\ \textcolor{blue}{0} &0\ \textcolor{blue}{0} &0\ \textcolor{blue}{0} &0\ \textcolor{blue}{0} &-&-&-\\ \hline
6&0\ \textcolor{blue}{0} &0\ \textcolor{blue}{0} &0\ \textcolor{blue}{0} &0\ \textcolor{blue}{0} &1\ \textcolor{blue}{1} &-&-\\ \hline
7&0\ \textcolor{blue}{0} &0\ \textcolor{blue}{0} &0\ \textcolor{blue}{0} &0\ \textcolor{blue}{0} &1\ \textcolor{blue}{1} &13\ \textcolor{blue}{13} &-\\ \hline
8&0\ \textcolor{blue}{0} &0\ \textcolor{blue}{0} &0\ \textcolor{blue}{0} &0\ \textcolor{blue}{0} &1\ \textcolor{blue}{1} &14\ \textcolor{blue}{14} &105\ \textcolor{blue}{105} \\ \hline
\end{tabular}
 \end{figure}

\item Using \verb|Tabular(8,8,7)|, we obtain that Conjecture~\ref{conj:conjrefinement} is true for~$2\leq z\leq\dd\leq 8$ and~$s_{\min} = 7$:

\begin{figure}[H]
 \centering
 \caption{Dimension of $\mathcal{S}_{z,d,7}$.}
 \begin{tabular}{|c|c|c|c|c|c|c|c|}
 \hline
d$\backslash$ z&2&3&4&5&6&7&8\\ \hline
2&0\ \textcolor{blue}{0} &-&-&-&-&-&-\\ \hline
3&0\ \textcolor{blue}{0} &0\ \textcolor{blue}{0} &-&-&-&-&-\\ \hline
4&0\ \textcolor{blue}{0} &0\ \textcolor{blue}{0} &0\ \textcolor{blue}{0} &-&-&-&-\\ \hline
5&0\ \textcolor{blue}{0} &0\ \textcolor{blue}{0} &0\ \textcolor{blue}{0} &0\ \textcolor{blue}{0} &-&-&-\\ \hline
6&0\ \textcolor{blue}{0} &0\ \textcolor{blue}{0} &0\ \textcolor{blue}{0} &0\ \textcolor{blue}{0} &0\ \textcolor{blue}{0} &-&-\\ \hline
7&0\ \textcolor{blue}{0} &0\ \textcolor{blue}{0} &0\ \textcolor{blue}{0} &0\ \textcolor{blue}{0} &0\ \textcolor{blue}{0} &1\ \textcolor{blue}{1} &-\\ \hline
8&0\ \textcolor{blue}{0} &0\ \textcolor{blue}{0} &0\ \textcolor{blue}{0} &0\ \textcolor{blue}{0} &0\ \textcolor{blue}{0} &1\ \textcolor{blue}{1} &15\ \textcolor{blue}{15} \\ \hline
\end{tabular}
 \end{figure}

\item Using \verb|Tabular(8,8,8)|, we obtain that Conjecture~\ref{conj:conjrefinement} is true for~$2\leq z\leq\dd\leq 8$ and~$s_{\min} = 8$:

\begin{figure}[H]
 \centering
 \caption{Dimension of $\mathcal{S}_{z,d,8}$.}
 \label{fig:smin=8}
 \begin{tabular}{|c|c|c|c|c|c|c|c|}
 \hline
d$\backslash$ z&2&3&4&5&6&7&8\\ \hline
2&0\ \textcolor{blue}{0} &-&-&-&-&-&-\\ \hline
3&0\ \textcolor{blue}{0} &0\ \textcolor{blue}{0} &-&-&-&-&-\\ \hline
4&0\ \textcolor{blue}{0} &0\ \textcolor{blue}{0} &0\ \textcolor{blue}{0} &-&-&-&-\\ \hline
5&0\ \textcolor{blue}{0} &0\ \textcolor{blue}{0} &0\ \textcolor{blue}{0} &0\ \textcolor{blue}{0} &-&-&-\\ \hline
6&0\ \textcolor{blue}{0} &0\ \textcolor{blue}{0} &0\ \textcolor{blue}{0} &0\ \textcolor{blue}{0} &0\ \textcolor{blue}{0} &-&-\\ \hline
7&0\ \textcolor{blue}{0} &0\ \textcolor{blue}{0} &0\ \textcolor{blue}{0} &0\ \textcolor{blue}{0} &0\ \textcolor{blue}{0} &0\ \textcolor{blue}{0} &-\\ \hline
8&0\ \textcolor{blue}{0} &0\ \textcolor{blue}{0} &0\ \textcolor{blue}{0} &0\ \textcolor{blue}{0} &0\ \textcolor{blue}{0} &0\ \textcolor{blue}{0} &1\ \textcolor{blue}{1} \\ \hline
\end{tabular}
 \end{figure}
 \end{enumerate}
\end{remark}

\section{Computations regarding Lemma~\ref{lem:boxkernel2}}

\subsection{Setup and basic functions}
We use the same setup as in Section~\ref{ssec:code:setup} and the functions~\verb|part| and~\verb|ppart| from there.

\subsection{Stuffle product and box product}
We define the stuffle product on index level and call the function~\verb|stuffleprod|.

\begin{function}
For indices~\verb|L1| and~\verb|L2| (input as lists), the function~$\verb|stuffleprod|(\verb|L1|,\verb|L2|)$ returns a list of indices (as lists) with the property that their formal sum is exactly the stuffle product~$\verb|L1|\ast \verb|L2|$.
\begin{lstlisting}[language=python]
def stuffleprod(L1,L2):
    if len(L1) == 0:
        return([L2])
    elif len(L2) == 0:
        return([L1])
    L = []
    for L3 in stuffleprod(L1[1:],L2):
        L.append([L1[0]]+L3)
    for L3 in stuffleprod(L1,L2[1:]):
        L.append([L2[0]]+L3)
    for L3 in stuffleprod(L1[1:],L2[1:]):
        L.append([L1[0]+L2[0]]+L3)
    return(L)
\end{lstlisting}
\end{function}

Furthermore, we define the box product on index level and call the function~\verb|boxprod|.

\begin{function}
For two indices~\verb|L1| and~\verb|L2| (input as lists), the function~$\verb|boxprod|(\verb|L1|,\verb|L2|)$ returns a list of indices (as lists) with the property that their formal sum is exactly the box product~$\verb|L1|\ast \verb|L2|$.
\begin{lstlisting}[language=python]
def boxprod(L1,L2):
    s = len(L1)
    d = len(L2)
    if s>d:
        return([])
    elif s==0:
        return([L2])
    L = []
    for L3 in boxprod(L1[1:],L2[1:]):
        L.append([L1[0]+L2[0]]+L3)
    for L3 in boxprod(L1,L2[1:]):
        L.append([L2[0]]+L3)
    return(L)
\end{lstlisting}
\end{function}

\subsection{The numbers~$\dim_\Q\operatorname{span}_\Q\kernelbox{z}{\dd}$}
First, we implement for given~$1\leq z\leq\dd$ the conjectured dimension of $\operatorname{span}_\Q\kernelbox{z}{\dd}$. Following Conjecture~\ref{conj:systemApp},~\eqref{eq:sumdim}, and~\eqref{eq:idim}, this number is
\begin{align}
\eqlabel{eq:code:kerneldim}
    \sum\limits_{j=2}^z \binom{z+\dd-1}{\dd+j-1}.
\end{align}
\begin{function}
    For~$z,\dd\in\Z_{>0}$ with~$z\leq\dd$, the function~\verb|kerneldimconj| returns the conjectured dimension of~$\operatorname{span}_\Q\kernelbox{z}{\dd}$, which is given by~\eqref{eq:code:kerneldim}.
\begin{lstlisting}[language=python]
def kerneldimconj(z,d):
    S = 0
    for j in range(d+1,z+d):
        S = S + math.comb(z+d-1,j)
    return(S)
\end{lstlisting}
\end{function}

The next function returns for given~$1\leq z\leq\dd$ the number~$\dim_\Q\operatorname{span}_\Q\kernelbox{z}{\dd}$.

\begin{function}
    Let be~$z,\dd\in\Z_{>0}$ with~$z\leq\dd$. The function~$\verb|kerneldim(z,d)|$ returns the number~$\dim_\Q\operatorname{span}_\Q\kernelbox{z}{\dd}$ via computing ranks of matrices.
\begin{lstlisting}[language=python]
def kerneldim(z,d):
    Rel = []
    for s in range(d+2,z+d+1):
        for partition in ppart(z+d,s):
            for t in range(d+1,s):
                Mind = partition[t:]
                Lind = partition[:d]
                Nind = partition[d:t]
                D = {}
                for s in range(d+1,z+d+1):
                    for ppartition in ppart(z+d,s):
                        D[str(ppartition)] = 0
                for P in boxprod(Mind,Lind):
                    D[str(Nind+P)] = D[str(Nind+P)] + 1
                for P in stuffleprod(Nind,Mind):
                    D[str(P+Lind)] = D[str(P+Lind)] - 1
                R = []
                for key in D:
                    R.append(D[key])
                Rel.append(R)
    return(np.linalg.matrix_rank(Rel))
\end{lstlisting}
\end{function}

\subsection{Results}

Via
\begin{lstlisting}[language=python]
for d in range(2,9):
    for z in range(2,d+1):
        print(z,d,(kerneldim(z,d),kerneldimconj(z,d)))
\end{lstlisting}

we obtain in the following in each row four entries, the first one corresponding to~$z$, the second to~$\dd$, the third to the numerical result for~$\dim_\Q\operatorname{span}_\Q\kernelbox{z}{\dd}$, and the fourth is the value we expect for~$\dim_\Q\operatorname{span}_\Q\kernelbox{z}{\dd}$:
\begin{lstlisting}[language=python]
2 2 (1, 1)
2 3 (1, 1)
3 3 (6, 6)
2 4 (1, 1)
3 4 (7, 7)
4 4 (29, 29)
2 5 (1, 1)
3 5 (8, 8)
4 5 (37, 37)
5 5 (130, 130)
2 6 (1, 1)
3 6 (9, 9)
4 6 (46, 46)
5 6 (176, 176)
6 6 (562, 562)
2 7 (1, 1)
3 7 (10, 10)
4 7 (56, 56)
5 7 (232, 232)
6 7 (794, 794)
7 7 (2380, 2380)
2 8 (1, 1)
3 8 (11, 11)
4 8 (67, 67)
5 8 (299, 299)
6 8 (1093, 1093)
7 8 (3473, 3473)
8 8 (9949, 9949)
\end{lstlisting}

\begin{remark}
    Regarding our results, Lemma~\ref{lem:boxkernel2} is proven.
\end{remark}

\bibliographystyle{alpha}
\bibliography{example.bib}

\end{document}